\documentclass[12pt,reqno]{article}
\usepackage[english]{babel}
\usepackage{amsmath,amsthm,commath,mathrsfs,amssymb,extarrows}
\usepackage{dsfont}
\usepackage{relsize}

\usepackage{amsfonts}
\usepackage{graphicx}
\usepackage{enumerate}
\usepackage{subcaption}
\usepackage[right,pagewise,displaymath, mathlines]{lineno}
\usepackage{epstopdf}
\usepackage{color}
\usepackage{multirow}
\usepackage{authblk}
\usepackage[round]{natbib}
\usepackage{float}
\restylefloat{figure,table}

\usepackage{authblk}
\usepackage[round]{natbib}
\usepackage{float}

\usepackage{enumitem}

\usepackage{tikz}
\usepackage{schemabloc}

\usepackage[bottom]{footmisc}

\usepackage[colorlinks = true,urlcolor = blue,citecolor = blue,breaklinks]{hyperref}

\usepackage{comment}

\usepackage{geometry}
\numberwithin{equation}{section}
\numberwithin{figure}{section}
\numberwithin{table}{section}

\newtheorem{theorem}{Theorem}[section]
\newtheorem{corollary}{Corollary}[section]
\newtheorem{lemma}{Lemma}[section]

\theoremstyle{definition}

\newtheorem{example}{Example}[section]
\newtheorem{note}{Note}[section]
\newtheorem{remark}{Remark}[section]

\newtheorem{computation}{Computation}[section]
\newtheorem{simulation}{Simulation}[section]

\numberwithin{equation}{section}

\definecolor{darkred}{rgb}{0.7, 0, 0}
\definecolor{darkbrown}{rgb}{0.55, 0.2, 0.15}
\definecolor{darkblue}{rgb}{0.1,0.1,0.6}
\definecolor{darkgreen}{rgb}{0.1,0.5,0.2}

\newcommand{\dd}{\mathrm{d}}

\newcommand{\Var}{\mathrm{Var}}

\newcommand{\Rem}{\mathrm{Rem}}
\newcommand{\dist}{\mathrm{dist}}

\newcommand{\ts}{\textsc{ts}}
\newcommand{\srs}{\textsc{srs}}

\newcommand{\lc}{\textsc{lc}}
\newcommand{\gc}{\textsc{gc}}

\newcommand{\zero}{\textsc{lower}}
\newcommand{\one}{\textsc{upper}}
\newcommand{\half}{\textsc{middle}}

\newcommand{\ATVaR}{\mathrm{TVaR}}
\newcommand{\FTVaR}{\mathrm{TV@R}}
\newcommand{\LC}{\mathrm{LC}}
\newcommand{\GC}{\mathrm{GC}}

\newcommand{\Cov}{\mathrm{Cov}}

\newcommand{\ave}{\textsc{ave}}
\newcommand{\med}{\textsc{med}}

\usepackage{setspace}
\usepackage[bottom]{footmisc}

\usepackage{pgfplots}

\title{\vspace{-20mm} Fundamentals of non-parametric statistical inference for integrated quantiles\thefootnote\relax\footnotetext{We are indebted to Vytaras Brazauskas, Jan Dhaene, and Ruodu Wang for generous suggestions that shaped our work on the current version of the manuscript. This research has been supported by the NSERC Alliance--MITACS Accelerate grant (ALLRP 580632-22) entitled ``New Order of Risk Management: Theory and Applications in the Era of Systemic Risk'' from the Natural Sciences and Engineering Research Council (NSERC) of Canada and the national research organization Mathematics of Information Technology and Complex Systems (MITACS) of Canada, as well as by the NSERC Discovery Grant RGPIN-2022-04426.}}

\author[,1,2]{Nadezhda V. Gribkova \thanks{e-mail: \href{mailto:n.gribkova@spbu.ru}{n.gribkova@spbu.ru}}}
\author[,3]{Mengqi Wang \thanks{e-mail: \href{mailto:mwan259@uwo.ca}{mwan259@uwo.ca}}}
\author[,3]{Ri\v{c}ardas Zitikis \thanks{Corresponding author; e-mail: \href{mailto:rzitikis@uwo.ca}{rzitikis@uwo.ca}}}
\affil[1]{\normalsize  Saint Petersburg State University, Saint Petersburg, 199034 Russia}
\affil[2]{\normalsize  Emperor Alexander I Saint Petersburg State Transport University, \break Saint Petersburg, 190031 Russia}
\affil[3]{\normalsize Western University, London, Ontario N6A 5B7, Canada}

\date{\vspace{-1mm}\normalsize\today}

\begin{document}

\maketitle 

\vspace{-7mm}

{\small 
\noindent 
\textbf{Abstract.}
We present a general non-parametric statistical inference theory for integrals of quantiles without assuming any specific sampling design or dependence structure. Technical considerations are accompanied by examples and discussions, including those pertaining to the bias of empirical estimators. To illustrate how the general results can be adapted to specific situations, we derive -- at a stroke and under minimal conditions -- consistency and asymptotic normality of the empirical tail-value-at-risk, Lorenz and Gini curves at any probability level in the case of the simple random sampling, thus facilitating a comparison of our results with what is already known in the literature. Results, notes and references concerning dependent (i.e., time series) data are also offered. As a by-product, our general results provide new and unified proofs of large-sample properties of a number of classical statistical estimators, such as trimmed means, and give additional insights into the origins of, and the reasons for, various necessary and sufficient conditions.  

\medskip

}

\newpage 

{\hypersetup{linkcolor=blue} \small \baselineskip 12.5pt \tableofcontents}
\newpage 

\section{Introduction}
\label{intro}

Although well understood, developing statistical inference for quantiles, denoted by $F^{-1}(p)$ for various $p\in [0,1]$ and known as values-at-risk in banking and insurance, is a challenging task. This is due to the fact that empirical quantiles are order statistics \citep[e.g.,][]{ABN2008,DN2003}, unlike empirical cumulative distribution functions (cdf's) that are sums of (Bernoulli) random variables and can therefore be easily tackled using classical techniques of mathematical statistics and probability theory \citep[e.g.,][]{s2017}. Inevitably, therefore, in the case of quantiles, quite restrictive assumptions on the population cdf $F$ arise. For example, consistency of the empirical quantile requires continuity of the population quantile function at the specified probability level, while asymptotic normality requires conditions on the underlying probability density function (pdf) at the specified quantile \citep[e.g.][]{s1980,s2003}.

In many applications that arise in economics, finance, and insurance, quantiles are often integrated. Visualized in Figure~\ref{curves} 
\begin{figure}[h!]
    \centering
    \begin{subfigure}[b]{0.43\textwidth}
        \includegraphics[width=\textwidth]{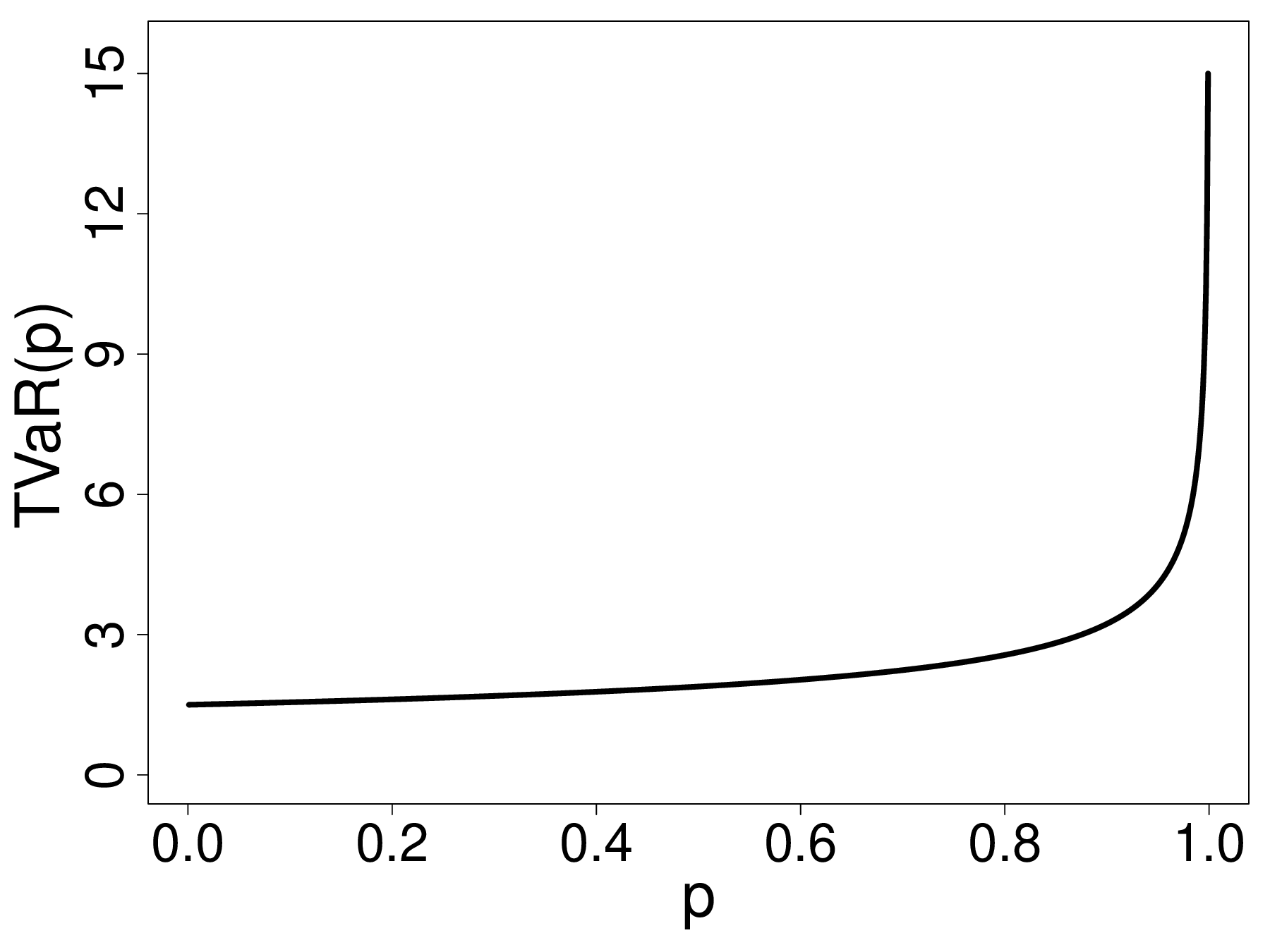}
        \caption{Upside tail-value-at-risk $\ATVaR(p)$.}
        \label{fig:ATVaR}
    \end{subfigure}
\qquad
    \begin{subfigure}[b]{0.43\textwidth}
        \includegraphics[width=\textwidth]{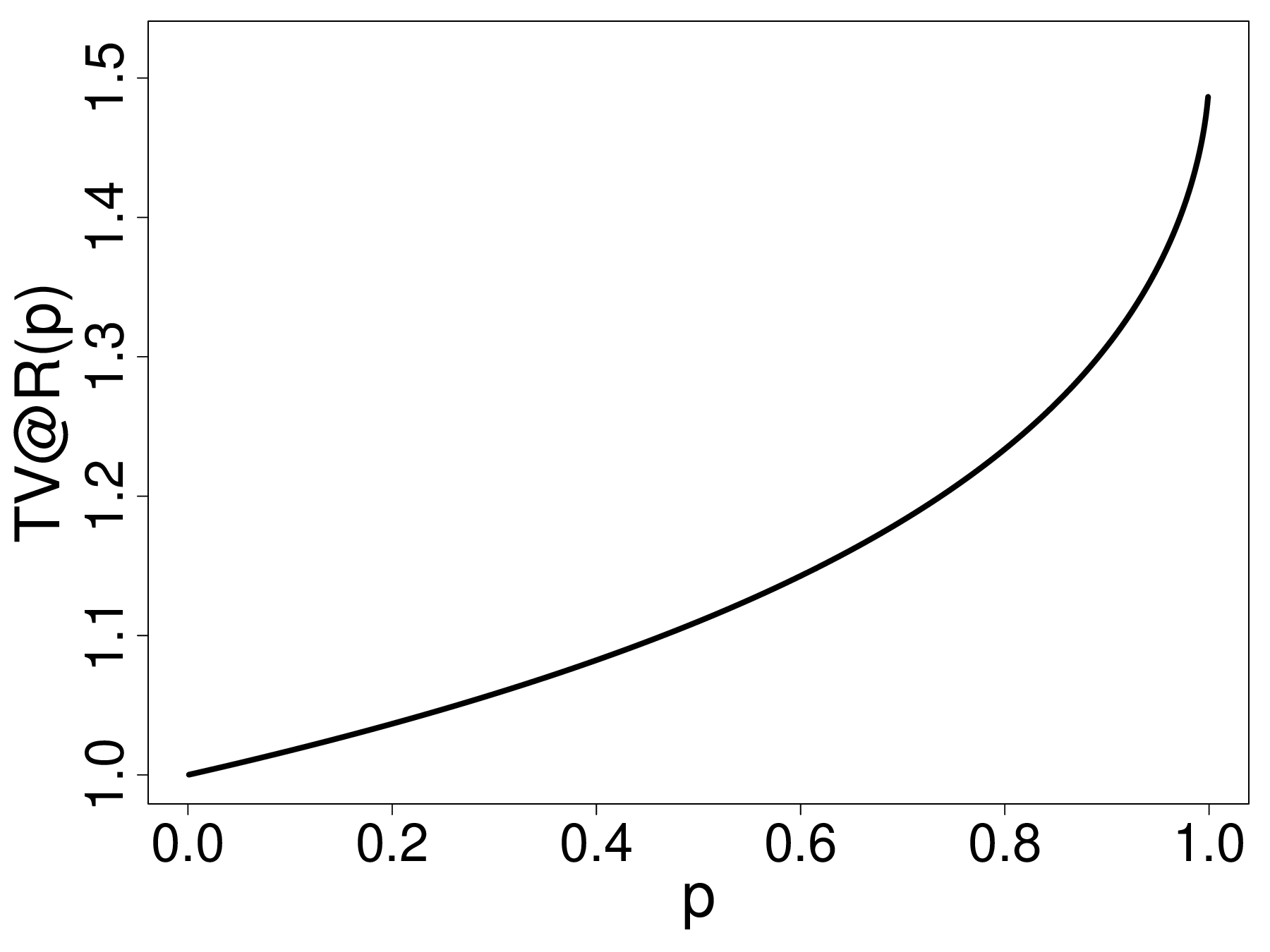}
        \caption{Downside tail-value-at-risk $\FTVaR(p)$.}
        \label{fig:FTVaR}
    \end{subfigure}
\\ \vspace*{5mm}
    \begin{subfigure}[b]{0.43\textwidth}
        \includegraphics[width=\textwidth]{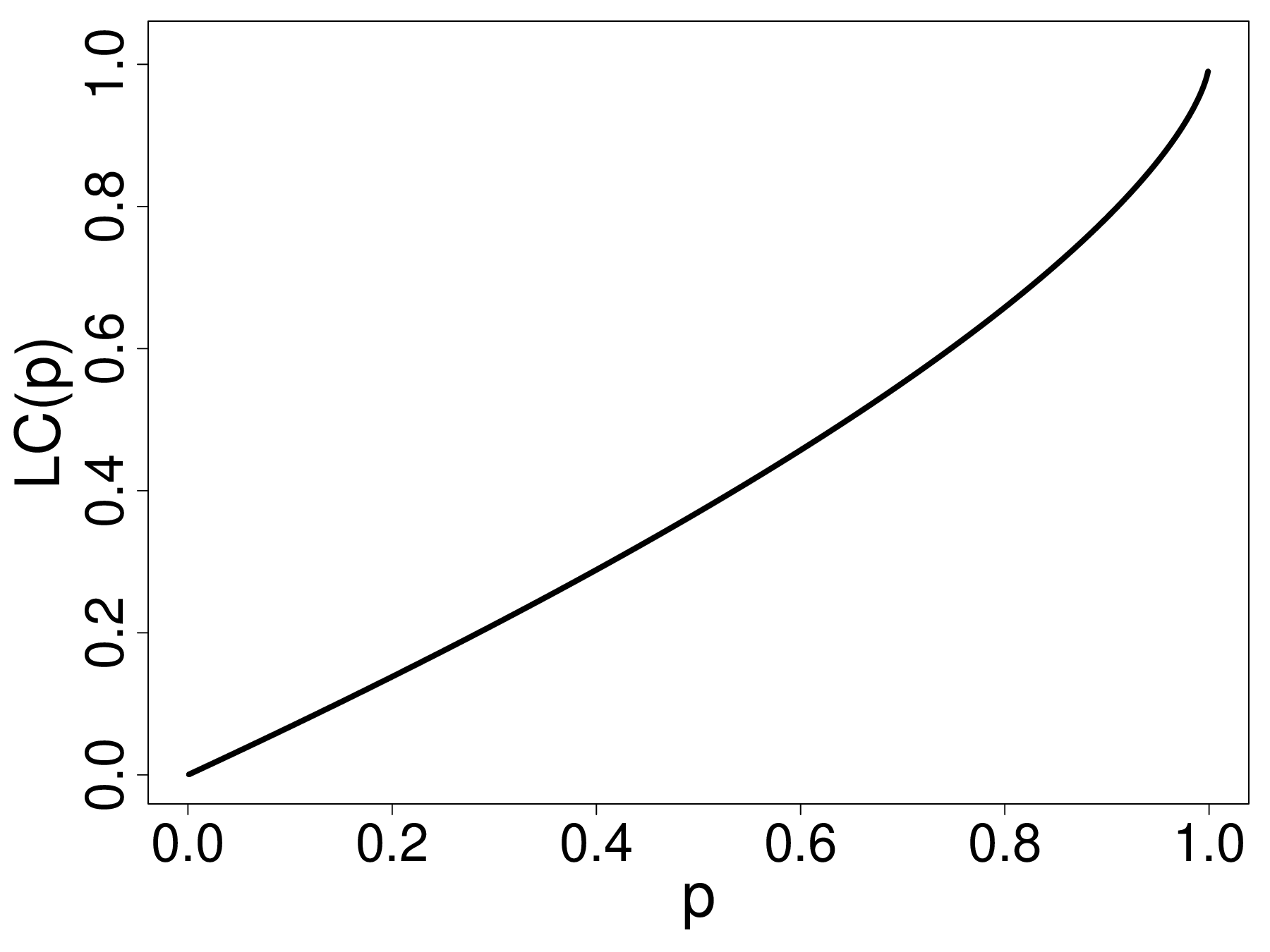}
        \caption{Lorenz curve $\LC(p)$.}
        \label{fig:LC}
    \end{subfigure}
\qquad
    \begin{subfigure}[b]{0.43\textwidth}
        \includegraphics[width=\textwidth]{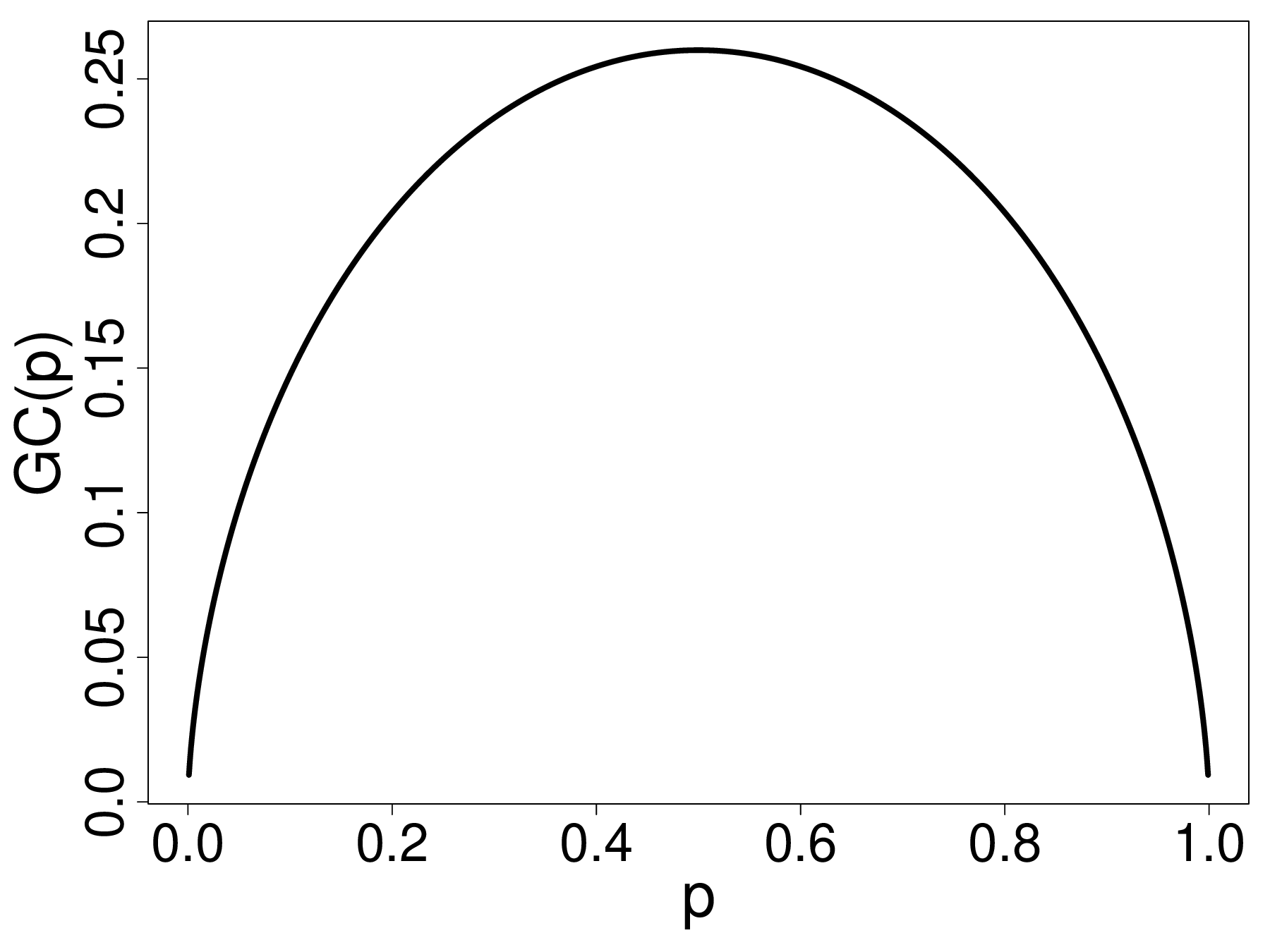}
        \caption{Gini curve $\GC(p)$.}
        \label{fig:GC}
    \end{subfigure}
    \caption{Four illustrative curves, whose definitions are based on integrated quantiles,  depicted here in the case of a Pareto distribution.}
    \label{curves}
\end{figure}
in the case of a Pareto distribution (specifics in Section~\ref{examples}) are four illustrative  examples of such integrals and their combinations.  Namely, with $p$ running through the unit interval $(0,1)$, they are
the upside tail-value-at-risk
\begin{equation}\label{ill-ATVaR-0}
\ATVaR(p)={1\over 1-p}\int_p^1 F^{-1}(u)\dd u , 
\end{equation} 
the downside tail-value-at-risk
\begin{equation}\label{ill-FTVaR-0}
\FTVaR(p)={1\over p}\int_0^p F^{-1}(u)\dd u ,  
\end{equation}
the Lorenz curve 
\begin{equation}\label{ill-Lorenz-0}
\LC(p)={1\over \mu }\int_0^p F^{-1}(u)\dd u , 
\end{equation}
and the Gini curve 
\begin{equation}\label{ill-Gini-0}
\GC(p) = {1\over \mu } 
\left( \int_{1-p}^1 F^{-1}(u)\dd u - \int_0^p F^{-1}(u)\dd u \right) ,   
\end{equation}
where 
\begin{equation}\label{ill-mean-0}
\mu=\int_0^1 F^{-1}(u)\dd u 
\end{equation}
is the mean of the population cdf $F$. Of course, the mean could serve a fifth illustrative example, but it is trivial in the context of the present paper and, therefore, is viewed here only as an auxiliary quantity. We shall give more details on the four measures of risk and economic inequality in Section~\ref{examples}. In Section~\ref{sect-L-functional}, we shall discuss more complex (i.e., distorted or weighted) integrals of quantiles, known in statistics as $L$-integrals, and show how they can be reduced to those of the type that we see in the above examples. 

Since, generally speaking, integrals are linear functionals of their integrands, which are quantiles in our case, researchers often use the aforementioned asymptotic results for quantiles and then apply the continuous mapping theorem \citep[e.g.,][]{b1999} to derive desired results for integrated quantiles \citep[e.g.,][]{cch1986,BFWW22,hww2024}. This approach, however, involves the aforementioned assumptions on the population pdf, although we shall soon see that such assumptions are unnecessary for integrated quantilies; even the very existence of pdf is unnecessary.

Indeed, when developing an asymptotic theory for the Lorenz curve, whose definition is based on integrated quantiles as pointed out by \citet{g1971}, \citet{cs1996} observed that neither the original formulation of the problem nor the obtained large-sample asymptotic distribution requires the existence of pdf, let alone assumptions on it. Based on the observation, they set out to find a path that would lead to  desired asymptotic results without involving pdf's. They succeeded in achieving this goal by introducing a technical tool that they called the Vervaat process, named after the Dutch mathematician Wim Vervaat, whose pioneering results \citep{v1972a,v1972b} on a combination of the uniform on $[0,1]$ quantile and empirical processes  served an inspiration. We refer to \citet{z1998} for details and references on the topic. 

\begin{note}\label{note-0} 
The origin of the research path taken by \citet{cs1996} was humble: it was the equation 
\begin{equation}\label{eq-0}
\int_0^1 \big( G^{-1}(u)-F^{-1}(u) \big) \dd u 
=\int^{\infty }_{-\infty}  \big( F(x)- G(x)\big) \dd x 
\end{equation}
that is known to hold (see Appendix~\ref{proof-eq-0} for details) for every pair of random variables $X$ and $Y$ with finite first moments, where $F$ is the cdf of $X$ (shorthanded as $X\sim F$) and $G$ is the cdf of $Y$ (shorthanded as $Y\sim G$), and $F^{-1}$ and $G^{-1}$ are the corresponding quantile functions. For example, the quantile function $F^{-1}$ of $X$ is given by the equation 
\begin{equation}\label{qq-0}
F^{-1}(u)=\inf\{ x \in \mathbb{R} ~:~ F(x)\ge u\}
\end{equation}
for all $u\in (0,1]$, while at the point $u=0$ it is defined as the right-hand limit 
\begin{equation}\label{qq-00}
F^{-1}(0)=\lim_{u\downarrow 0} F^{-1}(u). 
\end{equation} 
(Mathematicians would call $F^{-1}$ the left-continuous generalized inverse of $F$.)  
Since the quantile function is not the ordinary inverse of $F$, as such may not exist unless the cdf is continuous and strictly increasing, technical difficulties and even overlooks do arise \citep[e.g.][and references therein]{w2023}. Note also that $F^{-1}(0)$ and $F^{-1}(1)$ are the two endpoints (finite or infinite) of the support of the cdf $F$, which is the smallest closed-in-$\mathbb{R}$ interval that contains all the points $x \in \mathbb{R}$ such that $ F(x)\in (0,1)$. 
\end{note}

A long series of research articles by various authors followed \citet{cs1996}, exploring theoretical and empirical aspects of integrated quantiles \textit{without} involving pdf's. Recently, \citet{wz2023} suggested an extension of the theory initiated by \citet{cs1996}. The current paper takes these developments even further by establishing asymptotic results for integrated quantiles without assuming any specific sampling design or dependence structure between the underlying random variables. In fact, the main results are formulated for generic cdf's $F$ and approximating sequences $F_n$, $n\in \mathbb{N}$, of cdf's, which may or may not be random, depending on the problem.

\section{An overview}

We have organized the rest of the paper as follows. In Section~\ref{sect-upper}, we develop a general statistical inference theory for the integral 
\begin{equation}\label{int-upper}
\int_p^1 F^{-1}(u)\dd u  
\end{equation}
of quantiles over the upper-most layer $(p,1)$ of probabilities, where $p\in (0,1)$ is a fixed boundary probability. We call integral~\eqref{int-upper} the upper-layer integral. It plays a particularly prominent role in insurance, where loss random variables are often non-negative, and thus the right-hand tails of their distributions become of particular concern. Risk measures such as the expected shortfall (ES) and its sister risk measure called the ``upside'' tail-value-at-risk (TVaR) arise \citep[e.g.,][]{ddgk2005}, which in turn lead to considerations of more general integrals called distorted expectations \citep[e.g.,][and references therein]{dklt2012}. We shall rigorously define and discuss TVaR later in this paper (see Example~\ref{illustration-1} in particular).

In Section~\ref{sect-upper-sim} we shall offer several simulated experiments that will clarify some of the statistical properties of the empirical upper-layer integral.

Before we delve into the topics of the following sections, we present several notes that are useful for understanding and appreciating the results of this paper, and in particular the assumptions under which the results will be established.

\begin{note}\label{note-1a}
When $p=0$, integral~\eqref{int-upper} reduces to the mean $\mu $  
of the cdf $F$, given by equation~\eqref{ill-mean-0}. Statistical inference for $\mu $ under various sampling designs can, of course, be developed without invoking the existence of the pdf of $F$. When $p=1$, integral~\eqref{int-upper} vanishes. Hence, by restricting ourselves to $p\in (0,1)$ we are just excluding the two statistically trivial cases $p=0$ and $p=1$. 
\end{note}

\begin{note}\label{note-1aa}
There is a deeper reason for separating the boundary case  $p=0$ from $p\in (0,1)$. First, when $p=0$, in the case of the simple random sampling (SRS), that is, when we work with $n$ independent and identically distributed (iid) random variables, the empirical estimator of $\mu $ is an asymptotically normal estimator under the only requirement that the second moment of the underlying random variable is finite. No assumption related to continuity of the cdf $F$ at any point of the real line is needed. In the case $p\in (0,1)$, however, 
for an empirical estimator of the upper-layer integral $\int_p^1 F^{-1}(u)\dd u $ to be asymptotically normal, it is necessary to assume that the quantile function $F^{-1}$ is continuous at the point $p$. The necessity of this condition is demonstrated by \citet{s1973} when the empirical estimator is a trimmed mean based on SRS. This assumption is also needed for our general results as they absorb the SRS-based trimmed mean as a special case. No such condition, however, is required for \textit{consistency} of the empirical estimator  of the upper-layer integral $\int_p^1 F^{-1}(u)\dd u $. We shall illustrate this phenomenon with a simulated example in Section~\ref{sect-upper-sim}. 
\end{note}

\begin{note}\label{note-quant} 
The aforementioned continuity of the quantile function $F^{-1}$ at the point $p$ permeates the entire paper. In our minds, however, we rarely visualize distributions in terms of quantiles -- we often think of them in terms of cdf's. For this reason it is beneficial to follow \citet{s2003} and reformulate the continuity of $F^{-1}$ at the point $p \in (0,1)$ in the form of the bounds 
\begin{equation}\label{cont-quantile}
F\left(F^{-1}(p)-\varepsilon\right) < p < F\left(F^{-1}(p)+\varepsilon\right)
\end{equation}
that need to be satisfied for every $\varepsilon >0$. 
Figure~\ref{figure-gap-0} 
\begin{figure}[h!]
    \centering
    \begin{subfigure}[b]{0.43\textwidth}
        \includegraphics[width=\textwidth]{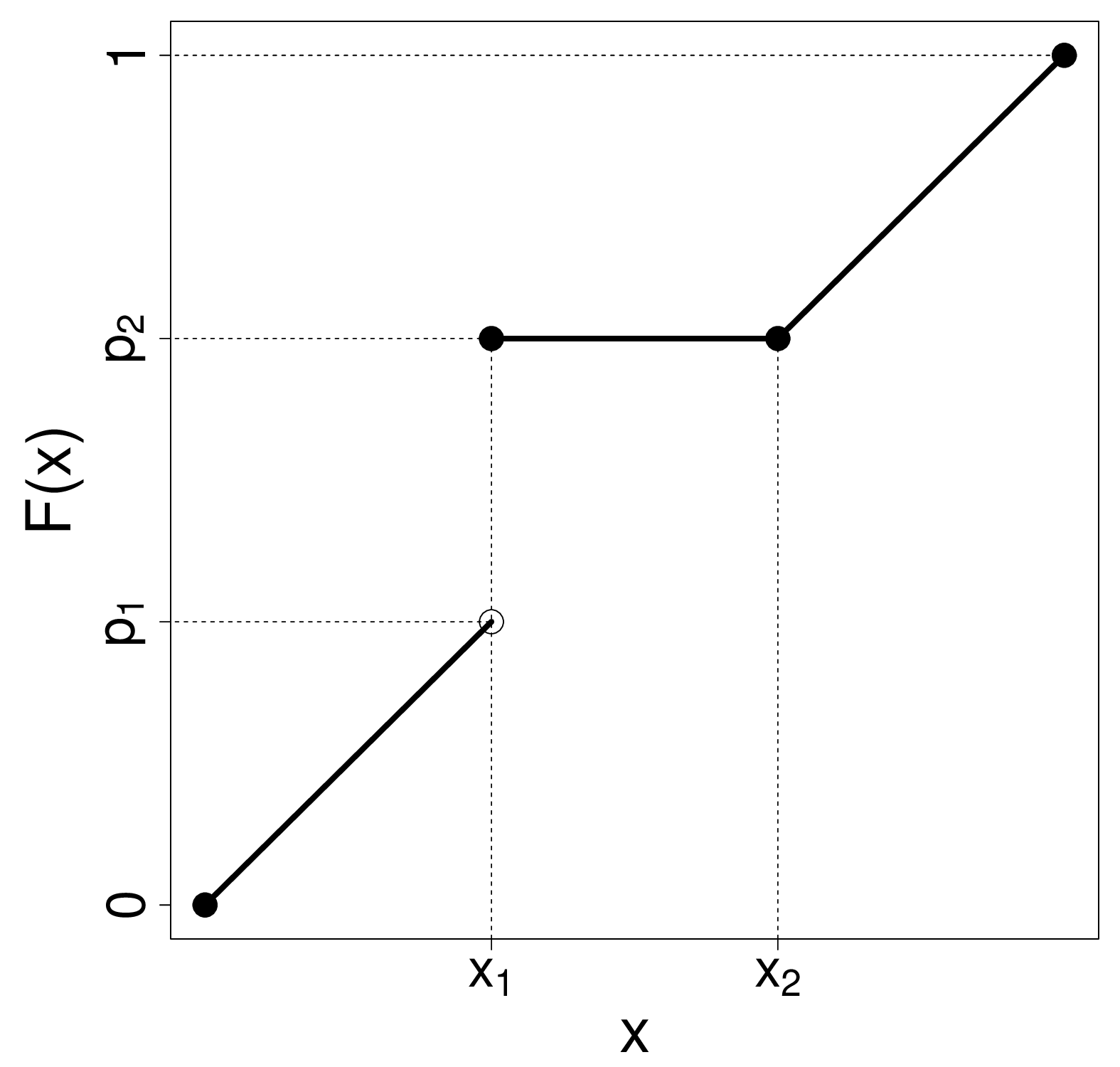}
        \caption{Cdf: The only $p\in (0,1)$ that fails condition~\eqref{cont-quantile} is $p=p_2$, provided that $x_1<x_2$. When $x_1=x_2$ (no gap), every $p\in (0,1)$ satisfies the condition.}
        \label{fig:CDF-new}
    \end{subfigure}
\qquad
    \begin{subfigure}[b]{0.43\textwidth}
        \includegraphics[width=\textwidth]{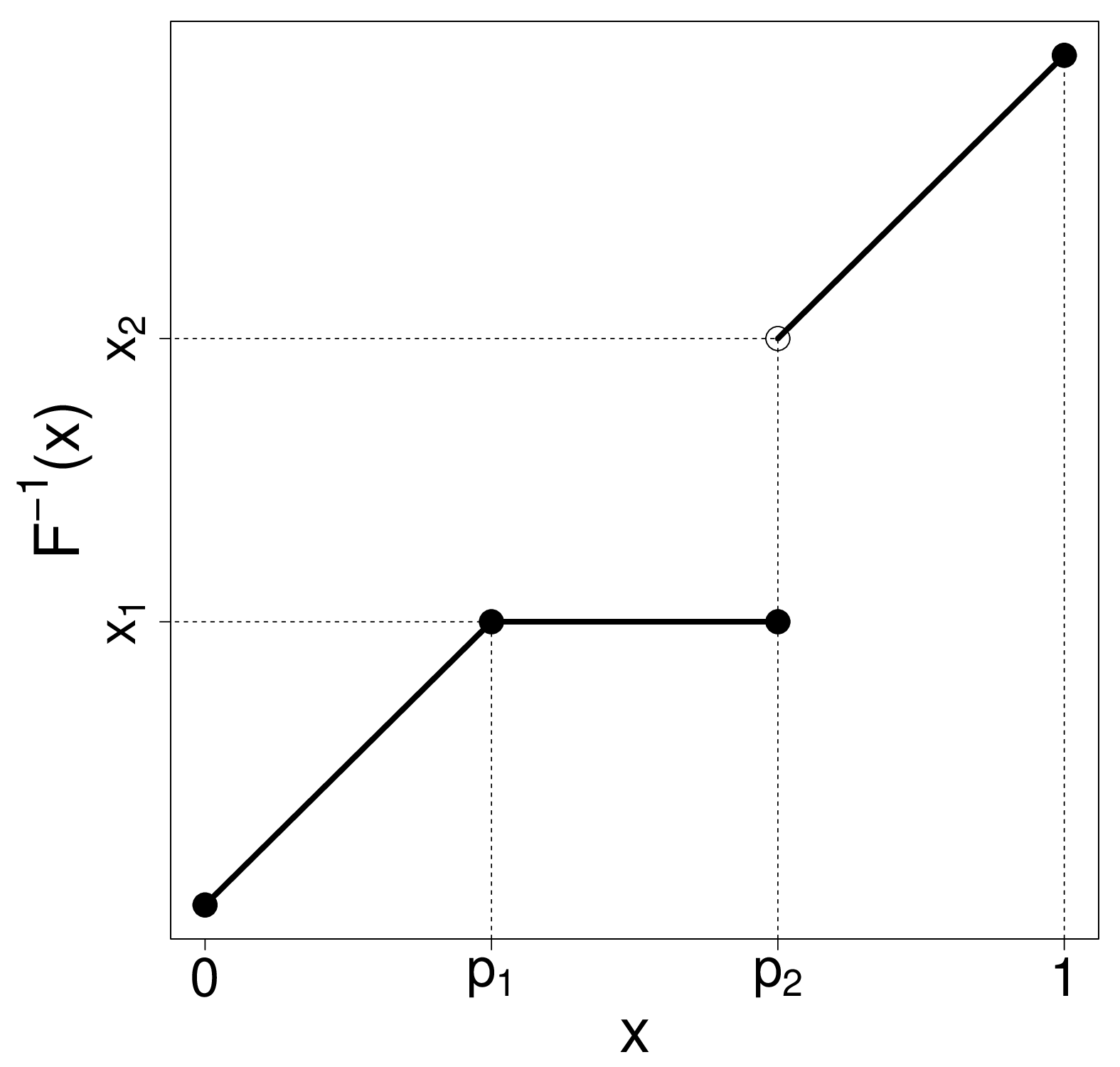}
        \caption{Quantile function: The only discontinuity of $F^{-1}$ is at $p=p_2$, provided that $x_1<x_2$. When $x_1=x_2$ (no jump), the quantile function is continuous.}
        \label{fig:QF-new}
    \end{subfigure}
    \caption{An illustrative cdf and its quantile function with the boundary points $p_1\le p_2 $ and  $x_1\le x_2$ of the distributional jump and gap, respectively, where $x_1=F^{-1}(p_2)$.}
    \label{figure-gap-0}
\end{figure}
illustrates the reformulation. Note in the figure that $F$ is strictly increasing at the point $F^{-1}(p)$ for every $p\in (0,1)$ except when $p=p_2$, in which case there is a flat region adjacent to $x_1=F^{-1}(p_2)$ and, therefore, the quantile function has a jump of size $x_2-x_1$ at the point $p_2$. 
Given this explanation, for the sake of brevity, throughout the rest of the paper we follow \citet{s1973} and formulate the condition simply as continuity of the quantile function $F^{-1}$ at the point $p$. 
\end{note}

In Section~\ref{sect-lower}, we develop a general statistical inference theory for the integral 
\begin{equation}\label{int-lower}
\int_0^p F^{-1}(u)\dd u  
\end{equation}
of quantiles over the lower-most layer $(0,p)$ of probabilities, where $p\in (0,1)$ is a fixed boundary probability. Throughout the current paper we call integral~\eqref{int-lower} the lower-layer integral, although in the literature the integral is often called the generalized (or absolute) Lorenz curve \citep[e.g.,][and references therein]{dz2002,hz2005,dksz2007} to reflect the fact that by dividing the integral by the mean $\mu $ of $F$, we obtain the Lorenz curve (see Example~\ref{illustration-3}) that permeates the literature on economic inequality. Integral~\eqref{int-lower} also plays a significant role in finance, because financial losses are usually modelled using negative random variables. Risk measures such as the ``downside'' tail-value-at-risk (TV@R) arise \citep[e.g.,][]{pr2007}.  Its formula with an accompanying discussion will be provided in Example~\ref{illustration-2}.

\begin{note}\label{note-1ab}
Analogous observations to those made in Notes~\ref{note-1a} and~\ref{note-1aa} apply to integral~\eqref{int-lower} as well: first, we exclude the statistically trivial cases $p=0$ and $p=1$, and second, asymptotic normality of the empirical estimator when $p\in (0,1)$, unlike in the case $p=1$ that gives the mean $\mu $, requires the quantile function $F^{-1}$ to be continuous at the point $p$, as shown by \citet{s1973} in the case of trimmed mean arising from SRS.  No such condition is required for consistency of the estimator of the lower-layer integral.  Details are in Section~\ref{sect-lower}. 
\end{note}

There are also problems, like those associated with insurance layers \citep[e.g,][]{w1996},  when aggregating middle quantiles is of interest. Looking from a different angle, there are also problems, like those concerning robust parameter estimation \citep{bjz2009,zbg2018}, when chopping off, or tempering, a certain percentage of extreme (smallest and largest) quantiles during moment calculations becomes  warranted, thus naturally leading to integrals of moderate (i.e., middle) quantiles. Hence, in Section~\ref{sect-middle} we develop a general statistical inference theory for the integral 
\begin{equation}\label{int-middle}
\int_{p_1}^{p_2} F^{-1}(u)\dd u 
\end{equation}
of quantiles over the middle layer $(p_1,p_2)$ of probabilities, where $0<p_1<p_2<1$ are fixed boundary probabilities. We call integral~\eqref{int-middle} the middle-layer integral. 

\begin{note}\label{note-1b}
When $p_1=0$ and $p_2=1$, integral~\eqref{int-middle} reduces to the mean $\mu$ of the cdf $F$. When $p_1=0<p_2<1$, integral~\eqref{int-middle} reduces to integral~\eqref{int-lower}, and when $0<p_1<p_2=1$, integral~\eqref{int-middle} reduces to integral~\eqref{int-upper}. Hence, by restricting our study of the middle-layer integral to the case  $0<p_1<p_2<1$ only, we are not losing anything from the statistical point of view. 
\end{note}

\begin{note}\label{note-1ac}
Asymptotic normality of the empirical estimator when $0<p_1<p_2<1$ requires the quantile function $F^{-1}$ to be continuous at the points $p_1$ and $p_2$, as seen from \citet{s1973} in the case of trimmed mean arising from SRS. No such condition is required for consistency of the estimator.  Details are in Section~\ref{sect-middle}. 
\end{note}

When modeling both profits and losses (so-called P\&L variables), the underlying population distribution may span the entire real line $\mathbb{R}$, and thus both upper- and lower-layer integrals~\eqref{int-upper} and~\eqref{int-lower} may show up simultaneously. There are also problems that give rise to various combinations of the three integrals introduced above, and one of such instances is the Gini curve (GC), whose formula with an accompanying discussion will be provided in Example~\ref{illustration-4}. The general results that we shall develop in the following sections are readily applicable to such combinations as well.

As we have already alluded to, in Section~\ref{examples} we shall illustrate asymptotic results that can be obtained for the upside and downside tail-values-at-risk, as well as for the Lorenz and Gini curves. To maximally simplify the illustrations, yet to maintain their informative nature, we shall use SRS, that is, we shall work with iid random variables. We stress, however, that this choice of the sampling design is only to facilitate illustrations of our general results, which do not rely on any particular sampling design or dependence structure.

In Section~\ref{time-series-data} we shall show how our general results can be used in the case of time series data. In particular, we shall show that all that needs to be done is: 
\begin{itemize} 
\item 
to verify \textit{uniform} convergence, as well as the rate of \textit{uniform} convergence, of approximating empirical cdf's to the population cdf; 
\item 
to establish convergence of empirical \textit{integrated} cdf's to the corresponding \textit{integrated} population cdf's; 
\item 
to determine limiting distributions. 
\end{itemize}
We shall implement this program in the case of two very broad classes of stationary time series, whose dependence structures are called $S$- and $M$-mixing.

Section~\ref{sect-L-functional} connects the results of previous sections with research areas such as $L$-statistics, distortion and spectral risk measures.  Section~\ref{conclusion} concludes the paper with a brief summary. Longer proofs, accompanying lemmas, and some cursory technicalities are in Appendix~\ref{appendix}.

\section{The upper-layer integral}
\label{sect-upper}

In this section we develop statistical inference for the upper-layer integral $\int_p^1 F^{-1}(u)\dd u  $ at any fixed lower probability level 
\[
p\in (0,1).
\]
Of course, we need to assume that $F$ is from the class  $\mathcal{F}^{+}_{1}$ of those cdf's for which the upper quantiles are integrable, that is, the integral $\int_p^1 F^{-1}(u)\dd u  $ is finite. Note that this condition is the same as to require that the positive part $X^{+}:=X \vee 0$ of the random variable $X\sim F$ has a finite first moment $\mathbb{E}(X^{+})<\infty $.

The functional 
\[
\mathcal{F}^{+}_{1} \ni F \mapsto \int_p^1 F^{-1}(u)\dd u \in \mathbb{R} 
\]
is not linear, which makes it challenging to derive statistical inference. It appears, however, that when developing statistical inference, the functional can be approximated by the linear one
\begin{equation}\label{linear-upper}
\mathcal{F}^{+}_{1} \ni G \mapsto  \int^{\infty }_{F^{-1}(p)}  \big( 1- G(x)\big) \dd x \in \mathbb{R} . 
\end{equation}
Indeed, we shall see from the following corollaries that the difference between the empirical upper-layer integral $\int_p^1 F_n^{-1}(u)\dd u$ and its population counterpart $\int_p^1 F^{-1}(u)\dd u$ gets asymptotically close to the difference between $\int^{\infty }_{F^{-1}(p)}  \big( 1- F_n(x)\big) \dd x $ and $\int^{\infty }_{F^{-1}(p)}  \big( 1- F(x)\big) \dd x $, where $F_n \in \mathcal{F}^{+}_{1}$, $n \in \mathbb{N}$, are cdf's approaching $F$ when the parameter $n$ grows indefinitely.  

\begin{note} 
If $F_n$ is a deterministic cdf, that is, $F_n: \mathbb{R} \to [0,1]$, then the meaning of statements like $F_n \in \mathcal{F}^{+}_{1}$  and $\int^{\infty }_{F^{-1}(p)}  \big( 1- F_n(x)\big) \dd x \in \mathbb{R} $ is clear. 
If, however, $F_n$ is a random cdf, that is, $F_n: \Omega \times \mathbb{R} \to [0,1]$, where $\Omega $ denotes the sample space, then $F_n \in \mathcal{F}^{+}_{1}$ means that, for almost all $\omega \in \Omega $, the deterministic cdf $F_n(\omega , \cdot): \mathbb{R} \to [0,1]$ is an element of $\mathcal{F}^{+}_{1}$. Likewise, $\int^{\infty }_{F^{-1}(p)}  \big( 1- F_n(x)\big) \dd x \in \mathbb{R} $ means $\int^{\infty }_{F^{-1}(p)}  \big( 1- F_n(\omega ,x)\big) \dd x \in \mathbb{R} $ for almost all $\omega \in \Omega $. Furthermore, to enable ourselves to calculate, e.g., the expected values of integrals like $\int_p^1 F_n^{-1}(u) \dd u $ and $\int^{\infty }_{F^{-1}(p)}  \big( 1- F_n(x)\big) \dd x $, we need joint measurability of $F_n: \Omega \times \mathbb{R} \to [0,1]$. We do not emphasize these subtle properties in the paper because they are automatically satisfied by empirical cdf's arising from $n$ random variables $X_1,\ldots, X_n$, and all random cdf's in this paper are of such type. Note in this regard, for example, that $\int_0^1 F_n^{-1}(u) \dd u \in \mathbb{R} $ is equivalent to $(X_1+\cdots + X_n)/n\in \mathbb{R} $, and the latter statement is, of course, true almost surely. Nevertheless, it is prudent to keep a watchful eye on all these matters. 
\end{note}

The aforementioned corollaries rely on the following fundamental theorem, whose proof is given in Appendix~\ref{appendix-2}. 

\begin{theorem}\label{theorem-0upper}
Let $F$ and $G$ be any two cdf's from $\mathcal{F}^{+}_{1}$. Then 
\begin{equation}\label{eq-1upper}
\int_p^1 \big( G^{-1}(u)-F^{-1}(u) \big) \dd u 
=\int^{\infty }_{F^{-1}(p)}  \big( F(x)- G(x)\big) \dd x - \Rem(p;F,G), 
\end{equation}
where the remainder term 
\begin{equation}\label{qq-3upper} 
\Rem(p;F,G) 
=\int^{F^{-1}(p)}_{G^{-1}(p)}  \big( G(x)-p\big) \dd x 
\end{equation}
is non-negative, that is, $\Rem(p;F,G)\ge 0$, and satisfies the bounds 
\begin{align}
\Rem(p;F,G) 
&\le \int^{F^{-1}(p)}_{G^{-1}(p)}  \big( G(x) - F(x)\big) \dd x  
\label{r-prop-0upper}
\\
&\le  \big| G^{-1}(p)-F^{-1}(p)\big|~\sup_{x\in \mathbb{R}} \big| G(x) - F(x)\big| . 
\label{r-prop-1upper}
\end{align}
\end{theorem}

\begin{note}\label{note-rem}
The functional $(F,G)\mapsto \Rem(p;F,G)\in [0,\infty) $ is not symmetric. Note that its symmetrization $\Rem(p;F,G)+\Rem(p;G,F)$ is equal to the right-hand side of bound~\eqref{r-prop-0upper}, that is, 
\begin{equation}\label{note-symm} 
\Rem(p;F,G)+\Rem(p;G,F)= \int^{F^{-1}(p)}_{G^{-1}(p)}  \big( G(x) - F(x)\big) \dd x . 
\end{equation}
This equation implies bound~\eqref{r-prop-0upper} because the remainder term $\Rem(p;G,F)$ is non-negative, just like $\Rem(p;F,G)$ is.  
\end{note}

We are now ready to formulate and discuss several corollaries to Theorem~\ref{theorem-0upper}. Their formulations will involve approximating sequences $F_1,F_2,\ldots $ of cdf's, which may or may not be random, depending on the application. For example, we might be dealing with deterministic ``contaminated'' cdf's $F_n$ given by the equation  
\[
F_n(x)=(1-\delta_n)F(x)+\delta_n C(x), 
\] 
where $C$ is a (deterministic) cdf and $(0,1) \ni\delta_n\downarrow 0$ are contamination proportions. An obvious example of random cdf's would be the empirical cdf 
\[
F_{n}(x) = {1\over n} \sum_{i=1}^n \mathds{1}\{X_i\le x\}
\]
based on random variables $X_1,\dots, X_n$, where  $\mathds{1}\{X_i\le x\} $ is the indicator that takes value $1$ when the statement $X_i\le x$ is correct and $0$ otherwise.  Nevertheless, since every deterministic case is a special (i.e., degenerate) case of the random one, irrespective of the nature of the approximating cdf's $F_n$ we can always treat them as random functions, that is, defined on $\Omega \times \mathbb{R} $, where $\Omega $ is the sample space, given the underlying probability space $\{\Omega, \mathcal{A}, \mathbb{P}\}$. Hence, for example, the requirement $F_n\in \mathcal{F}^{+}_{1}$ that we shall see in the following corollary will mean that, almost surely, $F_n$ is an element of the class $\mathcal{F}^{+}_{1}$, that is, the integral $\int_p^1 F_n^{-1}(u) \dd u $  is finite almost surely, for every value of $p\in (0,1)$.

\begin{corollary}[consistency]\label{corollary-1upper} 
Suppose that $F\in \mathcal{F}^{+}_{1}$, and let $F_1,F_2,\ldots \in \mathcal{F}^{+}_{1}$ be any sequence of cdf's such that 
\begin{equation}
\sup_{x\in \mathbb{R}} \big| F_n(x) - F(x)\big| \stackrel{\mathbb{P}}{\to} 0 
\label{cond-0upper}
\end{equation}
when $n\to \infty $. Then 
\begin{equation}\label{eq-1aupper}
\int_p^1 F_n^{-1}(u) \dd u -\int_p^1 F^{-1}(u) \dd u 
=\int^{\infty }_{F^{-1}(p)}  \big( F(x)- F_n(x)\big) \dd x +o_{\mathbb{P}}(1) 
\end{equation}
and, therefore, the consistency statement 
\begin{equation}\label{eq-1aupper-a}
\int_p^1 F_n^{-1}(u) \dd u \stackrel{\mathbb{P}}{\to}\int_p^1 F^{-1}(u) \dd u 
\end{equation}
holds if and only if 
\begin{equation}\label{eq-1aupper-b}
\int^{\infty }_{F^{-1}(p)}  \big( 1- F_n(x)\big) \dd x \stackrel{\mathbb{P}}{\to} \int^{\infty }_{F^{-1}(p)}  \big( 1- F(x)\big) \dd x . 
\end{equation}
\end{corollary}

\begin{proof} 
The corollary follows from Theorem~\ref{theorem-0upper} with $F_n$ instead of $G$, because $F^{-1}(p)$ is finite and $F_n^{-1}(p)$ is asymptotically bounded. Hence,    condition~\eqref{cond-0upper} implies $\Rem(p;F,G)=o_{\mathbb{P}}(1)$ due to bound~\eqref{r-prop-1upper}. This establishes statement~\eqref{eq-1aupper}, which implies the equivalence of statements~\eqref{eq-1aupper-a} and~\eqref{eq-1aupper-b}, and concludes the proof of Corollary~\ref{corollary-1upper}. 
\end{proof}

\begin{example}[SRS] \label{example-upper1}
Let $X_1,\dots, X_n$ be SRS, that is, iid random variables. The empirical cdf $ F_{n,\srs}$ is defined by the equation  
\[
 F_{n,\srs}(x) = {1\over n} \sum_{i=1}^n \mathds{1}\{X_i\le x\},
\]
which is the arithmetic mean of $n$ independent copies of the Bernoulli random variable  $\mathds{1}\{X\le x\} $ taking value $1$ with probability $F(x)$ and $0$ otherwise. In other words, $\mathds{1}\{X\le x\} $ is the indicator that takes value $1$ when the statement $X\le x$ is true and $0$ otherwise.  Analogously to equation~\eqref{qq-0}, the empirical quantile function $F_{n,\srs}^{-1}$ is defined  by 
\begin{align}
F_{n,\srs}^{-1}(u)
&=\inf\{ x \in \mathbb{R} ~:~ F_{n,\srs}(x)\ge u\}
\label{qq-0srs}
\\
&= X_{\lceil nu \rceil :n} 
\label{quantile-upper-2} 
\end{align}
for all $u\in (0,1]$, where $X_{1:n}\le \cdots \le X_{n:n}$ are the order statistics of the random variables $X_{1}, \dots , X_{n}$, and, for any real $x\in \mathbb{R}$, $\lceil x \rceil$ denotes the smallest integer $k$ such that $x\le k$, that is, $x \mapsto \lceil x \rceil$ is the classical ceiling function. When $u=0$, we define 
\[
F_{n,\srs}^{-1}(0)=X_{1:n}=\min_{1\le i \le n} X_i. 
\]
Condition~\eqref{cond-0upper} with $F_{n,\srs}$ in place of $F_n$ is satisfied by the Glivenko-Cantelli theorem. Hence, we are left to verify condition~\eqref{eq-1aupper-b}. By linearity of functional~\eqref{linear-upper}, the integral $\int^{\infty }_{F^{-1}(p)}  \big( 1- F_{n,\srs}(x)\big) \dd x$ is the arithmetic mean $n^{-1}\sum_{i=1}^n Y_{i,\one}(p)$ of $n$ independent copies of the random variable 
\begin{align}
Y_{\one}(p)
&=\int^{\infty }_{F^{-1}(p)}  \mathds{1}\{X>x\} \dd x  
\notag 
\\
&=\big( X-F^{-1}(p) \big)^{+} . 
\label{rv-upper}
\end{align}
The random variable $Y_{\one}(p)$ has a finite first moment because $F\in \mathcal{F}^{+}_{1}$ (see Lemma~\ref{lemma-1upper} for details). Hence, statement~\eqref{eq-1aupper-b} holds and so, by Corollary~\ref{corollary-1upper}, we have that $\int_p^1 F_{n,\srs}^{-1}(u)\dd u $ is a consistent estimator of $\int_p^1 F^{-1}(u)\dd u $. In summary, we conclude that when $F\in \mathcal{F}^{+}_{1}$, we have 
\begin{equation}\label{consistency-upper}
\int_p^1 F_{n,\srs}^{-1}(u)\dd u  \stackrel{\mathbb{P}}{\to} \int_p^1 F^{-1}(u)\dd u . 
\end{equation}
\end{example} 

\begin{note}  
In addition to equation~\eqref{quantile-upper-2}, there are also other ways to express $F_{n,\srs}^{-1}(u)$ in terms of the order statistics $X_{1:n}\le \cdots \le X_{n:n}$. For example, using the integer part $[x]$ of $x\in \mathbb{R}$, in the section on sample quantiles by  \citet{s2003}, we find the equation 
\[
F_{n,\srs}^{-1}(u)=c_{nu} X_{[nu]:n}+(1-c_{nu})X_{([nu]+1):n}, 
\]
where $c_{nu}=1$ if $nu$ is an integer and $c_{nu}=0$ otherwise. 
\end{note}

\begin{corollary}[bias]\label{corollary-2upper} 
Suppose that  $F\in \mathcal{F}^{+}_{1}$, and let $F_1,F_2,\ldots \in \mathcal{F}^{+}_{1}$ be any sequence of cdf's that are unbiased estimators of $F$, that is, 
\begin{equation}
\mathbb{E}\big( F_n(x) \big) =F(x)
\label{cond-0bupper}
\end{equation}
for all $x\in \mathbb{R}$  and $n\in \mathbb{N}$, and such that 
\begin{equation}
\mathbb{E}\big( F_n^{-1}(p) \big) \in \mathbb{R}
\label{cond-0bupper-bias}
\end{equation}
for all $n\in \mathbb{N}$. Then $\int_p^1 F_n^{-1}(u)\dd u $ is a non-positively biased estimator of $\int_p^1 F^{-1}(u)\dd u $. 
\end{corollary}

\begin{proof} 
Using Theorem~\ref{theorem-0upper}  with $F_n$ instead of $G$, we have 
\begin{align*}
\mathrm{Bias}_n^{\one}(p):= &\mathbb{E}\bigg(\int_p^1 F_n^{-1}(u)\dd u\bigg)  -\int_p^1 F^{-1}(u)\dd u 
\\
= &  -\mathbb{E}\big(\Rem(p;F,F_n)\big) 
\\
\in & (-\infty, 0]
\end{align*}
for every $n\in \mathbb{N}$.  Note that $\mathbb{E}\big(\Rem(p;F,F_n)\big) $ is finite because of 
bound~\eqref{r-prop-1upper} and condition~\eqref{cond-0bupper-bias}. This establishes Corollary~\ref{corollary-2upper}. 
\end{proof}

\begin{note}\label{note+bias}
From the proof of Corollary~\ref{corollary-2upper} we see why and how the bias 
\begin{equation}\label{note+bias-up}
\mathrm{Bias}_n^{\one}(p)=
\left\{
  \begin{array}{ll}
    -\Rem(p;F,F_n)                     & \hbox{when $F_n$ is deterministic,} \\
    -\mathbb{E}\big(\Rem(p;F,F_n)\big) &  \hbox{when $F_n$ is random,}
  \end{array}
\right.
\end{equation}
arises, where 
\[
\Rem(p;F,F_n)=\int^{F^{-1}(p)}_{F_n^{-1}(p)}  \big( F_n(x)-p\big) \dd x.  
\]
These formulas can, in turn, be used to further analyze the bias $\mathrm{Bias}_{n}^{\one}(p)$, assess its magnitude, and possibly find ways to mitigate it in real data-driven applications. For example, if the cdf $F$ has a pdf $f=F'$ that is strictly positive and continuous at the quantile $F^{-1}(p)$, then following the studies of \citet{gh2006,gh2007} on trimmed means,  we can show that in the (random) case of $F_n=F_{n,\srs}$, the SRS-based bias $\mathrm{Bias}_{n,\srs}^{\one}(p)$ attains the asymptotic expansion  
\begin{equation}\label{bias-as-upper}
\mathrm{Bias}_{n,\srs}^{\one}(p)
= -\frac{p(1-p)}{2nf(F^{-1}(p))}+ o\left(\frac 1n\right)
\end{equation}
when $n\to \infty $. We can now construct an empirical estimator for the main term in asymptotic expansion~\eqref{bias-as-upper} by using, for example,  the empirical estimator $F_{n,\srs}^{-1}(p)$ for $F^{-1}(p)$ and a kernel density estimator for $f$ \citep[e.g.,][]{r1983}. For a discussion of the role of the density-quantile function $p\mapsto f(F^{-1}(p))$ in statistics, and for a comprehensive statistical inference theory for it, including its behaviour near the end points $p=0$ and $p=1$, we refer to the pioneering studies of \citet{p1979a,p1979b}.
\end{note}

\begin{example}[SRS] \label{example-upper2}
When $F_n$ is $F_{n,\srs}$, condition~\eqref{cond-0bupper} is obviously satisfied. Condition~\eqref{cond-0bupper-bias} is satisfied when \citep[][Proposition~2, p.~679]{s1974} 
\begin{quote}
there is $\varepsilon>0$ such that $x^{\varepsilon}(1-F(x)+F(-x))\to 0$ when $x\to \infty $, which is equivalent to saying that $x^{\varepsilon}\mathbb{P}(|X|\ge x)\to 0$, where $X\sim F$.  
\end{quote} 
For a related study that connects condition~\eqref{cond-0bupper-bias} with the order of finite moments of $F$, we refer to \citet{g1995} (see also next Note~\ref{note+epsilon}). Denote the class of all cdf's satisfying the aforementioned tail-based condition of \citet{s1974} by $\mathcal{T}_{\varepsilon}$.  Hence, Corollary~\ref{corollary-2upper} tells us that when $F\in \mathcal{F}^{+}_{1} \cap \mathcal{T}_{\varepsilon}$  for some $\varepsilon>0$, the estimator 
$\int_p^1 F_{n,\srs}^{-1}(u)\dd u  $ of the upper-layer integral $ \int_p^1 F^{-1}(u)\dd u $, although being consistent by Example~\ref{example-upper1}, has non-positive bias
\begin{equation}
\mathrm{Bias}_{n,\srs}^{\one}(p):=-\mathbb{E}\big(\Rem(p;F,F_{n,\srs})\big) \in (-\infty, 0] . 
\label{bias-upper-srs}
\end{equation}
\end{example}

\begin{note}\label{note+epsilon}
The condition that $F\in \mathcal{F}^{+}_{1} \cap \mathcal{T}_{\varepsilon}$  for some $\varepsilon>0$ can equivalently be reformulated as $F\in \mathcal{F}^{+}_{1} \cap \mathcal{T}^{-}_{\varepsilon}$  for some $\varepsilon>0$, where $ \mathcal{T}^{-}_{\varepsilon}$ consists of all those cdf's that satisfy $|x|^{\varepsilon}F(x)\to 0$ when $x\to -\infty $. This reformulation is possible due to the fact that $F\in \mathcal{F}^{+}_{1} $ implies $x^{\varepsilon}(1-F(x))\to 0$ when $x\to \infty $ for every $\varepsilon \in (0,1)$, and we need this condition to hold for just one $\varepsilon>0$. For those researchers who think in terms of moments, a sufficient condition that ensures the existence of $\varepsilon>0$ such that $F\in \mathcal{F}^{+}_{1} \cap \mathcal{T}^{-}_{\varepsilon}$ would be the requirement that 
\begin{equation}
\mathbb{E}(X^{+})<\infty \quad \text{and} \quad \mathbb{E}((X^{-})^{\varepsilon})<\infty 
\label{req-upper}
\end{equation}
for some $\varepsilon >0$, where $X^{-}=(-X) \vee 0$ is the negative part of the random variable $X\sim F$.   Note that requirement~\eqref{req-upper} can equivalently be rewritten as $F\in \mathcal{F}^{+}_{1} \cap \mathcal{F}^{-}_{\varepsilon}$. 
\end{note}

The next corollary plays a pivotal role in deriving the asymptotic distribution of the empirical upper-layer integral. We note at the outset that the corollary crucially relies on the rate of convergence to $0$, denoted by $O_{\mathbb{P}}(1/A_n)$, of the supremum in condition~\eqref{cond-0upper}. For example, we would typically have $A_n=\sqrt{n}$ for iid sequences and, more broadly, for weakly-dependent stationary sequences $(X_n)$, whereas we would have $A_n=n^\theta $ with some $\theta \in (0,1/2)$ for long-memory random sequences. We shall give a more detailed note on the topic near the end of this section. At the moment we emphasize that next Corollary~\ref{corollary-3upper} does not depend on any sampling design or dependence structure. 

\begin{corollary}[asymptotic distribution]\label{corollary-3upper} 
Suppose that  $F\in \mathcal{F}^{+}_{1}$, and let $F_1,F_2,\ldots \in \mathcal{F}^{+}_{1}$ be any sequence of cdf's such that 
\begin{equation}
F_n^{-1}(p) \stackrel{\mathbb{P}}{\to} F^{-1}(p) 
\label{cond-1upper}
\end{equation}
and 
\begin{equation}
A_n\sup_{x\in \mathbb{R}} \big| F_n(x) - F(x)\big| =O_{\mathbb{P}}(1) 
\label{cond-2upper}
\end{equation}
when $n\to \infty $, where $A_n\to \infty $ are some normalizing constants. Then 
\begin{equation}\label{eq-1dupper}
A_n\left( \int_p^1 F_n^{-1}(u) \dd u -\int_p^1 F^{-1}(u) \dd u \right) 
=A_n\int^{\infty }_{F^{-1}(p)}  \big( F(x)- F_n(x)\big) \dd x +o_{\mathbb{P}}(1) 
\end{equation}
and, therefore, the convergence-in-distribution statement 
\begin{equation}\label{eq-1dupper-a}
A_n\left( \int_p^1 F_n^{-1}(u) \dd u -\int_p^1 F^{-1}(u) \dd u \right) 
\stackrel{d}{\to} \mathcal{L}_{\one}(p)  
\end{equation}
holds if and only if 
\begin{equation}\label{eq-1dupper-b}
A_n\left( \int^{\infty }_{F^{-1}(p)}  \big( 1- F_n(x)\big) \dd x 
-\int^{\infty }_{F^{-1}(p)}  \big( 1- F(x)\big) \dd x  \right)\stackrel{d}{\to} \mathcal{L}_{\one}(p) ,  
\end{equation}
where $\mathcal{L}_{\one}(p)$ is a random variable determined by statement~\eqref{eq-1dupper-b}. 
\end{corollary}

\begin{proof} 
Using  Theorem~\ref{theorem-0upper} with $F_n$ instead of $G$, we only need to show that $A_n\,\Rem(p;F,F_n)$ converges in probability to $0$. Since $\Rem(p;F,F_n)$ is non-negative and does not exceed 
\[
\big| F_n^{-1}(p)-F^{-1}(p)\big|~\sup_{x\in \mathbb{R}} \big| F_n(x) - F(x)\big|, 
\]
the statement $A_n\Rem(p;F,F_n)=o_{\mathbb{P}}(1)$ follows from conditions~\eqref{cond-1upper} and~\eqref{cond-2upper}.  This establishes equation~\eqref{eq-1dupper}, which implies the equivalence of statements~\eqref{eq-1dupper-a} and~\eqref{eq-1dupper-b}, and concludes the proof of Corollary~\ref{corollary-3upper}. 
\end{proof}

\begin{note} 
The limiting random variable $\mathcal{L}_{\one}(p)$ may or may not be normal. The nature of this random variable is determined by the sequence $(X_n)$: if it is iid or, more broadly, stationary and weakly dependent, then $\mathcal{L}_{\one}(p)$ is normal, but if the stationary sequence $(X_n)$ is long-range dependent, then the limiting random variable $\mathcal{L}_{\one}(p)$, if it exists, is typically non-normal. 
\end{note}

Concerning condition~\eqref{cond-1upper}, it is a simple probability exercise to check the validity of the following lemma by imitating related considerations in the section on SRS-based quantiles by \citet[][]{s2003}.

\begin{lemma}\label{lemma-1} 
When the quantile function $F^{-1}$ is continuous at the point $p$, condition~\eqref{cond-0upper} implies condition~\eqref{cond-1upper} and, therefore, the latter condition can be dropped from  Corollary~\ref{corollary-3upper} because condition~\eqref{cond-2upper} implies~\eqref{cond-0upper} and, therefore, it also implies condition~\eqref{cond-1upper}. 
\end{lemma}

\begin{note} 
In the SRS-based context, that is, when $F_n$ is $F_{n,\srs}$, \citet[][]{s2003} derives exponential bounds for the probability $\mathbb{P}(|F_{n,\srs}^{-1}(p) - F^{-1}(p)|\ge \varepsilon )$ with the aim at establishing strong consistency. We are concerned with (weak) consistency and so do not need exponential bounds. In other words, we just need convergence of $\mathbb{P}(|F_{n}^{-1}(p) - F^{-1}(p)|\ge \varepsilon )$ to $0$ for every $\varepsilon >0$ when $n\to \infty $, and for this, condition~\eqref{cond-2upper} imposed on the generic sequence $F_n$, $n\in \mathbb{N}$, is sufficient. 
\end{note}

In Section~\ref{sect-upper-sim} we shall give an illustrative numerical study showing that statement~\eqref{cond-1upper} fails for a cdf $F$ for which the quantile function $F^{-1}$ is discontinuous at a point $p$. The same example will also show that the empirical upper-layer integral $\int_p^1 F_{n,\srs}^{-1}(u) \dd u $ fails asymptotic normality. This phenomenon is natural given the results of \citet{s1973}; see also \citet{gh2011} for further notes on the topic. 

\begin{example}[SRS] \label{example-upper3} 
When $F_n$ is $F_{n,\srs}$, condition~\eqref{cond-2upper} is a consequence of the Kolmogorov-Smirnov theorem. Since condition~\eqref{cond-2upper} is satisfied, condition~\eqref{cond-1upper} is satisfied as well, provided that, according to Lemma~\ref{lemma-1}, the quantile function $F^{-1}$ is continuous at the point $p$.  
Furthermore, since the integral $\int^{\infty }_{F^{-1}(p)}  \big( 1- F_{n,\srs}(x)\big) \dd x$ is the arithmetic mean of $n$ independent copies of random variable~\eqref{rv-upper}, 
we conclude that the integral $\int^{\infty }_{F^{-1}(p)}  \big( 1- F_{n,\srs}(x)\big) \dd x$ satisfies the central limit theorem provided that  random variable~\eqref{rv-upper} has a finite second moment. This is so provided that $F\in \mathcal{F}^{+}_{2}$ (see Lemma~\ref{lemma-1upper} for details), where $\mathcal{F}^{+}_{2}$ denotes the set of all cdf's such that the integral $\int_p^1 \big(F^{-1}(u)\big)^2\dd u $ is finite. Note that $F\in \mathcal{F}^{+}_{2}$ is equivalent to saying that the positive part $X^{+}$ of the random variable $X\sim F$ has a finite second moment $\mathbb{E}((X^{+})^2)<\infty $. 
In summary, when $F\in \mathcal{F}^{+}_{2}$ and the quantile function $F^{-1}$ is continuous at the point $p$, we have the asymptotic normality result 
\begin{equation}\label{normality-upper}
\sqrt{n}\left(\int_p^1 F_{n,\srs}^{-1}(u)\dd u  - \int_p^1 F^{-1}(u)\dd u  \right) 
\stackrel{d}{\to} \mathcal{N}(0,\sigma_{\one}^2(p)), 
\end{equation}
where $\sigma_{\one}^2(p)$ is the variance of the random variable $Y_{\one}(p)$. Note that the variance can be expressed in terms of the cdf $F$ by the formula 
\begin{equation}
\sigma_{\one}^2(p)
=\int^{\infty }_{F^{-1}(p)} \int^{\infty }_{F^{-1}(p)} \big( F(x \wedge y) -F(x)F(y) \big) \dd x  \dd y,      
\label{normality-upper-var2}
\end{equation}
although we have found that the following alternative representations (recall equation~\eqref{rv-upper})  
\begin{align}
\sigma_{\one}^2(p)
&=\Var\left( \big( X-F^{-1}(p) \big)^{+} \right)
\notag 
\\
&=\Var\left( \big( F^{-1}(U)-F^{-1}(p) \big)^{+} \right) 
\label{normality-upper-var2b}
\end{align}  
are sometimes easier to use, where $U$ denotes a uniform on $[0,1]$ random variable. 
\end{example} 

\begin{note}[existence of uniform rv's] 
We often assume for granted that a uniform on $[0,1]$ random variable exists, but whether or not this is the case depends on the underlying probability space $\{\Omega, \mathcal{A}, \mathbb{P}\}$ upon which we are building our statistical experiments and theories. As shown by, for example, \citet[][Proposition~A.31, p.~547]{FS16}, a necessary and sufficient condition for the existence of a uniform on $[0,1]$ random variable is for the underlying probability space to be atomless. Pertaining to the equality $X=F^{-1}(U)$ $\mathbb{P}$-a.s., we refer to \citet[][Lemma~A.32, p.~548]{FS16}. The write-up by \citet[][Appendix~A.1]{rw2024} on atomless probability spaces and their properties is particularly illuminating.  
\end{note}

\begin{note}[dependent data] 
Quite often in applications, data arrive in the form of time series, and thus researchers have been interested in deriving large sample properties for various tail risk measures based on such data \citep[e.g.,][and references therein]{dz2002,dz2003,hz2005,dksz2007,C08,lx2013,zkjf2020,lw2023,mb2024,G2025}.  Corollaries~\ref{corollary-1upper} and~\ref{corollary-3upper}, which are not attached to any particular sampling design or dependence structure, are applicable in such scenarios. 
Indeed, statements like those in condition~\eqref{cond-0upper} for consistency and condition~\eqref{cond-2upper} for asymptotic distribution have been verified for various classes of time series, where the empirical cdf plays the role of $F_n$. 
To give some flavour of what is known, we first note that in the case of weakly-dependent stationary time series, the normalizing constant $A_n$ is $\sqrt{n}$ and the limiting process of $A_n(F_n - F)$ is normal (centered at $0$ and with a variance reflecting the marginal variability as well as the dependence structure of the underlying time series), whereas in the case of long-memory stationary time series, the normalizing constant $A_n$ is of the order $n^{\theta}$ for various $\theta\in (0,1/2)$ and the limiting process, if exists,  is non-normal. For a glimpse of the various normalizing constants and limiting distributions, albeit in the case of the lower-layer integral, we refer to \citet{dz2004}. The monographs by \citet{DMS2002}, \citet{R2017}, and \citet{D2018} offer a wealth of information on empirical processes arising from various classes of times series. In Section~\ref{time-series-data} we shall provide more detailed user-oriented notes and references.  
\end{note}

We next present convenient computational formulas for the SRS-based empirical upper-layer integral $\int_p^1 F_{n,\srs}^{-1}(u)\dd u $.

\begin{computation}[SRS] \label{computation-upper} 
When computing, estimators in the form of integrals are not convenient to work with. Hence, using equation~\eqref{quantile-upper-2}, we shall next express $\int_p^1 F_{n,\srs}^{-1}(u)\dd u$ in terms of the order statistics $X_{1:n}\le \cdots \le X_{n:n}$, whose realized values are known in practice. Namely, we have  
\begin{align}
\int_{p}^{1} F_{n,\srs}^{-1}(u)\dd u
&= \int_{p}^{\lceil np \rceil /n} F_{n,\srs}^{-1}(u)\dd u 
+ \sum_{i=\lceil np \rceil +1}^{n} \int_{(i-1)/n}^{i/n}F_{n,\srs}^{-1}(u)\dd u 
\notag 
\\
&=\underbrace{ {1\over n}\sum_{i=\lceil np \rceil +1}^{n} X_{i :n}}_{\textrm{main term}} 
+ \underbrace{ \vphantom{\sum_{i=\lceil np \rceil +1}^{n}} \left( {\lceil np \rceil \over n}- p\right) X_{\lceil np\rceil :n}}_{\textrm{asymptotically negligible term}}    
\label{quantile-upper-3i}
\\
&=\left( 1-{\lceil np \rceil \over n}\right)
\underbrace{ {1\over n - \lceil np \rceil }\sum_{i=\lceil np \rceil +1}^{n} X_{i :n}}_{\textrm{left-trimmed mean}} 
+ \underbrace{ \vphantom{\sum_{i=\lceil np \rceil +1}^{n}} \left( {\lceil np \rceil \over n}- p\right) X_{\lceil np\rceil :n}}_{\textrm{asymptotically negligible term}} .    
\label{quantile-upper-3ii}
\end{align}
Equations~\eqref{quantile-upper-3i} and~\eqref{quantile-upper-3ii} give  computationally-friendly formulas for the empirical upper-layer integral, which is connected to  the left-trimmed mean \citep[e.g.,][]{s1973} via equation~\eqref{quantile-upper-3ii}. 
\end{computation}

\section{The upper-layer integral: simulated experiments}
\label{sect-upper-sim}

In this section we present several experiments that empirically illustrate some of the theoretical considerations in previous Section~\ref{sect-upper}.

\begin{simulation}[SRS; bias]\label{simul-b} 
We shall now empirically illustrate that $\mathrm{Bias}_{n,\srs}^{\one}(p) $ is non-positive, as stated in Example~\ref{example-upper2}. To avoid confounding negativity of bias with skewness of distributions, we work with a symmetric (around its mean) and light-tailed distribution. Namely, let $V$ be the uniform on $[0,2]$ random variable, whose quantile function is $F^{-1}(u)=2u$ for $0\le u \le 1$. 
We fix the lower-layer probability level 
\[
p={3\over 4} 
\]
and choose several values of $n$ (to be specified in a moment) to get an insight into the behaviour of the bias with respect to the sample size.  
The population upper-layer integral $\int_p^1 F^{-1}(u)\dd u $ is equal to $1-p^2=7/16$. We obtain the first SRS by observing $n$ independent copies $V_1,\dots , V_n$ of $V$, and then calculate the empirical upper-layer integral $\int_{p}^{1} F_{n,\srs}^{-1}(u)\dd u$ according to formula~\eqref{quantile-upper-3i}. This gives us the first difference  
\[
\underbrace{{1\over n}\sum_{i=\lceil 3n/4 \rceil +1}^{n} V_{i :n} + \bigg( {\lceil 3n/4 \rceil \over n} - {3 \over 4}\bigg) V_{\lceil 3n/4\rceil :n} }_{\textrm{empirical}}
- \underbrace{ \vphantom{\sum_{i=\lceil np \rceil +1}^{n}} {7\over 16}}_{\textrm{population}} 
\]
between the empirical and population upper-layer integrals. We denote the difference by $B_n(1)$ and call it the first individual bias. We repeat the procedure $m=10\,000$ times, keeping the same $p$ and $n$ values, and obtain $m$ values of the individual biases $B_n(1), \dots, B_n(m)$. We calculate their arithmetic mean $B^{\ave}_n$ and also the median $B^{\med}_n$. In Table~\ref{table-bias-1} 
\begin{table}[h!]
  \centering
\begin{tabular}{c|ccccc}
  \hline\hline 
   $n$           & 40        & 100      & 200        & 500        & 1000\\
  \hline 
   $B^{\ave}_n$  & -0.004855 & -0.002113 & -0.000944 & -0.000311 & -0.000154 \\
   $B^{\med}_n$  & -0.002482 & -0.001203 & -0.000463 & -0.000207 & -0.000099 \\ 
  \hline
\end{tabular}
  \caption{The averages and medians of $m=10\,000$ biases for various sample sizes  $n$.}
  \label{table-bias-1}
\end{table}
we report the values of $B^{\ave}_n$ and $B^{\med}_n$ for different $n$'s. Since $m$ is very large, we  expect the averages $B^{\ave}_n$ reported in Table~\ref{table-bias-1} to be close to the corresponding theoretical biases $\mathrm{Bias}_{n,\srs}^{\one}$, which are defined by equation~\eqref{bias-upper-srs}. Indeed, the table shows that all $B^{\ave}_n$'s are negative. Furthermore, their absolute values are decreasing when $n$ grows, which is expected given the proven consistency of the empirical upper-layer integral. Nevertheless, as a follow-up research topic, it is worth thinking of modifying the empirical upper-layer integral to reduce its bias. Finally, we note that the reason behind presenting the medians $B^{\med}_n$ in Table~\ref{table-bias-1} is that by comparing them with the averages $B^{\ave}_n$, we can get insights about the skewness of the distribution of individual biases. 
\end{simulation}

\begin{simulation}[SRS; asymptotic normality holds]\label{simul-c} 
Let $V$ be a uniform on $[0,2]$ random variable, and so its cdf is $F(x)=x/2$ for all $0\le x \le 2$. We set the parameter $p$ value to  
\[
p={1\over 2}  
\]
and let the sample size be $n=100\,000$, which is an even number and thus satisfies $\lceil n/2 \rceil =n/2$. We simulate $n$ independent copies  $V_1,\dots , V_n$ of $V$. Statement~\eqref{normality-upper} becomes 
\begin{equation}\label{h-true-2}
\Delta_n:={\sqrt{n}\over \sigma_{\one}(1/2)}\left(\int_{1/2}^1 F_{n,\srs}^{-1}(u)\dd u  - \int_{1/2}^1 F^{-1}(u)\dd u  \right)
\stackrel{d}{\rightarrow} \mathcal{N}(0,1) ,   
\end{equation}
where 
\begin{align}
\int_{1/2}^1 F_{n,\srs}^{-1}(u)\dd u  
& ={1\over n}\sum_{i=(n/2)+1}^{n} V_{i :n} \quad \big(\textrm{because $\lceil n/2 \rceil =n/2$}\big), 
\notag 
\\
\int_{1/2}^1 F^{-1}(u)\dd u & = {3\over 4} ,  
\notag 
\\
\sigma_{\one}^2(1/2) & = {5\over 48}. 
\label{eq-variance-0}
\end{align}
To arrive at the value of $\sigma_{\one}^2(p) $ noted in equation~\eqref{eq-variance-0}, we have used representation~\eqref{normality-upper-var2b} with $F^{-1}(u)=2u$ for all $0\le u \le 1$. Simulation results based on $m=10\,000$ replications of  $\Delta_n$ are depicted in the form of a relative histogram in Figure~\ref{figure-normality-2c},    
\begin{figure}[h!]
  \centering
  \includegraphics[width=0.7\textwidth]{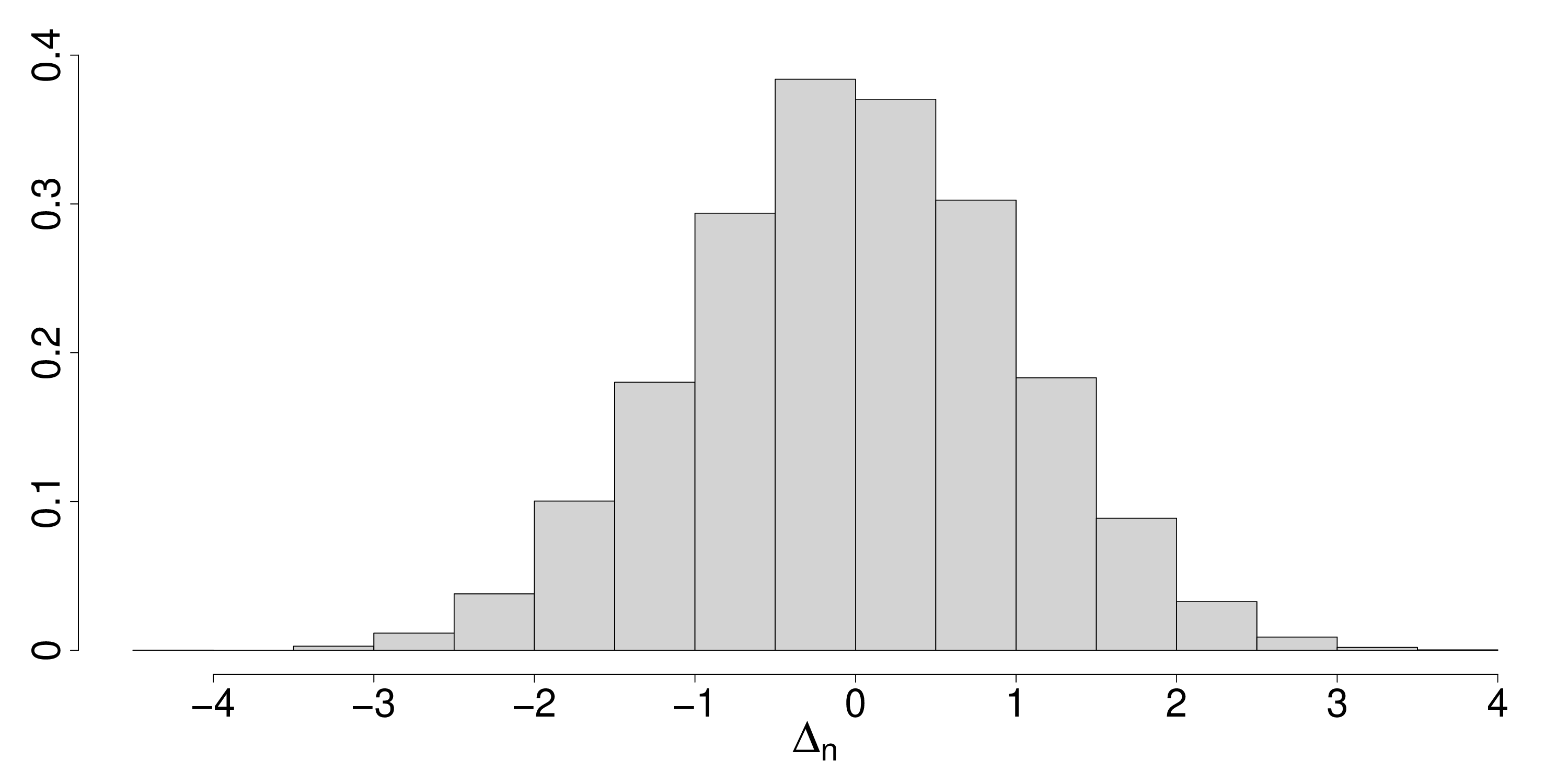}
  \caption{The relative histogram of $m=10\,000$ values of $\Delta_n$ in~\eqref{h-true-2} when $n=100\,000$.}
  \label{figure-normality-2c}
\end{figure}
which supports the limiting standard normal distribution claimed in statement~\eqref{h-true-2}.
\end{simulation}

\begin{simulation}[SRS; asymptotic normality fails]\label{simul-a} 
To illustrate the pivotal role of continuity of the quantile function $F^{-1}$ at the point $p$ and thus of condition~\eqref{cond-1upper} in the statement of asymptotic normality of the empirical upper-layer integral, we next introduce a new random variable, denoted by $Z$. Its definition is based on a uniform on $[0,2]$ random variable $V$. Namely, we fix any constant $a\in [0,1]$ and define a ``gapped'' random variable $Z$ by the formula 
\[
Z= (1-a)V \mathds{1}\{V \le 1\}+\big( 2a+(1-a)V\big) \mathds{1}\{V > 1\}. 
\] 
Note that when $a=0$, the gap vanishes and $Z$ turns into the original uniform random variable $V$, the case that we have discussed in previous Simulation~\ref{simul-c}. When $a=1$, the random variable $Z$ turns into a discrete random variable taking values $0$ and $2$ with probabilities $1/2$. To better understand and appreciate the random variable $Z$ with respect to the parameter $a\in (0,1)$, we note the following properties of $Z$ when it is viewed as a function of the argument $V$: 
\begin{itemize}
\item 
if $V=0$, then $Z=0$; 
\item 
if $V=1$, then $Z=1-a$; 
\item 
if $V \downarrow 1$, then $Z \downarrow 1+a$; 
\item
if $V=2$, then $Z=2$; 
\end{itemize}
Hence, all the values of $Z $ are in the interval $ [0,2]$, although there are no values in $(1-a,1+a]$, which represents the gap in the distribution of $Z$. The cdf $z\mapsto H(z)$ of $Z$, which is visualized in Figure~\ref{figure-gap-1},  
\begin{figure}[h!]
    \centering
    \begin{subfigure}[b]{0.41\textwidth}
        \includegraphics[width=\textwidth]{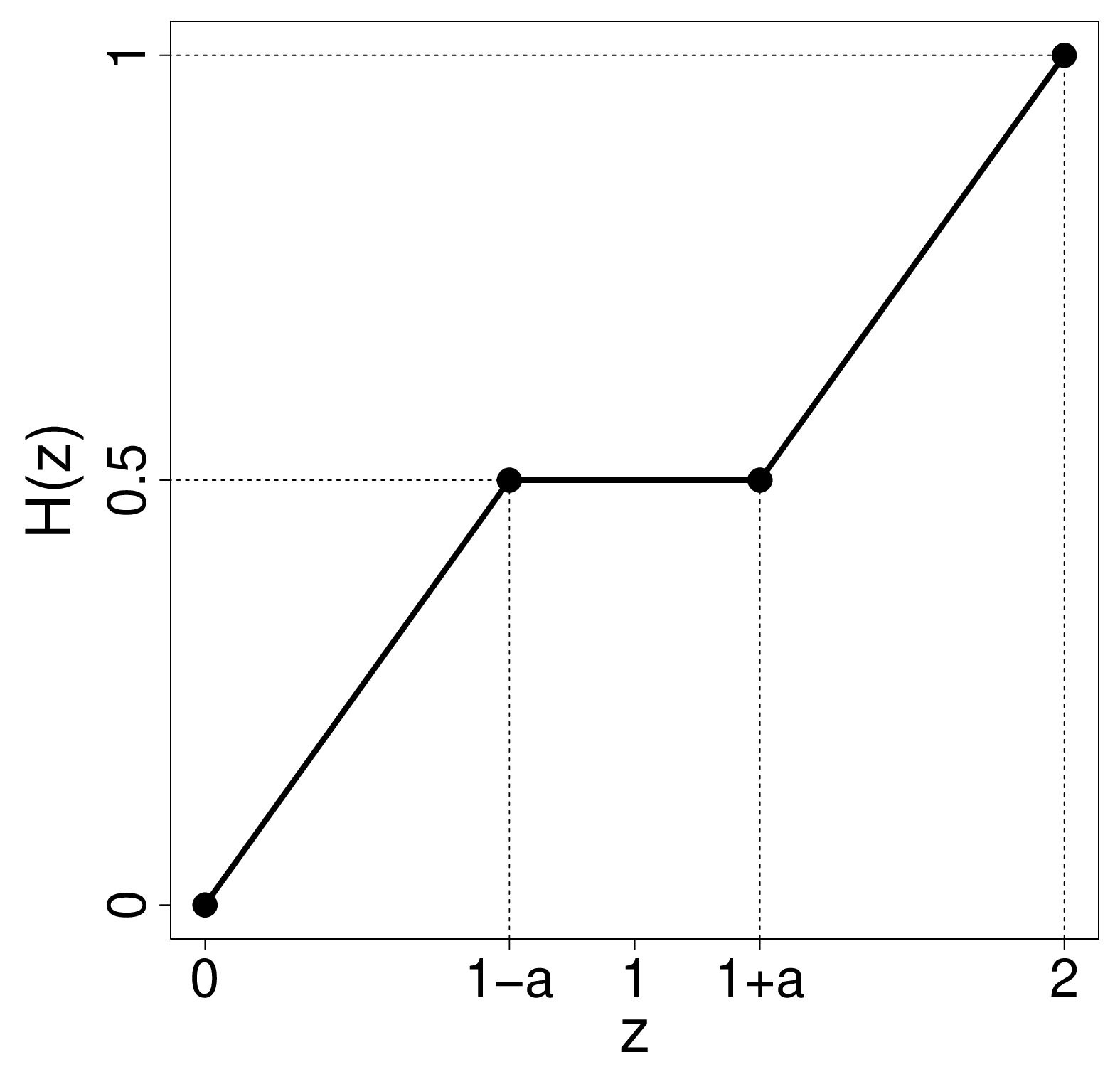}
        \caption{The cdf $H$.}
        \label{fig:ATVaR-old}
    \end{subfigure}
\qquad
    \begin{subfigure}[b]{0.41\textwidth}
        \includegraphics[width=\textwidth]{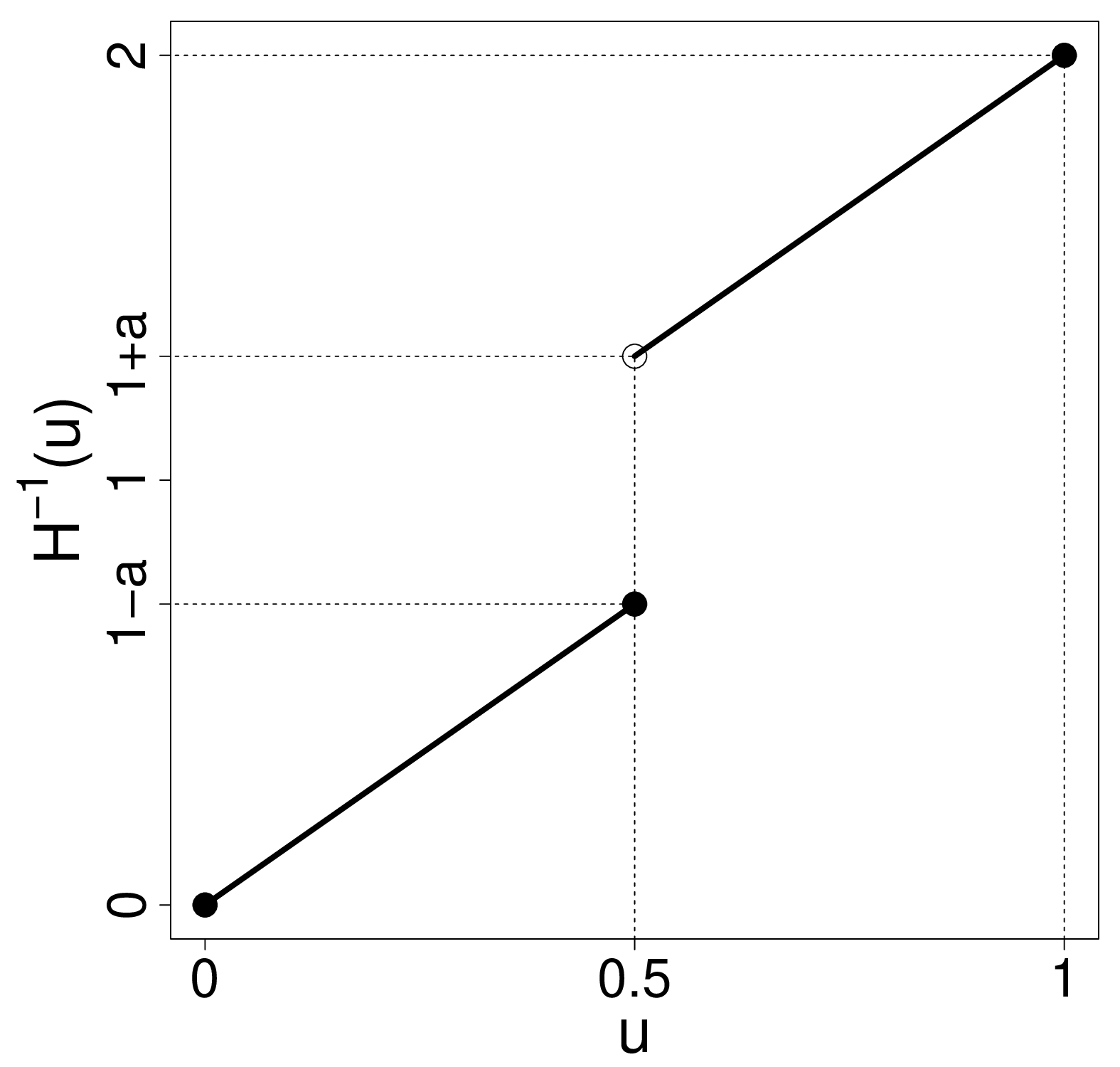}
        \caption{The quantile function $H^{-1}$.}
        \label{fig:FTVaR-old}
    \end{subfigure}
    \caption{The distribution of the gapped random variable $Z$.}
    \label{figure-gap-1}
\end{figure}
is strictly increasing at every point $z\in (0,2)\setminus [1-a,1+a]$ and is constant on the interval $ [1-a,1+a] $. The quantile $H^{-1}(1/2)$, which is the median of $Z$, is the left-hand  endpoint $1-a$ of the gap of $Z$. Consequently, the quantile function $H^{-1}$ is not continuous at the point $p=1/2$. This helps us to understand the reason why in the case 
\begin{equation}\label{p05}
p={1\over 2}, 
\end{equation}
we have the following two non-convergence statements: first, non-consistency 
\begin{equation}
H_{n,\srs}^{-1}(1/2) \stackrel{\mathbb{P}}{\nrightarrow} H^{-1}(1/2) ,  
\label{h-fail-1}
\end{equation} 
and, second, asymptotic non-normality 
\begin{equation}\label{h-fail-2}
\Delta_n:={\sqrt{n}\over \sigma_{\one}(1/2)}\left(\int_{1/2}^1 H_{n,\srs}^{-1}(u)\dd u  - \int_{1/2}^1 H^{-1}(u)\dd u  \right)
\stackrel{d}{\nrightarrow} \mathcal{N}(0,1), 
\end{equation}
where (recall equation~\eqref{normality-upper-var2})
\begin{equation}
\sigma_{\one}^2(1/2) 
=\int^{2}_{1-a} \int^{2}_{1-a} \big( H(x \wedge y) -H(x)H(y) \big) \dd x  \dd y  .  
\label{h-fail-2var}
\end{equation}
Although statements~\eqref{h-fail-1} and~\eqref{h-fail-2} can be verified theoretically, which we shall discuss in a moment, we find it illuminating to demonstrate the two statements empirically. For this, we set the gap radius to 
\[
a={1\over 2} .   
\]
Hence, the distribution of $Z$ has the gap $[1/2,3/2] $ in its support $ [0,2]$. When simulating we conveniently work with \textit{even} sample sizes $n$  and thus have the equation $\lceil n/2 \rceil =n/2$, which simplifies formulas of the empirical estimators, especially of the upper-layer integral. We proceed by noting that given the nature of the random variable $V$, which is symmetric around its mean $1$, and given the definition of the gapped $Z$, we expect to see approximately a half of the simulated $Z$ values below $1/2$ and the remaining ones above $3/2$. Indeed, when $n\to \infty$, we have the following two statements (see Appendix~\ref{app-z-conv}) 
\begin{gather} 
\mathbb{P}\left( Z_{\lceil n/2 \rceil :n} > 3/2\right) \to 1/2,  
\label{z-conv-32}
\\
\mathbb{P}\left( Z_{\lceil n/2 \rceil :n} < 1/2\right) \to 1/2.   
\label{z-conv-12}
\end{gather}
This explains non-consistency statement~\eqref{h-fail-1}, which we have visualized in Figure~\ref{figure-normality-1}  
\begin{figure}[h!]
  \centering
  \includegraphics[width=0.7\textwidth]{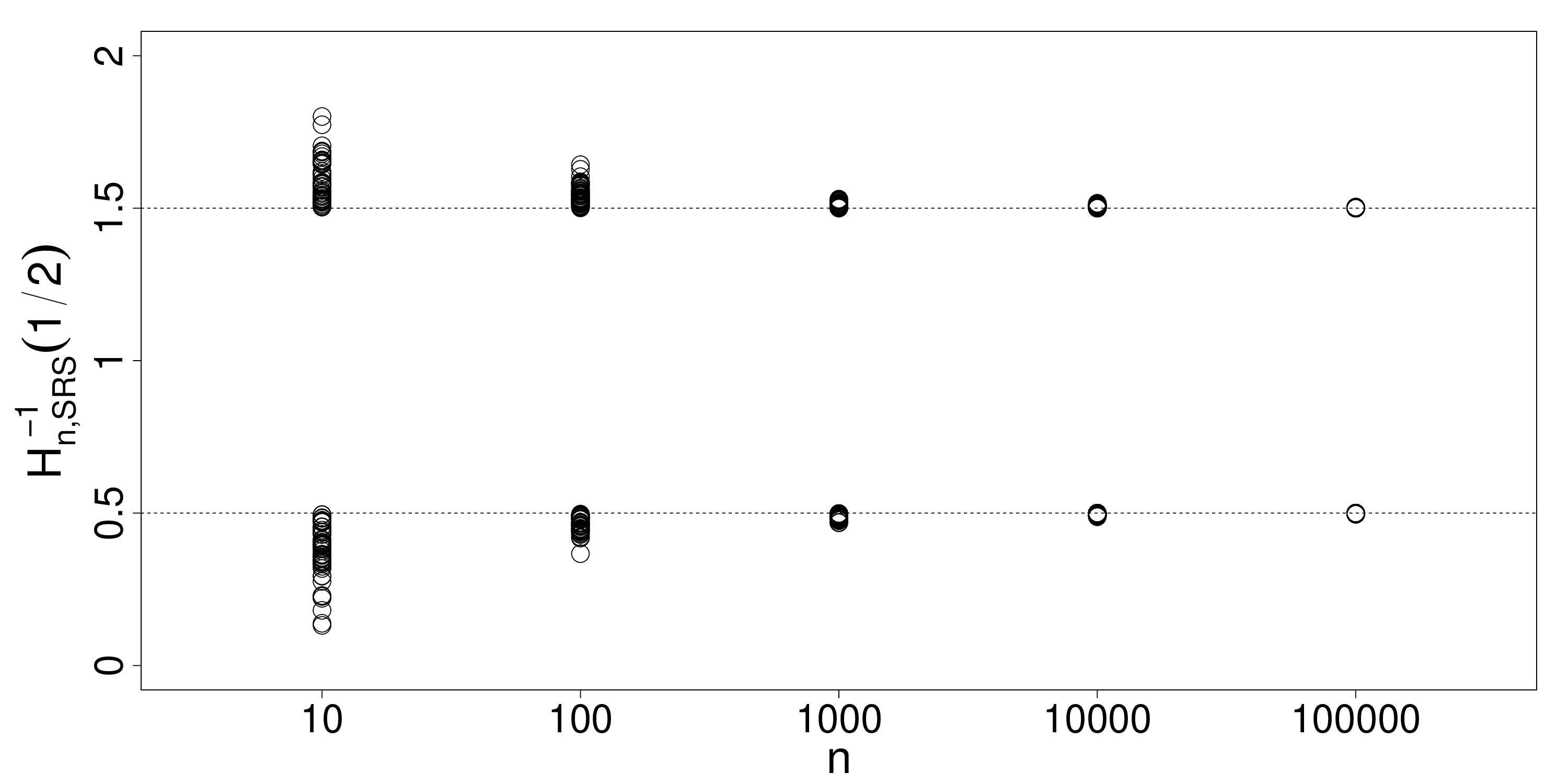}
  \caption{Simulated $m=100$ values of $H_{n,\srs}^{-1}(1/2)$ for various sample sizes  $n$.}
  \label{figure-normality-1}
\end{figure}
using the formulas 
\begin{align*}
H_{n,\srs}^{-1}(1/2) & =  Z_{\lceil n/2 \rceil :n} ~~ (=Z_{(n/2) :n}), 
\\
H^{-1}(1/2) & = 1/ 2   
\end{align*}
for various (even) sample sizes $n$. 
As to asymptotic non-normality statement~\eqref{h-fail-2}, we have used the formulas 
\begin{align}
\int_{1/2}^1 H_{n,\srs}^{-1}(u)\dd u  
&= {1\over n}\sum_{i=(n/2) +1}^{n} Z_{i :n} \quad \big(\textrm{because $\lceil n/2 \rceil =n/2$}\big), 
\notag 
\\
\int_{1/2}^1 H^{-1}(u)\dd u & = {7\over 8} ,  
\notag
\\
\sigma_{\one}^2(1/2) & = {77 \over 192}.   
\label{eq-variance}
\end{align}
Figure~\ref{figure-normality-2} 
\begin{figure}[h!]
  \centering
  \includegraphics[width=0.7\textwidth]{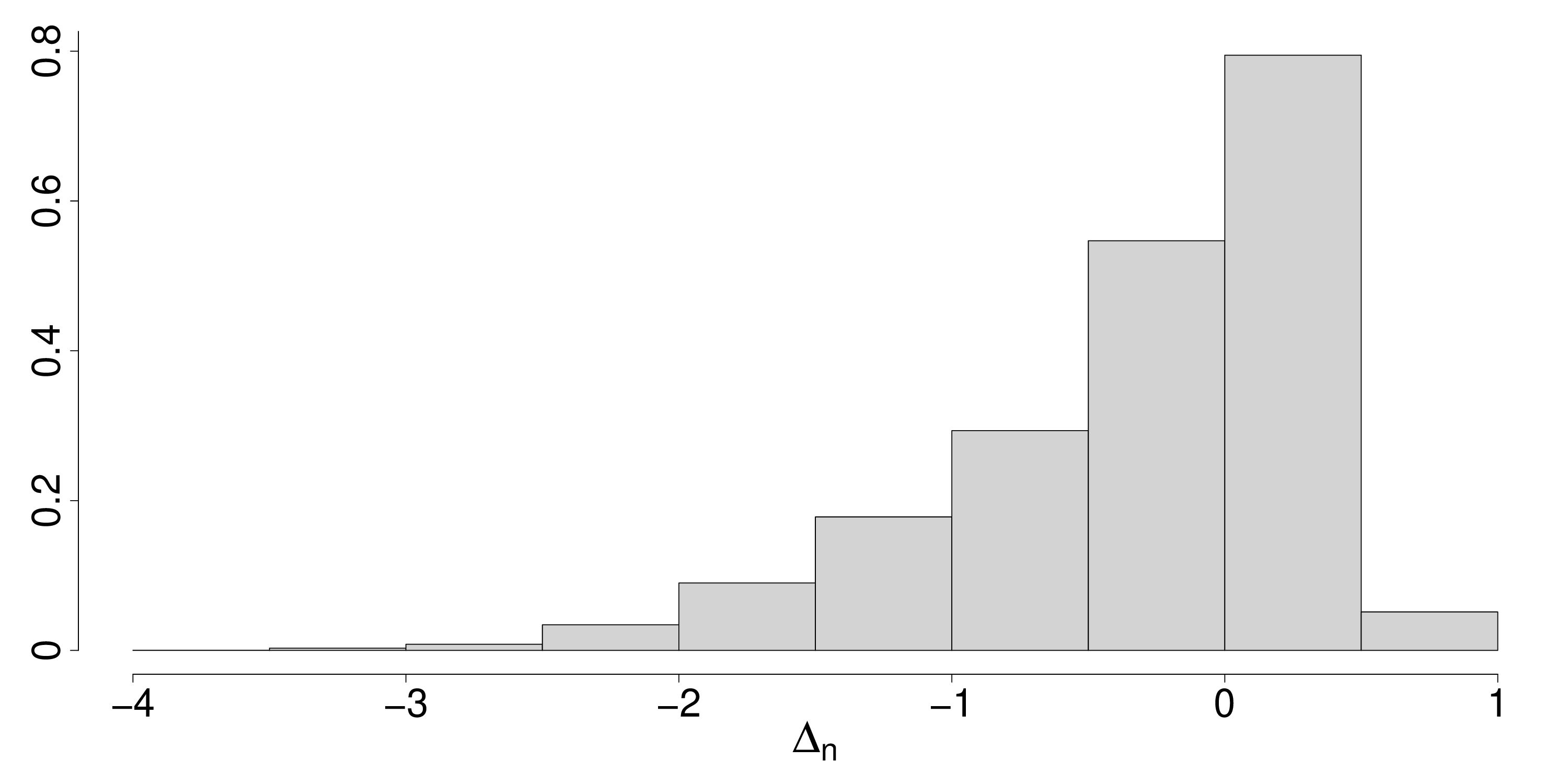}
  \caption{The relative histogram of $m=10\,000$ values of $\Delta_n$ when $n=100\,000$.}
  \label{figure-normality-2}
\end{figure}
depicts the simulation results based on the sample size $n=100\,000$ and the number of replications $m=10\,000$. Both $n$ and $m$ are very large, thus depriving us of any hope that a standard normal density could somehow emerge if even larger values were used. In fact, we know from theoretical studies on  trimmed means \citep{s1973} that  asymptotic normality does fail when (recall Note~\ref{note-quant}) continuous distributions have gaps (i.e., their cdf's have flat regions) adjacent to the quantiles where trimming occurs, as is the case in the present example (see Figure~\ref{figure-gap-1}).
\end{simulation}

\begin{note}
When calculating the variance $\sigma_{\one}^2(1/2)$ in equation~\eqref{eq-variance}, we found that the quantile-based expression (recall equation~\eqref{normality-upper-var2b}) 
\begin{align}
\sigma_{\one}^2(1/2)
&=\Var\left( \big( H^{-1}(U)-H^{-1}(1/2) \big)^{+} \right) 
\notag 
\\ 
&=\Var\left( \big( H^{-1}(U)-(1-a) \big)^{+} \right) 
\label{eq-variance-spec0}
\end{align} 
is a particularly convenient starting point, where $U$ denotes a uniform on $[0,1]$ random variable, with $1/2$ referring to our choice of $p$ made in~\eqref{p05}. In this way, we also easily arrive at the expression 
\begin{equation}
{29 a^2+14 a+5 \over 48} 
\label{eq-variance-spec}
\end{equation} 
for $\sigma_{\one}^2(1/2)$ in the case of arbitrary $a\in [0,1]$ (see Appendix~\ref{proof-eq-variance-spec} for details). Hence, $\sigma_{\one}^2(1/2)=5/48$ when $a=0$ (equation~\eqref{eq-variance-0}),  $\sigma_{\one}^2(1/2)=77/192$ when $a=1/2$ (equation~\eqref{eq-variance}), and $\sigma_{\one}^2(1/2)=1$ when $a=1$ (in which case $Z$ turns into a discrete random variable taking values $0$ and $2$ with probabilities $1/2$).     
Of course, we could have also calculated $\sigma_{\one}^2(1/2)$ using equation~\eqref{h-fail-2var} but this cdf-based route is more involved.   
\end{note}

\section{The lower-layer integral}
\label{sect-lower}

In this section we develop statistical inference for the lower-layer integral $\int_0^p F^{-1}(u)\dd u  $ for any fixed upper probability level 
\[
p\in (0,1). 
\]
Naturally, we assume that the cdf $F$ is from the class  $\mathcal{F}^{-}_{1}$ of those cdf's for which the integral $\int_0^p F^{-1}(u)\dd u  $ is finite, which is the same as to require that the negative part $X^{-}=(-X) \vee 0$ of the random variable $X\sim F$ has a finite first moment $\mathbb{E}(X^{-})<\infty $. 

\begin{note}  
When progressing through this section, we shall see that the narrative closely follows that of Section~\ref{sect-upper}. This may prompt us to think that the results could be derived from those of Section~\ref{sect-upper} by simply replacing $X$ by $-X$. This train of thought, though useful, is fraught with danger because $F^{-1}$ is not the ordinary inverse of $F$ and so subtleties do arise. It is our desire, therefore, to present the results in such a way that the user would be able to use them directly, without needing to work out the underlying theory.  
\end{note}

We continue our analysis of the lower-layer integral by noting that the functional 
\[
\mathcal{F}^{-}_{1} \ni F \mapsto \int_0^p F^{-1}(u)\dd u\in \mathbb{R}
\]
is not linear. It appears that when developing statistical inference it can be approximated by the linear one 
\begin{equation}\label{linear-lower}
\mathcal{F}^{-}_{1} \ni G \mapsto \int_{-\infty }^{F^{-1}(p)}  G(x) \dd x \in \mathbb{R} . 
\end{equation}
Indeed, we shall see from the following corollaries that the difference between the empirical lower-layer integral  $\int_0^p F_n^{-1}(u)\dd u$ and its population counterpart $\int_0^p F^{-1}(u)\dd u$ gets asymptotically close to the difference between the integrals $\int_{-\infty }^{F^{-1}(p)}  F(x) \dd x $ and  $\int_{-\infty }^{F^{-1}(p)}  F_n(x) \dd x$, where $F_n \in \mathcal{F}^{-}_{1}$, $n \in \mathbb{N}$, are cdf's approaching $F$ when the parameter $n$ grows indefinitely. The corollaries rely on the following theorem, whose proof is given in Appendix~\ref{appendix-2}.

\begin{theorem}\label{theorem-0lower}
Let $F$ and $G$ be any two cdf's from $\mathcal{F}^{-}_{1}$. Then 
\begin{equation}\label{eq-1lower}
\int_0^p \big( G^{-1}(u)-F^{-1}(u) \big) \dd u 
=\int_{-\infty }^{F^{-1}(p)}  \big( F(x)- G(x)\big) \dd x + \Rem(p;F,G), 
\end{equation}
where the non-negative remainder term $\Rem(p;F,G)$ is defined by equation~\eqref{qq-3upper} and satisfies bounds~\eqref{r-prop-0upper} and~\eqref{r-prop-1upper}. 
\end{theorem}

Theorem~\ref{theorem-0lower} is fundamental for the current section, and it also serves a direct link between the topic of the current paper and the Vervaat process $V_{n}$ mentioned in Section~\ref{intro}. The following note elucidates the connection in detail.

\begin{note} \label{note-vervaat} 
The (general) Vervaat process $V_{n}$ is defined on the unit interval $(0,1)$ by the equation 
\begin{align*}
V_{n}(p) 
:=& \int_0^p \big( F_{n}^{-1}(u)-F^{-1}(u) \big) \dd u 
+\int_{-\infty }^{F^{-1}(p)}  \big( F_{n}(x)- F(x)\big) \dd x 
\\
=& \Rem(p;F,F_{n}) , 
\end{align*}
where $F_{n}$, which replaces $G$ in equation~\eqref{eq-1lower}, is the empirical cdf based on $n$ random variables $X_1,\dots , X_n $, each following the cdf $F$. Originally, the process $V_n$ was introduced and considered by \citet{cs1996} in the case of iid random variables, with a number of subsequent studies by various authors devoted to asymptotic properties of $V_n$ in the case of independent as well as dependent random variables, such as those generated by time series models. As to the pioneering studies of \citet{v1972a,v1972b}, they concentrated on the process $V_n$ in the case of $n$ independent and uniformly on $[0,1]$ distributed random variables $U_1,\dots , U_n$, in which case $V_n$ reduces to  
\begin{align*}
V_{n,\srs}^U(p)
:=&\int_0^p \underbrace{\big( E_{n,\srs}^{-1}(u)- u \big)}_{\textrm{uniform quantile process}} \dd u 
~ + ~ \int_{0}^{p}  \underbrace{\big( E_{n,\srs}(u)- u\big)}_{\textrm{uniform empirical process}} \dd u   
\\
=&  \int_0^p \Big( \underbrace{\big( E_{n,\srs}^{-1}(u)- u \big) 
+ \big( E_{n,\srs}(u)- u\big)}_{\textrm{uniform Bahadur-Kiefer process}}\Big) \dd u  , 
\end{align*}
where $E_{n,\srs}$ is the empirical cdf based on $U_1,\dots , U_n$, and $E_{n,\srs}^{-1}$ is the corresponding empirical quantile function. We infer from Theorem~\ref{theorem-0lower} that $V_{n,\srs}^U(p)\ge 0$ for all $p\in [0,1]$, with the boundary values $V_{n,\srs}^U(0)=0$ and $V_{n,\srs}^U(1)=0$. Sample paths of the uniform Vervaat process $V_{n,\srs}^U(p)$, $0\le p \le 1$, and those of the uniform empirical process $E_{n,\srs}(p)- p$, $0\le p \le 1$, are illustrated in  Figure~\ref{figure-vervaat-empirical}. 
\begin{figure}[h!]
    \centering
    \begin{subfigure}[b]{0.41\textwidth}
        \includegraphics[width=\textwidth]{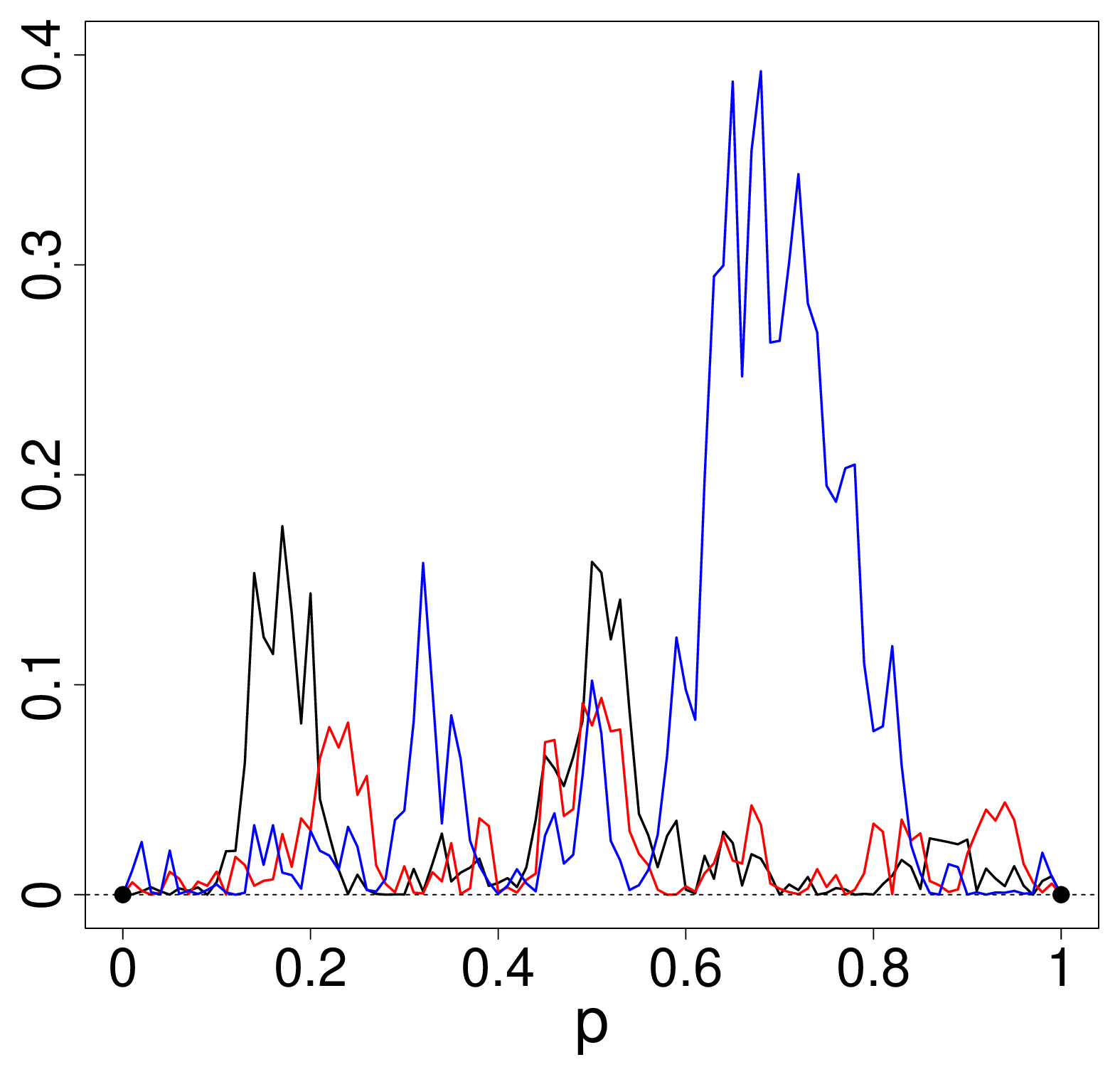}
        \caption{Vervaat process $n\,V_{n,\srs}^U(p)$.}
        \label{fig:vervaat}
    \end{subfigure}
\qquad
    \begin{subfigure}[b]{0.41\textwidth}
        \includegraphics[width=\textwidth]{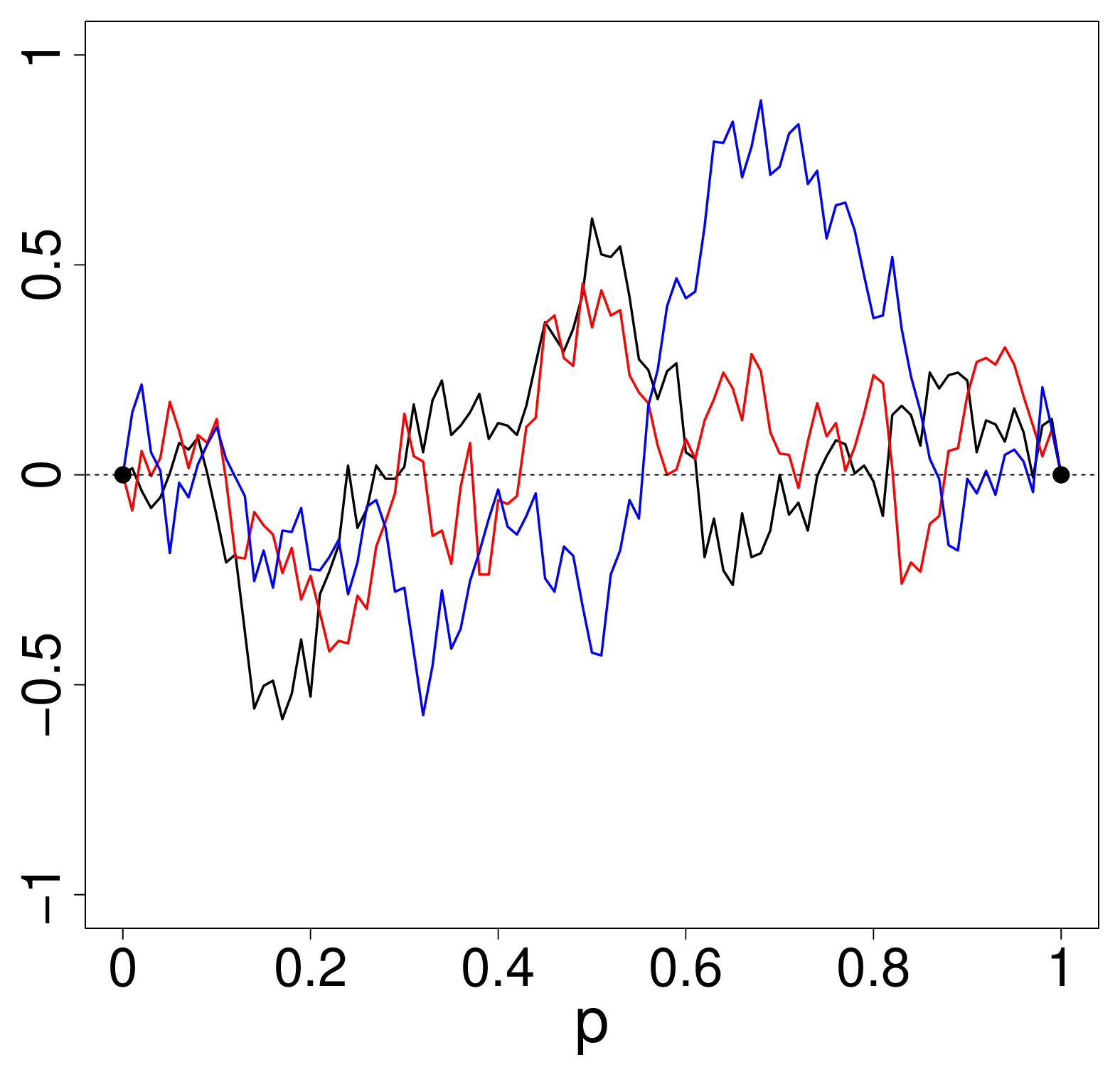}
        \caption{Empirical process $n^{1/2}(E_{n,\srs}(p)- p)$.}
        \label{fig:empirical}
    \end{subfigure}
    \caption{Sample paths of normalized uniform Vervaat and empirical processes based on three uniform-on-$[0,1]$ SRS's of sizes $n=100\,000$.}
    \label{figure-vervaat-empirical}
\end{figure}
(We shall discuss computational formulas and normalizing constants of these two stochastic processes at the end of this section; see Computation~\ref{computation-vervaat}.)  
Mathematically, $V_{n,\srs}^U(p)$ is the lower-layer integral of the uniform version of the Bahadur-Kiefer process \citep{b1966,k1967}. In the case of general sequences $X_1,\dots , X_n \sim F$ of iid random variables, $V_{n,\srs}(p)$  is the lower-layer integral of the (general) Bahadur-Kiefer process, as seen from the equations 
\begin{align*}
V_{n,\srs}(p) 
&= \int_0^p \big( F_{n,\srs}^{-1}(u)-F^{-1}(u) \big) \dd u  
+\int_{-\infty }^{F^{-1}(p)}  \big( F_{n,\srs}(x)- F(x)\big) \dd x
\\
&= \int_0^p \bigg ( \underbrace{\big( F_{n,\srs}^{-1}(u)-F^{-1}(u) \big) 
+ {1\over f(F^{-1}(u)) }  \big( F_{n,\srs}(F^{-1}(u))- u\big)}_{\textrm{Bahadur-Kiefer process}} \bigg) \dd u  , 
\end{align*}
but this interpretation necessarily requires the pdf $f=F'$ to exist and be positive at all quantiles $F^{-1}(u)$, $0<u<p$, as only in this case the (general) Bahadur-Kiefer process is meaningfully defined.  
\end{note}

We are now ready to formulate and discuss several corollaries to Theorem~\ref{theorem-0lower}. 

\begin{corollary}[consistency]\label{corollary-1lower} 
Suppose that $F\in \mathcal{F}^{-}_{1}$, and let $F_1,F_2,\ldots \in \mathcal{F}^{-}_{1}$ be any sequence of cdf's satisfying condition~\eqref{cond-0upper}. Then 
\begin{equation}\label{eq-1a}
\int_0^p F_n^{-1}(u) \dd u - \int_0^p F^{-1}(u) \dd u 
=\int_{-\infty }^{F^{-1}(p)}  \big( F(x)- F_n(x)\big) \dd x +o_{\mathbb{P}}(1) 
\end{equation}
and, therefore, the consistency statement 
\begin{equation}\label{eq-1a-a}
\int_0^p F_n^{-1}(u) \dd u \stackrel{\mathbb{P}}{\to} \int_0^p F^{-1}(u) \dd u 
\end{equation}
holds if and only if 
\begin{equation}\label{eq-1a-b}
\int_{-\infty }^{F^{-1}(p)}  F_n(x) \dd x \stackrel{\mathbb{P}}{\to} \int_{-\infty }^{F^{-1}(p)}F(x) \dd x .  
\end{equation}
\end{corollary}

\begin{proof} 
The corollary follows from Theorem~\ref{theorem-0lower} and bound~\eqref{r-prop-1upper} with $F_n$ instead of $G$, because $F^{-1}(p)$ is finite and $F_n^{-1}(p)$ is asymptotically bounded, and therefore condition~\eqref{cond-0upper} implies $\Rem(p;F,G)=o_{\mathbb{P}}(1)$ due to bound~\eqref{r-prop-1upper}. This establishes statement~\eqref{eq-1a}, which implies the equivalence of statements~\eqref{eq-1a-a} and~\eqref{eq-1a-b}, and concludes the proof of Corollary~\ref{corollary-1lower}. 
\end{proof}

\begin{example}[SRS] \label{example-lower1}
Condition~\eqref{cond-0upper} with $F_{n,\srs}$ instead of $F_n$ is satisfied by the Glivenko-Cantelli theorem. Furthermore, by linearity of functional~\eqref{linear-lower}, the integral $\int_{-\infty }^{F^{-1}(p)}  F_n(x) \dd x$ is the arithmetic mean $n^{-1}\sum_{i=1}^n Y_{i,\zero}$ of $n$ independent copies of the random variable 
\begin{align}\label{rv-lower}
Y_{\zero}(p)
&=\int_{-\infty }^{F^{-1}(p)}  \mathds{1}\{X\le x\} \dd x   
\notag 
\\
&=\big( F^{-1}(p)-X \big)^{+}. 
\end{align}
The random variable $Y_{\zero}(p)$ has a finite first moment because $F\in \mathcal{F}^{-}_{1}$  (see Lemma~\ref{lemma-1lower} for details), and so statement~\eqref{eq-1a-b} with $F_{n,\srs}$ instead of $F_n$ holds. By Corollary~\ref{corollary-1lower}, therefore, the empirical lower-level integral $\int_0^p F_{n,\srs}^{-1}(u)\dd u $ is a consistent estimator of $\int_0^p F^{-1}(u)\dd u $. In summary, when $F\in \mathcal{F}^{-}_{1}$, we have 
\begin{equation}\label{consistency-lower}
\int_0^p F_{n,\srs}^{-1}(u)\dd u  \stackrel{\mathbb{P}}{\to} \int_0^p F^{-1}(u)\dd u . 
\end{equation}
\end{example}

\begin{note}[dependent data]
Although we work with the case $p\in (0,1)$, let us set $p=1$ for a moment. In this case, statement~\eqref{consistency-lower} turns into the law of large numbers 
\begin{equation}\label{mean-1}
\mu_{n,\srs}=\int_0^1 F_{n,\srs}^{-1}(u)\dd u\stackrel{\mathbb{P}}{\to} \mu  , 
\end{equation}
which of course holds under SRS. A departure from this sampling design leads to ergodic theory, and in particular to sequences $X_1,X_2,\dots $ that are stationary and ergodic. Hence, it is only natural that results of the type of statement~\eqref{consistency-lower} when $p\in (0,1)$ have been established within the class of stationary and ergodic sequences \citep[e.g.,][and references therein]{abdghw1996,gh1997,dz2002,hz2005}. It is also useful to point out that while the latter two references deal with the integral $\int_0^p F^{-1}(u)\dd u $, the first two references deal with the more general integral $\int_0^1 F^{-1}(u)w(u)\dd u $, which we shall discuss in more detail in  Section~\ref{sect-L-functional}. 
\end{note}

\begin{corollary}[bias]\label{corollary-2lower} 
Suppose that  $F\in \mathcal{F}^{-}_{1}$, and let $F_1,F_2,\ldots \in \mathcal{F}^{-}_{1}$ be any sequence of cdf's that are unbiased estimators of $F$, that is, satisfy  condition~\eqref{cond-0bupper}, and let these cdf's also be such that condition~\eqref{cond-0bupper-bias} is satisfied.  Then $\int_0^p F_n^{-1}(u)\dd u $ is a non-negatively biased estimator of $\int_0^p F^{-1}(u)\dd u $. 
\end{corollary}

\begin{proof} 
Using Theorem~\ref{theorem-0lower} with $F_n$ instead of $G$, and also recalling that the remainder term $\Rem(p;F,G)$ is non-negative, we have 
\begin{align*}
\mathrm{Bias}_n^{\zero}(p):= &\mathbb{E}\bigg(\int_0^p F_n^{-1}(u)\dd u\bigg)  -\int_0^p F^{-1}(u)\dd u 
\\
= &  \mathbb{E}\big(\Rem(p;F,F_n)\big) 
\\
\in & [0,\infty )  
\end{align*}
for every $n\in \mathbb{N}$.  
Note that $\mathbb{E}\big(\Rem(p;F,F_n)\big) $ is finite because of 
bound~\eqref{r-prop-1upper} and condition~\eqref{cond-0bupper-bias}. 
This establishes Corollary~\ref{corollary-2lower}. 
\end{proof}

\begin{example}[SRS] \label{example-lower2}
When $F_n$ is $ F_{n,\srs}$, condition~\eqref{cond-0bupper} is satisfied. Furthermore, condition~\eqref{cond-0bupper-bias} is satisfied when $F\in \mathcal{T}_{\varepsilon}$  for some $\varepsilon>0$ (recall Example~\ref{example-upper2} for details). Hence, we conclude from  Corollary~\ref{corollary-2lower} that when $F\in \mathcal{F}^{-}_{1} \cap \mathcal{T}_{\varepsilon}$, the empirical estimator $\int_0^p F_{n,\srs}^{-1}(u)\dd u  $ of the lower-layer integral $ \int_0^p F^{-1}(u)\dd u $, although being consistent by Example~\ref{example-lower1}, has the positive bias 
\begin{equation}
\mathrm{Bias}_{n,\srs}^{\zero}(p):= \mathbb{E}\big(\Rem(p;F,F_{n,\srs})\big) \in [0,\infty ) . 
\label{bias-lower-srs}
\end{equation}
\end{example}

\begin{note}\label{note-epsilon}
The condition that $F\in \mathcal{F}^{-}_{1} \cap \mathcal{T}_{\varepsilon}$  for some $\varepsilon>0$ can equivalently be reformulated as  $F\in \mathcal{F}^{-}_{1} \cap \mathcal{T}^{+}_{\varepsilon}$  for some $\varepsilon>0$, where $ \mathcal{T}^{+}_{\varepsilon}$ consists of all those cdf's that satisfy $x^{\varepsilon}(1-F(x))\to 0$ when $x\to \infty $. This reformulation is possible due to the fact that  $F\in \mathcal{F}^{-}_{1} $ implies $|x|^{\varepsilon}F(x)\to 0$ when $x\to -\infty $ for every $\varepsilon \in (0,1)$, and we need this condition to hold for just one $\varepsilon>0$. For those researchers who think in terms of moments, a sufficient condition that ensures the existence of $\varepsilon>0$ such that $F\in \mathcal{F}^{-}_{1} \cap \mathcal{T}^{+}_{\varepsilon}$ would be the requirement that 
\begin{equation}
\mathbb{E}(X^{-})<\infty \quad \text{and} \quad \mathbb{E}((X^{+})^{\varepsilon})<\infty 
\label{req-lower}
\end{equation}
for some $\varepsilon >0$, where $X \sim F$. Note that requirement~\eqref{req-lower} can equivalently be rewritten as   $F\in \mathcal{F}^{-}_{1} \cap \mathcal{F}^{+}_{\varepsilon}$. 
\end{note}

\begin{note}\label{example-lower2a} 
Let $F\in \mathcal{F}_{1}$, where $ \mathcal{F}_{1}:=\mathcal{F}^{-}_{1} \cap \mathcal{F}^{+}_{1} $. Hence, $F\in \mathcal{F}_{1}$ is equivalent to saying that $\int_0^1 F^{-1}(u)\dd u\in \mathbb{R}$, which is the same as saying that the random variable $X\sim F$ has a finite first moment $\mathbb{E}(X)\in \mathbb{R}$. When in addition to this moment requirement, condition~\eqref{cond-0bupper} is also satisfied, we have the equation 
\[
\mathrm{Bias}_{n}^{\zero}(p)=- \mathrm{Bias}_{n}^{\one}(p)
\]
because, due to equation~\eqref{eq-0} with $F_n$ instead of $G$,  
\begin{align*}
\mathrm{Bias}_{n}^{\one}(p)+\mathrm{Bias}_{n}^{\zero}(p)
&=\mathbb{E}\bigg(\int_0^1 F_n^{-1}(u)\dd u -\int_0^1 F^{-1}(u)\dd u \bigg) 
\\
& = \mathbb{E}\bigg(\int^{\infty }_{-\infty}  \big( F(x)- F_n(x)\big) \dd x \bigg) 
\\
&=\int^{\infty }_{-\infty}  \mathbb{E}\big( F(x)- F_n(x)\big) \dd x ,   
\end{align*}
where the right-hand side is equal to $0$ due to condition~\eqref{cond-0bupper}. 
\end{note}

\begin{corollary}\label{corollary-3lower} 
Suppose that  $F\in \mathcal{F}^{-}_{1}$, and let $F_1,F_2,\ldots \in \mathcal{F}^{-}_{1}$ be any sequence of cdf's satisfying condition~\eqref{cond-1upper} and also condition~\eqref{cond-2upper} with some normalizing constants $A_n\to \infty $ when $n\to \infty $. Then 
\begin{equation}\label{eq-1da}
A_n\left( \int_0^p F_n^{-1}(u) \dd u - \int_0^p F^{-1}(u) \dd u \right) 
=A_n\int_{-\infty }^{F^{-1}(p)}  \big( F(x)- F_n(x)\big) \dd x +o_{\mathbb{P}}(1)  
\end{equation}
and, therefore, the convergence-in-distribution statement 
\begin{equation}\label{eq-1da-a}
A_n\left( \int_0^p F_n^{-1}(u) \dd u - \int_0^p F^{-1}(u) \dd u \right) 
\stackrel{d}{\to}  \mathcal{L}_{\zero}(p)
\end{equation}
holds if and only if 
\begin{equation}\label{eq-1da-b}
A_n\left( \int_{-\infty }^{F^{-1}(p)}  F_n(x) \dd x - \int_{-\infty }^{F^{-1}(p)}  F(x) \dd x \right) 
\stackrel{d}{\to}  - \mathcal{L}_{\zero}(p) , 
\end{equation}
where $\mathcal{L}_{\zero}(p)$ is a random variable determined by statement~\eqref{eq-1da-b}.  
\end{corollary}

\begin{proof} 
Using  Theorem~\ref{theorem-0lower} with $F_n$ instead of $G$, we only need to show that $A_n\Rem(p;F,F_n)$ converges in probability to $0$, which we showed in the proof of Corollary~\ref{corollary-3upper}.  This establishes equation~\eqref{eq-1da}, which implies the equivalence of statements~\eqref{eq-1da-a} and~\eqref{eq-1da-b}, and concludes the proof of Corollary~\ref{corollary-3lower}. 
\end{proof}

\begin{example}[SRS] \label{example-lower3} 
When $F_n$ is $F_{n,\srs}$, condition~\eqref{cond-2upper}, which is a requirement in Corollary~\ref{corollary-3lower}, is satisfied by the Kolmogorov-Smirnov theorem. Hence, condition~\eqref{cond-1upper} is satisfied as well, provided that, according to Lemma~\ref{lemma-1}, the quantile function $F^{-1}$ is continuous at the point $p$. 
Furthermore, since the integral $\int_{-\infty }^{F^{-1}(p)}  F_n(x) \dd x$ is the arithmetic mean of $n$ independent copies of random variable~\eqref{rv-lower}, 
we conclude that $\int_{-\infty }^{F^{-1}(p)}  F_n(x) \dd x$ satisfies the central limit theorem provided that  random variable~\eqref{rv-lower} has a finite second moment. This is so provided that $F\in \mathcal{F}^{-}_{2}$ (see Lemma~\ref{lemma-1lower} for details), where $\mathcal{F}^{-}_{2}$ denotes the set of all cdf's such that the integral $\int_0^p \big(F^{-1}(u)\big)^2\dd u $ is finite. 
Note that $F\in \mathcal{F}^{-}_{2}$ is equivalent to saying that the negative part $X^{-}$ of the random variable $X\sim F$ has a finite second moment $\mathbb{E}((X^{-})^2)<\infty $. 
In summary, when $F\in \mathcal{F}^{-}_{2}$ and the quantile function $F^{-1}$ is continuous at the point $p$, 
we have the asymptotic normality result 
\begin{equation}\label{normality-lower}
\sqrt{n}\left(\int_0^p F_{n,\srs}^{-1}(u)\dd u  - \int_0^p F^{-1}(u)\dd u \right) 
\stackrel{d}{\to} \mathcal{N}(0,\sigma_{\zero}^2(p)) , 
\end{equation}
where $\sigma_{\zero}^2(p)$ is the variance of the random variable $Y_{\zero}(p)$. The variance can be expressed by the formula 
\begin{equation}
\sigma_{\zero}^2(p)
=\int_{-\infty }^{F^{-1}(p)} \int_{-\infty }^{F^{-1}(p)} 
 \big( F(x \wedge y) -F(x)F(y) \big) \dd x  \dd y. 
\label{normality-lower-var2}
\end{equation}
\end{example}

\begin{note}[dependent data]
Limiting distributions extending statement~\eqref{normality-lower} to classes of dependent random sequences were studied by \citet{dz2003}, \citet{dz2004}, and \citet{dksz2007}. 
The knowledge of empirical processes and their asymptotic behaviour based on such random sequences becomes  particulary useful, and we refer to \citet{DMS2002} for details on the topic. Due to space considerations, we only note here that dependence structures affect asymptotic variances, normalizing constants, and even limiting distributions. For a bird's-eye view  of the variety of normalizing constants and limiting distributions in the case of lower-layer integrals (also known as convexifications, absolute Lorenz curves, and by some other names), we refer to \citet{dz2004}. 
\end{note}

\begin{computation}[SRS] \label{computation-lower} 
Using equation~\eqref{quantile-upper-2} for calculating the quantile $F_{n,\srs}^{-1}(u)$, we have 
\begin{align}
\int_0^p F_{n,\srs}^{-1}(u)\dd u
&= \sum_{i=1}^{\lceil np \rceil } \int_{(i-1)/n}^{i/n}F_{n,\srs}^{-1}(u)\dd u 
-\int_{p}^{\lceil np \rceil /n} F_{n,\srs}^{-1}(u)\dd u 
\notag 
\\
&= \underbrace{{1\over n}\sum_{i=1}^{\lceil np \rceil } X_{i :n}}_{\textrm{main term}} 
-\underbrace{ \vphantom{\sum_{i=\lceil np \rceil +1}^{n}} \left( {\lceil np \rceil \over n}- p\right) X_{\lceil np\rceil :n}}_{\textrm{asymptotically negligible term}} 
\label{quantile-lower-3xyz}
\\
&= {\lceil np \rceil \over n}  
\underbrace{{1\over \lceil np \rceil }\sum_{i=1}^{\lceil np \rceil } X_{i :n}}_{\textrm{right-trimmed mean}} 
-\underbrace{ \vphantom{\sum_{i=\lceil np \rceil +1}^{n}} \left( {\lceil np \rceil \over n}- p\right) X_{\lceil np\rceil :n}}_{\textrm{asymptotically negligible term}}.  
\label{quantile-lower-3ii}
\end{align}
These expressions give computationally-friendly formulas for the empirical lower-layer integral, which is connected to  the right-trimmed mean \citep[e.g.,][]{s1973} via equation~\eqref{quantile-lower-3ii}. 
\end{computation}

\begin{computation}[Uniform Vervaat process] \label{computation-vervaat} 
Using formulas~\eqref{rv-lower} and~\eqref{quantile-lower-3xyz} in the case of independent and uniformly on $[0,1]$ distributed random variables $U_1,\dots , U_n$, we obtain the formulas 
\begin{align*}
V_{n,\srs}^U(p)
&=\int_0^p E_{n,\srs}^{-1}(u) \dd u 
 + {1\over n}\sum_{i=1}^n (p-U_i)^{+} - p^2 
\\
&={1\over n}\sum_{i=1}^{\lceil np \rceil } U_{i :n}
-\left( {\lceil np \rceil \over n}- p\right) U_{\lceil np\rceil :n}
 + {1\over n}\sum_{i=1}^n (p-U_i)^{+} - p^2  
\end{align*}
that we used to visualize the normalized Vervaat process 
\[
n\,V_{n,\srs}^U(p), \quad 0\le p \le 1,  
\]
in Figure~\ref{fig:vervaat}. The appropriateness of the normalization $n$ for the process $V_n^U$ was established by \citet{v1972a,v1972b}, who showed that $n\,V_{n,\srs}^U$ converges weakly to a half of the squared Brownian bridge, that is, to $B^2/2$, when $n\to \infty $. On the other hand, the normalized uniform empirical process 
\[
n^{1/2}(E_{n,\srs}(p)- p), \quad 0\le p \le 1,  
\]
which we have visualized in Figure~\ref{fig:empirical}, converges weakly to the Brownian bridge $B$ when $n\to \infty $ \citep[e.g.,][]{b1999}. Hence, what we essentially see in the two panels of Figure~\ref{figure-vervaat-empirical} are sample paths of a half of the squared Brownian bridge $B^2/2$ (left-hand panel) and of the classical Brownian bridge $B$ (right-hand panel). For additional insights into these and related results, we refer to  \citet{cz1999,cs2001}, and \citet{ccfsz2002}. 
\end{computation}

\section{The middle-layer integral}
\label{sect-middle}

Let $\mathcal{F}$ denote the class of all cdf's. 
Fix any pair of probability levels $p_1$ and $p_2$ such that 
\[
0<p_1<p_2<1. 
\]
The middle-layer integral $\int_{p_1}^{p_2} F^{-1}(u)\dd u $ is finite for every $F\in \mathcal{F}$. The functional 
\[
\mathcal{F} \ni F \mapsto \int_{p_1}^{p_2} F^{-1}(u)\dd u \in \mathbb{R}
\]
is not linear, but when developing statistical inference it can be approximated by the linear one
\begin{equation}\label{linear-middle}
\mathcal{F} \ni G \mapsto \int_{F^{-1}(p_1)}^{F^{-1}(p_2)}  G(x) \dd x \in \mathbb{R} . 
\end{equation}
Indeed, we shall see 
from the following corollaries that the difference between the empirical middle-layer integral  $\int_{p_1}^{p_2} F_n^{-1}(u)\dd u$ and its population counterpart $\int_{p_1}^{p_2} F^{-1}(u)\dd u$ gets asymptotically close to the difference between the integrals $\int_{F^{-1}(p_1)}^{F^{-1}(p_2)}  F(x) \dd x $ and  $\int_{F^{-1}(p_1)}^{F^{-1}(p_2)}  F_n(x) \dd x$, where $F_n $, $n\in \mathbb{N}$, are cdf's approaching $F$ when the parameter $n$ grows indefinitely. The corollaries rely on the following theorem, whose proof is given in Appendix~\ref{appendix-2}.

\begin{theorem}\label{theorem-0middle}
Let $F$ and $G$ be any two cdf's. Then 
\begin{equation}\label{eq-1middle}
\int_{p_1}^{p_2} \big( G^{-1}(u)-F^{-1}(u) \big) \dd u 
=\int_{F^{-1}(p_1)}^{F^{-1}(p_2)}  \big( F(x)- G(x)\big) \dd x + \Rem(p_2;F,G)-\Rem(p_1;F,G), 
\end{equation}
where the non-negative remainder terms $\Rem(p_1;F,G)$  and $\Rem(p_2;F,G)$ are defined by equation~\eqref{qq-3upper} and satisfy bounds~\eqref{r-prop-0upper} and~\eqref{r-prop-1upper}. 
\end{theorem}

\begin{note} \label{note-4}
It is tempting to collapse  equations~\eqref{eq-0}, \eqref{eq-1upper}, \eqref{eq-1lower}, and~\eqref{eq-1middle} into one by relaxing the restriction $0<p_1<p_2<1$ to $0\le p_1<p_2 \le 1$. This is indeed possible by augmenting the definition of the remainder term $\Rem(p;F,G)$, which has so far been given only for $p\in (0,1)$, by setting $\Rem(p;F,G)$ to $0$ when $p\in \{0,1\}$. If we agree with this augmentation, then it is also imperative to replace $F^{-1}(p_1)$ by $-\infty $ when $p_1=0$ in the integral  
\begin{equation}\label{eq-1middle-extra}
\int_{F^{-1}(p_1)}^{F^{-1}(p_2)}  \big( F(x)- G(x)\big) \dd x
\end{equation}
on the right-hand side of equation~\eqref{eq-1middle}, and to also replace $F^{-1}(p_2)$ by $\infty $ when $p_2=1$. Note that  $F^{-1}(0)$ may or may not be equal to $-\infty $, and $F^{-1}(1)$ may or may not be equal to $\infty $. Hence, although the noted replacements may look artificial, they are absolutely necessary, as next Example~\ref{ex-counter} shows. 
Consequently, the role of infinities $-\infty $ and $\infty $ instead of $F^{-1}(p_1)$ and $F^{-1}(p_2)$  in integral~\eqref{eq-1middle-extra} when $p_1=0$ and $p_2=1$, respectively, is crucial. In summary, we recommend having four separate equations~\eqref{eq-0}, \eqref{eq-1upper}, \eqref{eq-1lower}, and~\eqref{eq-1middle} in the toolbox, instead of having just one general equation with a number of caveats and adjustments -- the four separate equations should help to avoid potential overlooks and pitfalls. 
\end{note}

\begin{example}\label{ex-counter}
Let $F\sim U_{[0,1]}$ and $G\sim U_{[0,2]}$ be two random variables with uniform on the intervals $[0,1]$ and $[0,2]$ distributions, respectively. That is, $F(x)=x$ when $0\le x \le 1$ and  $G(x)=x/2$ when $0\le x \le 2$. Consequently, $F^{-1}(t)=t$ and  $G^{-1}(t)=2t$ when $0\le t \le 1$. From these formulas, we obtain 
\begin{gather*}
\int_{0}^{1}  \big( G^{-1}(t) - F^{-1}(t)\big) \dd t = {1\over 2}, 
\\
\int_{F^{-1}(0)}^{F^{-1}(1)}  \big( F(x) - G(x)\big) \dd x = {1\over 4} , 
\\
\int_{-\infty}^{\infty}  \big( F(x) - G(x)\big) \dd x = {1\over 2} . 
\end{gather*}
Hence, 
\[
\int_{0}^{1}  \big( G^{-1}(t) - F^{-1}(t)\big) \dd t  
= \int_{-\infty}^{\infty}  \big( F(x) - G(x)\big) \dd x,  
\]
which is what equation~\eqref{eq-0} says, but 
\[
\int_{0}^{1}  \big( G^{-1}(t) - F^{-1}(t)\big) \dd t  
\neq \int_{F^{-1}(0)}^{F^{-1}(1)}  \big( F(x) - G(x)\big) \dd x  . 
\]
This concludes Example~\ref{ex-counter}. 
\end{example}

We are now ready to formulate and discuss several corollaries to Theorem~\ref{theorem-0middle}.

\begin{corollary}[consistency]\label{corollary-1middle} 
Let $F$ be any cdf, and let $F_1,F_2,\ldots $ be any sequence of cdf's satisfying condition~\eqref{cond-0upper}. Then 
\begin{equation}\label{eq-1b}
\int_{p_1}^{p_2} F_n^{-1}(u)\dd u - \int_{p_1}^{p_2}F^{-1}(u)  \dd u 
=\int_{F^{-1}(p_1)}^{F^{-1}(p_2)}  \big( F(x)- F_n(x)\big) \dd x +o_{\mathbb{P}}(1) 
\end{equation} 
and, therefore, the consistency statement 
\begin{equation}\label{eq-1b-a}
\int_{p_1}^{p_2} F_n^{-1}(u) \dd u \stackrel{\mathbb{P}}{\to} \int_{p_1}^{p_2} F^{-1}(u)\dd u 
\end{equation} 
holds if and only if 
\begin{equation}\label{eq-1b-b}
\int_{F^{-1}(p_1)}^{F^{-1}(p_2)}  F_n(x) \dd x \stackrel{\mathbb{P}}{\to} \int_{F^{-1}(p_1)}^{F^{-1}(p_2)}F(x) \dd x . 
\end{equation} 
\end{corollary}

\begin{proof} 
The corollary follows from Theorem~\ref{theorem-0middle} and bound~\eqref{r-prop-1upper} with $F_n$ instead of $G$, because for $p\in \{p_1,p_2\} \subset (0,1)$, $F^{-1}(p)$ is finite and $F_n^{-1}(p)$ is asymptotically bounded. Therefore  condition~\eqref{cond-0upper} implies $\Rem(p;F,F_n)=o_{\mathbb{P}}(1)$ for both $p=p_1$ and $p=p_2$, and thus, in turn, implies  statement~\eqref{eq-1b}.  This establishes the equivalence of statements~\eqref{eq-1b-a} and~\eqref{eq-1b-b}, and concludes the proof of Corollary~\ref{corollary-1middle}. 
\end{proof}

\begin{example}[SRS] \label{example-middle1}
Condition~\eqref{cond-0upper} with $F_{n,\srs}$ in place of $F_n$ is satisfied by the Glivenko-Cantelli theorem. Furthermore, by linearity of functional~\eqref{linear-middle}, the integral $\int_{F^{-1}(p_1)}^{F^{-1}(p_2)}  F_n(x) \dd x$ is the arithmetic mean $n^{-1}\sum_{i=1}^n Y_{i,\half}(p_1,p_2)$ of $n$ independent copies of the random variable 
\begin{align}
Y_{\half}(p_1,p_2)
&=\int_{F^{-1}(p_1)}^{F^{-1}(p_2)}  \mathds{1}\{X\le x\} \dd x   
\notag 
\\
&= \big( F^{-1}(p_2)-X \big)^{+} - \big( F^{-1}(p_1)-X \big)^{+}.  
\label{rv-middle}
\end{align}
The random variable $Y_{\half}(p_1,p_2)$ always has a finite first moment (see Note~\ref{note-5} for details). Consequently, $\int_{p_1}^{p_2} F_{n,\srs}^{-1}(u)\dd u $ is a consistent estimator of $\int_{p_1}^{p_2} F^{-1}(u)\dd u $, that is, for every cdf $F$, we have 
\begin{equation}\label{consistency-middle}
\int_{p_1}^{p_2} F_{n,\srs}^{-1}(u)\dd u  \stackrel{\mathbb{P}}{\to} \int_{p_1}^{p_2} F^{-1}(u)\dd u . 
\end{equation}
\end{example}

\begin{corollary}[bias]\label{corollary-2middle} 
Let $F$ be any cdf, and let $F_1,F_2,\ldots $ be any sequence of cdf's that are unbiased estimators of $F$, that is, satisfy  condition~\eqref{cond-0bupper}, and let the cdf's be such that condition~\eqref{cond-0bupper-bias} is satisfied with $p\in \{p_1,p_2\}$. Then 
\begin{align*}
\mathrm{Bias}_n^{\half}(p_1,p_2):= &\mathbb{E}\bigg(\int_{p_1}^{p_2} F_n^{-1}(u)\dd u\bigg)  -\int_{p_1}^{p_2} F^{-1}(u)\dd u 
\\
= &  \mathbb{E}\big(\Rem(p_2;F,F_n)\big) - \mathbb{E}\big(\Rem(p_1;F,F_n)\big) 
\\
\in & \mathbb{R} . 
\end{align*}
\end{corollary}

\begin{proof} 
Theorem~\ref{theorem-0middle} with $F_n$ instead of $G$ implies Corollary~\ref{corollary-2middle}. 
\end{proof}

\begin{example}[SRS] \label{example-middle2}
When $F_n$ is $ F_{n,\srs}$, condition~\eqref{cond-0bupper} is satisfied. Furthermore, condition~\eqref{cond-0bupper-bias} is satisfied when $F\in \mathcal{T}_{\varepsilon}$ for some $\varepsilon>0$ (recall Example~\ref{example-upper2} for details). Hence, we conclude from  Corollary~\ref{corollary-2middle} that when $F \in \mathcal{T}_{\varepsilon}$, the estimator $\int_{p_1}^{p_2} F_{n,\srs}^{-1}(u)\dd u  $ of the middle-layer integral $ \int_{p_1}^{p_2} F^{-1}(u)\dd u $, although being consistent by Example~\ref{example-middle1}, has the bias 
\[
\mathrm{Bias}_{n,\srs}^{\half}(p_1,p_2)=\mathbb{E}\big(\Rem(p_2;F,F_{n,\srs})\big) - \mathbb{E}\big(\Rem(p_1;F,F_{n,\srs})\big) \in \mathbb{R},  
\]
which might be any real number.  Note that a sufficient condition for $F\in \mathcal{T}_{\varepsilon}$ for some $\varepsilon>0$ is the requirement $F\in \mathcal{F}_{\varepsilon}$, that is, $\mathbb{E}(|X|^{\varepsilon})<\infty$ for some $\varepsilon>0$. 
\end{example}

\begin{remark}  
Naturally, we can get more insights into the bias if we are willing to make additional assumptions about $F$. For example, following the studies of \citet{gh2006,gh2007} on trimmed means, if we assume that $F$ has a pdf $f=F'$ that is strictly positive and continuous at
the quantiles $F^{-1}(p_1)$ and $F^{-1}(p_2)$, then 
\begin{equation}\label{bias-as-form}
\mathrm{Bias}_{n,\srs}^{\half}(p_1,p_2)
= 
\underbrace{\frac{p_2(1-p_2)}{2nf(F^{-1}(p_2))}}_{\textrm{vanishes when $p_2=1$}}
-\underbrace{\frac{p_1(1-p_1)}{2nf(F^{-1}(p_1))}}_{\textrm{vanishes when $p_1=0$}} 
+\, o\left(\frac 1n\right).
\end{equation}
This asymptotic formula helps us to appreciate the fact that in Example~\ref{example-middle2} we were not able to say anything definitive about the sign of the bias $\mathrm{Bias}_{n,\srs}^{\half}(p_1,p_2)$, and it also helps us to see why we could earlier show that  $\mathrm{Bias}_{n,\srs}^{\one}(p) \le 0$ (statement~\eqref{bias-upper-srs}) and $\mathrm{Bias}_{n,\srs}^{\zero}(p) \ge 0$ (statement~\eqref{bias-lower-srs}) for any $p\in (0,1)$. Furthermore, asymptotic formula~\eqref{bias-as-form} helps us to understand why the Vervaat process needs to be normalized by $n$ to stabilize its asymptotic behaviour and hence to obtain a non-degenerate limit, which is a half of the squared Brownian bridge, that is, $B^2(p)/2$, $0\le p \le 1$. Note in this regard that the expected value of $B^2(p)/2$ is equal to $p(1-p)/2$, 
thus serving a further clarification of the form of the two leading terms on the right-hand side of equation~\eqref{bias-as-form}.
\end{remark}

\begin{corollary}[asymptotic distribution]\label{corollary-3middle} 
Let $F$ be any cdf, and let $F_1,F_2,\ldots $ be any sequence of cdf's satisfying condition~\eqref{cond-1upper} with $p\in \{p_1,p_2\}$ and condition~\eqref{cond-2upper} with some normalizing constants $A_n\to \infty $ when $n\to \infty $.  Then 
\begin{equation}\label{eq-1db}
A_n\left( \int_{p_1}^{p_2} F_n^{-1}(u) \dd u - \int_{p_1}^{p_2} F^{-1}(u) \dd u \right) 
=A_n\int_{F^{-1}(p_1)}^{F^{-1}(p_2)}  \big( F(x)- F_n(x)\big) \dd x +o_{\mathbb{P}}(1)
\end{equation} 
and, therefore, the convergence-in-distribution statement 
\begin{equation}\label{eq-1db-a}
A_n\left( \int_{p_1}^{p_2} F_n^{-1}(u) \dd u - \int_{p_1}^{p_2} F^{-1}(u) \dd u \right) 
\stackrel{d}{\to} \mathcal{L}_{\half}(p_1,p_2)
\end{equation} 
holds if and only if 
\begin{equation}\label{eq-1db-b}
A_n\left(\int_{F^{-1}(p_1)}^{F^{-1}(p_2)}  F_n(x)\dd x - \int_{F^{-1}(p_1)}^{F^{-1}(p_2)}  F(x)\dd x \right) 
\stackrel{d}{\to} - \mathcal{L}_{\half}(p_1,p_2) , 
\end{equation} 
where $\mathcal{L}_{\half}(p_1,p_2)$ is a random variable determined by statement~\eqref{eq-1db-b}.   
\end{corollary}

\begin{proof} 
Using  Theorem~\ref{theorem-0middle} with $F_n$ instead of $G$, we only need to show that $A_n\Rem(p;F,F_n)$ converges in probability to $0$, which we showed in the proof of Corollary~\ref{corollary-3upper}.  This establishes equation~\eqref{eq-1db}, which implies the equivalence of statements~\eqref{eq-1db-a} and~\eqref{eq-1db-b}, and concludes the proof of Corollary~\ref{corollary-3middle}. 
\end{proof}

\begin{example}[SRS] \label{example-middle3} 
When $F_n$ is $F_{n,\srs}$, condition~\eqref{cond-2upper} is a consequence of the Kolmogorov-Smirnov theorem. Since condition~\eqref{cond-2upper} is satisfied, condition~\eqref{cond-1upper} is satisfied as well, provided that, according to Lemma~\ref{lemma-1}, the quantile function $F^{-1}$ is continuous at the points $p_1$ and $p_2$.  
Furthermore, since the integral $\int_{F^{-1}(p_1)}^{F^{-1}(p_2)} F_n(x) \dd x$ is the arithmetic mean of $n$ independent copies of random variable~\eqref{rv-middle}, 
we conclude that the integral $\int_{F^{-1}(p_1)}^{F^{-1}(p_2)} F_n(x) \dd x$ satisfies the central limit theorem. In summary, therefore, when the quantile function $F^{-1}$ is continuous at the points $p_1$ and $p_2$, we have 
\begin{equation}\label{normality-middle}
\sqrt{n}\left(\int_{p_1}^{p_2} F_{n,\srs}^{-1}(u)\dd u  - \int_{p_1}^{p_2} F^{-1}(u)\dd u \right) 
\stackrel{d}{\to} \mathcal{N}(0,\sigma_{\half}^2(p_1,p_2)) , 
\end{equation}
where $\sigma_{\half}^2(p_1,p_2)$ is the variance of the random variable $Y_{\half}(p_1,p_2)$. The variance can be expressed by the formula (see Note~\ref{note-5})
\begin{equation}
\sigma_{\half}^2(p_1,p_2)
=\int_{F^{-1}(p_1)}^{F^{-1}(p_2)}\int_{F^{-1}(p_1)}^{F^{-1}(p_2)}
 \big( F(x \wedge y) -F(x)F(y) \big) \dd x  \dd y. 
\label{normality-middle-var2}
\end{equation}
\end{example}

\begin{computation}[SRS] \label{computation-middle} 
To obtain a convenient computational formula for the empirical middle-layer integral $\int_{p_1}^{p_2} F_{n,\srs}^{-1}(u)\dd u$, we can start with the equation 
\begin{equation}
\int_{p_1}^{p_2} F_{n,\srs}^{-1}(u)\dd u
= \int_{0}^{p_2} F_{n,\srs}^{-1}(u)\dd u - \int_{0}^{p_1} F_{n,\srs}^{-1}(u)\dd u   
\label{quantile-middle-3}
\end{equation}
and then apply, for example, computational formula~\eqref{quantile-lower-3ii} for the two empirical lower-layer integrals on the right-hand side of equation~\eqref{quantile-middle-3}. (All empirical integrals are finite, and so manipulations like those in equation~\eqref{quantile-middle-3} do not pose any technical issues.) In particular, we obtain the following equations 
\begin{align*}
\int_{p_1}^{p_2} F_{n,\srs}^{-1}(u)\dd u
&= {1\over n}\sum_{i=\lceil np_1 \rceil +1}^{\lceil np_2 \rceil } X_{i :n} 
-\left( {\lceil np_2 \rceil \over n}- p_2\right) X_{\lceil np_2\rceil :n}  
+\left( {\lceil np_1 \rceil \over n}- p_1\right) X_{\lceil np_1\rceil :n}    
\\
&= {\lceil np_2 \rceil-\lceil np_1 \rceil\over n} 
\underbrace{{1\over \lceil np_2 \rceil-\lceil np_1 \rceil}
\sum_{i=\lceil np_1 \rceil +1}^{\lceil np_2 \rceil } X_{i :n}}_{\textrm{trimmed mean}} 
\\
&\qquad \underbrace{-\left( {\lceil np_2 \rceil \over n}- p_2\right) X_{\lceil np_2\rceil :n}  
+\left( {\lceil np_1 \rceil \over n}- p_1\right) X_{\lceil np_1\rceil :n}}_{\textrm{asymptotically negligible terms}}   
\end{align*}
with the right-hand side connecting the empirical middle-layer integral with the trimmed (on both sides) mean that was considered by \citet{s1973} and played a pivotal role in the development of robust statistics, including the introduction of the method of trimmed moments (MTM) by \citet{bjz2009} and the method of Winsorized moments (MWM) by \citet{zbg2018}. 
\end{computation}

\section{Illustrations: iid random variables}
\label{examples}

To make the following illustrations less voluminous and easier connectable to what is known in the literature and hence maximally illuminating, we use SRS throughout this section. That is, we work with iid random variables  $X_1, \dots , X_n $ and from them arising empirical cdf $ F_{n,\srs}$, which is a consistent and unbiased estimator of the underlying population cdf $F$.

\subsection{Upside tail-value-at-risk}
\label{illustration-1} 

Given any $p\in (0,1)$ and $F\in \mathcal{F}^{+}_{1}$, the upside tail-value-at-risk $\ATVaR(p)$ is given by equation~\eqref{ill-ATVaR-0}. 
As an illustration, in Figure~\ref{fig:ATVaR} we depicted $p\mapsto \ATVaR(p)$ for the Pareto (Type I) distribution, whose quantile function is $p\mapsto x_0(1-p)^{-1/\alpha } $ 
with the scale $x_0>0$ and shape $\alpha>0$ parameters, which we set to $1$ and $3$, respectively. The empirical counterpart $\ATVaR_{n,\srs}(p)$ is defined by equation~\eqref{ill-ATVaR-0} with  $F_{n,\srs}$ instead of $F$. (For a comprehensive review of various estimators of $\ATVaR$ and its sister risk measure called Expected Shortfall, we refer to \citet{nzc20213}.) We have the following three statements concerning consistency, bias, and asymptotic normality of  $\ATVaR_{n,\srs}(p)$.

\subsubsection{Consistency}

When $F\in \mathcal{F}^{+}_{1}$, we obtain from Example~\ref{example-upper1} that  
$\ATVaR_{n,\srs}(p) $ is a consistent estimator of $ \ATVaR(p) $, that is,  
\begin{equation}\label{ill-ATVaR-1}
\ATVaR_{n,\srs}(p) \stackrel{\mathbb{P}}{\to} \ATVaR(p) . 
\end{equation}
This result is a special case of \citet[]{gsz2022a} who consider consistency and with it associated fixed-margin confidence intervals for the tail conditional allocation, which generalizes $\ATVaR(p)$. For related bootstrap-type considerations and fixed-margin confidence intervals, we refer to \citet[]{gsz2024}. 

\subsubsection{Bias}

When $F\in \mathcal{F}^{+}_{1} \cap \mathcal{T}_{\varepsilon}$ for some $\varepsilon>0$ (recall Note~\ref{note+epsilon}), we obtain from Example~\ref{example-upper2} that $\ATVaR_{n,\srs}(p)$ is a non-positively biased estimator of $\ATVaR(p)$, that is, 
\begin{equation}\label{ill-ATVaR-2}
\mathbb{E}\big(\ATVaR_{n,\srs}(p)\big) \le \ATVaR(p)  
\end{equation} 
for every $n\in \mathbb{N}$. This property, although not stated explicitly, was earlier established by \citet[][p.~3602]{bjpz2008}. \citet{G2025} offers an illuminating  discussion of the property from the heuristic point of view.

\subsubsection{Asymptotic normality}

When $F\in \mathcal{F}^{+}_{2}$ and  the quantile function $F^{-1}$ is continuous at the point $p$, we obtain from  Example~\ref{example-upper3} that 
\begin{equation}\label{ill-ATVaR-3}
\sqrt{n}\left(\ATVaR_{n,\srs}(p) - \ATVaR(p)  \right) 
\stackrel{d}{\to} \mathcal{N}(0,\sigma_{\ATVaR}^2(p)), 
\end{equation} 
where the asymptotic variance is  
\begin{equation}\label{ill-ATVaR-4}
\sigma_{\ATVaR}^2(p)={1\over (1-p)^2}\int^{\infty }_{F^{-1}(p)} \int^{\infty }_{F^{-1}(p)} \big( F(x \wedge y) -F(x)F(y) \big) \dd x  \dd y. 
\end{equation}
This statement is a special case of a more general result derived by \citet[]{bjpz2008}, who consider not just pointwise convergence to the limiting distribution but also uniform convergence over all $p\in (0,1)$, which allows one to establish confidence bands for the function $p\mapsto \ATVaR(p)$. Statement~\eqref{ill-ATVaR-3} is also a special case of \citet[]{gsz2022b} who consider the limiting distribution of the tail conditional allocation, which generalizes $\ATVaR(p)$.

\subsection{Downside tail-value-at-risk}
\label{illustration-2}

Given $p\in (0,1)$ and $F\in \mathcal{F}^{-}_{1}$, the downside tail-value-at-risk $\FTVaR(p)$ is given by equation~\eqref{ill-FTVaR-0}. 
As an illustration, in Figure~\ref{fig:FTVaR} we depicted $p\mapsto \FTVaR(p)$ for the Pareto (Type I) distribution with the scale $x_0=1$ and shape $\alpha=3$ parameters.  The empirical counterpart $\FTVaR_{n,\srs}(p)$ is defined by equation~\eqref{ill-FTVaR-0} with  $F_{n,\srs}$ instead of $F$.  We have the following three statements concerning consistency, bias, and asymptotic normality of  $\FTVaR_{n,\srs}(p)$.

\subsubsection{Consistency}
 
When $F\in \mathcal{F}^{-}_{1}$, we obtain from Example~\ref{example-lower1} that  
$\FTVaR_{n,\srs}(p) $ is a consistent estimator of $ \FTVaR(p) $, that is,  
\begin{equation}\label{ill-FTVaR-1}
\FTVaR_{n,\srs}(p) \stackrel{\mathbb{P}}{\to} \FTVaR(p) . 
\end{equation}

\subsubsection{Bias}
 
When $F\in \mathcal{F}^{-}_{1}\cap \mathcal{T}_{\varepsilon}$ for some $\varepsilon>0$ (recall Note~\ref{note-epsilon}), we obtain from Example~\ref{example-lower2} that $\FTVaR_{n,\srs}(p)$ is a non-negatively biased estimator of $\FTVaR(p)$, that is, 
\begin{equation}\label{ill-FTVaR-2}
\mathbb{E}\big(\FTVaR_{n,\srs}(p)\big) \ge \FTVaR(p)  
\end{equation} 
for every $n\in \mathbb{N}$.

\subsubsection{Asymptotic normality}
 
When $F\in \mathcal{F}^{-}_{2}$ and the quantile function $F^{-1}$ is continuous at the point $p$, we obtain from  Example~\ref{example-lower3} that 
\begin{equation}\label{ill-FTVaR-3}
\sqrt{n}\left(\FTVaR_{n,\srs}(p) - \FTVaR(p)  \right) 
\stackrel{d}{\to} \mathcal{N}(0,\sigma_{\FTVaR}^2(p)), 
\end{equation} 
where the asymptotic variance is  
\begin{equation}\label{ill-FTVaR-4}
\sigma_{\FTVaR}^2(p)={1\over p^2}\int_{-\infty }^{F^{-1}(p)} \int_{-\infty }^{F^{-1}(p)} 
 \big( F(x \wedge y) -F(x)F(y) \big) \dd x  \dd y. 
\end{equation}

\subsection{Lorenz curve}
\label{illustration-3} 

Let $F\in \mathcal{F}^{\mu\neq 0}_{1}$, where $ \mathcal{F}^{\mu\neq 0}_{1}$ denotes the set of all cdf's that have finite but non-zero first moments, that is, 
$\mu=\mathbb{E}(X) \in \mathbb{R} \setminus \{0\}$ for $X\sim F$. Following \citet{g1971}, the Lorenz curve $\LC(p)$ is defined by equation~\eqref{ill-Lorenz-0} with $p$ running through the unit interval $[0,1]$. 
(The Lorenz curve is usually defined and meaningfully interpreted only on the set of non-negative random variables $X\ge 0$.)  
Since the cases $p=0$ and $p=1$ are trivial, we restrict ourselves to $p\in (0,1)$. 
As an illustration, in Figure~\ref{fig:LC} we depicted $p\mapsto \LC(p)$ in the case of the Pareto (Type I) distribution with the scale $x_0=1$ and shape $\alpha=3$ parameters. 
When $F$ is replaced by $F_{n,\srs}$ on the right-hand side of equation~\eqref{ill-Lorenz-0}, we obtain the empirical Lorenz curve $\LC_{n,\srs}(p)$.  We have the following three statements concerning consistency, bias, and asymptotic normality of  $\LC_{n,\srs}(p)$ for any fixed $p\in (0,1)$. 

\subsubsection{Consistency}
 
When $F\in \mathcal{F}^{\mu\neq 0}_{1}$, we obtain from Example~\ref{example-lower1} and consistency of the arithmetic mean $\mu_{n,\srs}$ 
that  
$\LC_{n,\srs}(p) $ is a consistent estimator of $ \LC(p) $, that is, 
\begin{equation}\label{ill-LC-1}
\LC_{n,\srs}(p) \stackrel{\mathbb{P}}{\to} \LC(p) . 
\end{equation}

\subsubsection{Bias}
 
The result of Example~\ref{example-lower2} suggests that $\LC_{n,\srs}(p)$ might be a positively biased estimator of $\LC(p)$, given that the arithmetic mean $\mu_{n,\srs}$ is an unbiased estimator of $\mu $.  
This may or may not be true as demonstrated by \citet{av20215}, who summarize their findings on the subject as follows: 
\begin{quote}
For several parent
distributions it is possible to prove that the sample Lorenz curve is a positively biased
estimate of the population Lorenz curve. In this paper, several sufficient conditions
for such positive bias are investigated. An example shows that negative bias is not
impossible, though apparently not common. \citep[][p.~3]{av20215}
\end{quote}

\subsubsection{Asymptotic normality}

When $F\in \mathcal{F}^{\mu\neq 0}_{2}$ and the quantile function $F^{-1}$ is continuous at the point $p$, we obtain from results of Section~\ref{sect-lower} and equation~\eqref{eq-0} that (see Lemma~\ref{LC-clt} in Appendix~\ref{subsub-asresults} for a proof) 
\begin{equation}\label{ill-LC-3}
\sqrt{n}\left(\LC_{n,\srs}(p) - \LC(p)  \right) 
\stackrel{d}{\to} \mathcal{N}(0,\sigma_{\lc}^2(p)), 
\end{equation} 
where the asymptotic variance $\sigma_{\lc}^2(p)$ is the second moment of the mean-zero random variable 
\begin{equation}\label{ill-LC-4}
Y_{\lc}(p)
= {1\over \mu } \int_{-\infty }^{F^{-1}(p)} \big( \mathds{1}\{X\le x\} -F(x)\big) \dd x  
+{\LC(p) \over \mu } \big( X-\mu \big) . 
\end{equation}
For confidence bands (i.e., simultaneous confidence intervals over all $p\in (0,1)$) for the Lorenz curve under minimal conditions on the population distribution, we refer to \citet{cfs1998}. 
It is also useful to note that the variance $\sigma_{\lc}^2(p)$ can be expressed in a form resembling those for $\sigma_{\ATVaR}^2(p)$ and $\sigma_{\FTVaR}^2(p)$ in equations~\eqref{ill-ATVaR-4} and~\eqref{ill-FTVaR-4}, respectively, but apart from being nice theoretical exercises, we do not see much practical value in doing so: first, this would not lead to a convenient empirical estimator for  $\sigma_{\lc}^2(p)$, and second, the variance  $\sigma_{\lc}^2(p)$ can be easier estimated via a resampling technique \citep[e.g.,][]{st1995}.

\subsection{Gini curve}
\label{illustration-4}

Let $F\in \mathcal{F}^{\mu\neq 0}_{1}$, where the class $\mathcal{F}^{\mu\neq 0}_{1}$ is defined in Example~\ref{illustration-3}. The Gini curve $\GC(p)$ is given by equation~\eqref{ill-Gini-0}, 
where $p$ runs through the unit interval $[0,1]$. (The Gini curve is usually defined and meaningfully interpreted only on the set of non-negative random variables $X\ge 0$.) 
Since the cases $p=0$ and $p=1$ are trivial, we restrict ourselves to $p\in (0,1)$. 
As an illustration, in Figure~\ref{fig:GC} we depicted $p\mapsto \GC(p)$ in the case of the Pareto (Type I) distribution with the scale $x_0=1$ and shape $\alpha=3$ parameters. 
When $F$ is replaced by $F_{n,\srs}$  on the right-hand side of equation~\eqref{ill-Gini-0}, we obtain the empirical Gini curve $\GC_{n,\srs}(p)$.  We have the following three statements concerning consistency, bias, and asymptotic normality of  $\GC_{n,\srs}(p)$ for any $p\in (0,1)$. 

\subsubsection{Consistency}

When $F\in \mathcal{F}^{\mu\neq 0}_{1}$, we obtain from consistency of the arithmetic mean $\mu_{n,\srs}$ and Examples~\ref{example-upper1} and~\ref{example-lower1} that  
$\GC_{n,\srs}(p) $ is a consistent estimator of $ \GC(p) $, that is, 
\begin{equation}\label{ill-GC-1}
\GC_{n,\srs}(p) \stackrel{\mathbb{P}}{\to} \GC(p) . 
\end{equation}

\subsubsection{Bias}
 
Even more so than in the case of the Lorenz curve, we cannot say anything definitive about the bias of $\GC_{n,\srs}(p)$. 

\subsubsection{Asymptotic normality}
  
When $F\in \mathcal{F}^{\mu\neq 0}_{2}$ and the quantile function $F^{-1}$ is continuous at the points $p$ and $1-p$, we obtain from results of Section~\ref{sect-lower} and equation~\eqref{eq-0} that (see Lemma~\ref{GC-clt} in Appendix~\ref{subsub-asresults} for a proof)  
\begin{equation}\label{ill-GC-3}
\sqrt{n}\left(\GC_{n,\srs}(p) - \GC(p)  \right) 
\stackrel{d}{\to} \mathcal{N}(0,\sigma_{\gc}^2(p)), 
\end{equation} 
where the asymptotic variance $\sigma_{\gc}^2(p)$ is the second moment of the mean-zero random variable 
\begin{multline}\label{ill-GC-4a}
Y_{\gc}(p)
= {1\over \mu } \int_{-\infty }^{F^{-1}(1-p)} \big( \mathds{1}\{X\le x\} -F(x)\big) \dd x  
\\ 
+ {1\over \mu } \int_{-\infty }^{F^{-1}(p)} \big( \mathds{1}\{X\le x\} -F(x)\big) \dd x  
+{1-\GC(p) \over \mu } \big( X-\mu \big) . 
\end{multline}
Estimating the asymptotic variance using a resampling technique might be the most efficient and speediest way toward, e.g., constructing large-sample confidence intervals for $\GC(p)$.

\section{Illustrations: stationary time series}
\label{time-series-data}

Often in applications, data arrive in the form of time series. Depending on the class of time series, and there are many of them, we may see different normalizing constants, different asymptotic distributions, and even in the case of normal asymptotic distributions, we may see different asymptotic variances. Therefore, to illustrate how our general results work on time series, we need to make a choices, and our's is to work with those time series that are $S$- and $M$-mixing. These two notions of mixing have been introduced by \citet{bhs2009} and \citet{bhs2011}, respectively, and they cover many time series (linear and non-linear). Very importantly, it has also turned out that verifying $S$- and $M$-mixing conditions is often easier than verifying classical mixing conditions.  \citet{bhs2009,bhs2011} provide illuminating discussions of these matters with accompanying examples and references.

\subsection{$S$-mixing and statements~\eqref{eq-1dupper}, \eqref{eq-1da} and \eqref{eq-1db}}
\label{section-smix}

We follow \citet{bhs2009} and say that a time series $(X_t)_{t\in \mathbb{Z}}$ is  $S$-mixing if it satisfies the following two conditions: 
\begin{enumerate}[label={(\Alph*)}]
\item \label{cond-a}
For any $t\in \mathbb{Z}$ and $m \in \mathbb{N}$, there is a random variable $\Upsilon_{t}^{(m)}$ such that the property 
\begin{equation}\label{cond-a1}
\mathbb{P}\big( |X_{t}-\Upsilon_{t}^{(m)}|\ge \gamma_m \big) \le \delta_m
\end{equation}
holds for some deterministic sequences $\gamma_m \to 0$ and $\delta_m \to 0$. 
\item \label{cond-b}
For any disjoint intervals $I_1,\dots , I_r \subset  \mathbb{Z}$ of integers and for any positive integers $m_1,\dots , m_r \in \mathbb{N}$, the vectors $(\Upsilon_{t}^{(m_1)},t\in I_1), \dots, (\Upsilon_{t}^{(m_r)},t\in I_r)$ are independent, provided that the separation between the pairs $I_i$ and $I_j$ is greater than $m_i+m_j$ for all $1\le i < j \le r$, that is, 
\begin{align}
\dist(I_i,I_j)
:=&\inf\big\{|a-b|,a\in I_i,b\in I_j\big\}  
\notag 
\\
> & m_i+m_j . 
\label{cond-b1}
\end{align}  
\end{enumerate}

We can clearly see why $S$-mixing has turned out to be such an attractive notion, in particular from the applications point of view: it  is based on various portions of the time series $(X_t)_{t\in \mathbb{Z}}$ and not on mathematical constructs such as $\sigma$-algebras, as is the case with many classical mixing notions. The only challenge with $S$-mixing is that one needs to construct random variables $\Upsilon_{t}^{(m)}$, but \citet{bhs2009} suggest several recipes for accomplishing this task. To gain intuition on the matter, we next follow one of the recipes and show that the causal autoregressive of order $1$ time series is $S$-mixing.

\begin{example}[AR(1) is $S$-mixing] \label{ex-ar-1}
Let $(X_t)_{t\in \mathbb{Z}}$ be a causal AR(1) time series, that is, 
$X_{t}=\varphi X_{t-1}+\varepsilon_t$ 
for all $t\in \mathbb{Z}$, where $|\varphi |<1$ is a constant and   $(\varepsilon_t)_{t\in \mathbb{Z}}$ is a mean-zero white noise with finite marginal variances $\sigma_{\varepsilon}^2=\Var(\varepsilon_t)<\infty $. In addition, we assume that $\varepsilon_t$'s are iid random variables, thus making $(X_t)_{t\in \mathbb{Z}}$ strictly stationary. The representation 
\begin{equation}\label{ar-1}
X_t=\sum_{i=0}^{\infty} \varphi^i \varepsilon_{t-i}
\end{equation}
holds for all $t\in \mathbb{Z}$. Denote 
\begin{equation}\label{ar-1a}
\Upsilon_t^{(m)}=\sum_{i=0}^{m} \varphi^i \varepsilon_{t-i}. 
\end{equation}
Condition~\ref{cond-b} of $S$-mixing is satisfied because the two random vectors
\begin{gather*}
(\varepsilon_{s-m_i}, \dots, \varepsilon_{s}), \quad s \in I_i, 
\\
(\varepsilon_{t-m_j}, \dots, \varepsilon_{t}), \quad t \in I_j, 
\end{gather*}
do not overlap and are therefore independent whenever $\dist(I_i,I_j) >\max\{m_i,m_j \}$, and so definitely when $\dist(I_i,I_j) >m_i+m_j $. To check condition~\ref{cond-a}, we write 
\begin{align*}
\mathbb{P}\big( |X_{t}-\Upsilon_{t}^{(m)}|\ge \gamma_m \big) 
&\le {1\over \gamma_m^2} \mathbb{E}\bigg( \bigg( \sum_{i=m+1}^{\infty} \varphi^i \varepsilon_{t-i}\bigg)^2 \bigg) 
\\
&=  {\varphi^{2(m+1)}\over \gamma_m^2} {\sigma_{\varepsilon}^2 \over 1-\varphi^2} 
\\
&=: \delta_m , 
\end{align*}
where the bound is due to Markov's inequality. 
For example, if set $\gamma_m=m^{-a}$ for any constant $a>0$, then we have $\delta_m=O(m^{-A})$ when $m\to \infty $ for any constant $A>0$. In fact, for $\gamma_m$ given above, the decay of $\delta_m$  is exponential. 
\end{example}

The following theorem is a very special case of \citet[][Theorem~2, p.~1303]{bhs2009}, but it is exactly what we currently need.

\begin{theorem}\label{bhs-theorem}
Let $(X_t)_{t\in \mathbb{Z}}$ be a strictly stationary and $S$-mixing time series whose marginal cdf $F$ is Lipschitz continuous of order $\theta>0$, that is, there is a constant $c>0$ such that 
\begin{equation}\label{bhs-1}
\big| F(x)-F(y) \big| \le c |x-y|^{\theta }
\end{equation}
for all $x,y\in \mathbb{R}$. Furthermore, assume that condition~\ref{cond-a} is satisfied with 
\begin{equation}\label{bhs-2}
\gamma_m ={1\over m^{A/\theta }} \quad \textrm{and} \quad \delta_m =O\bigg({1\over m^{A}}\bigg)
\end{equation}
for some $A >4$. Then 
\begin{equation}
\sqrt{n}\,\sup_{x\in \mathbb{R}} \big| F_{n,\ts}(x) - F(x)\big| =O_{\mathbb{P}}(1) 
\label{cond-2upper-smix}
\end{equation}
when $n\to \infty $, where $F_{n,\ts}$ is the empirical cdf based on the observable portion  $X_1,\dots, X_n$ of the time series $(X_t)_{t\in \mathbb{Z}}$.  
\end{theorem}

\begin{note} \label{bhs-note-1}
If the marginal cdf $F$ of $(X_t)_{t\in \mathbb{Z}}$ has a bounded pdf, then condition~\eqref{bhs-1} is satisfied with $\theta=1$, and so 
$\gamma_m =m^{-A} $ in this case. 
\end{note}

We can now formulate the following corollaries to Theorem~\ref{bhs-theorem}. 

\begin{corollary}\label{cor-smix-1}
Let all the conditions of Theorem~\ref{bhs-theorem} be satisfied. 
If $F\in \mathcal{F}^{+}_{1}$ and the quantile function $F^{-1}$ is continuous at the point $p$, then statement~\eqref{eq-1dupper} holds, that is, 
\begin{multline}\label{eq-1dupper-smix}
\sqrt{n}\,\left( \int_p^1 F_{n,\ts}^{-1}(u) \dd u -\int_p^1 F^{-1}(u) \dd u \right) 
\\
=\sqrt{n}\,\bigg(  {1\over n} \sum_{i=1}^n  h_{\one}(X_i)
-\mathbb{E}\left( h_{\one}(X) \right) \bigg)  +o_{\mathbb{P}}(1) , 
\end{multline}
where 
\[
h_{\one}(x)  =\big( x-F^{-1}(p) \big)^{+} . 
\]
\end{corollary}

The moment condition $F\in \mathcal{F}^{+}_{1}$ is sufficient to have $o_{\mathbb{P}}(1)$ in asymptotic equation~\eqref{eq-1dupper-smix}, but a stronger moment condition is surely needed to enable the normalized sum on the right-hand side of equation~\eqref{eq-1dupper-smix} to converge to a non-degenerate limit. An analogous note applies to the next corollary.

\begin{corollary}\label{cor-smix-2}
Let all the conditions of Theorem~\ref{bhs-theorem} be satisfied. 
If $F\in \mathcal{F}^{-}_{1}$ and the quantile function $F^{-1}$ is continuous at the point $p$, then statement~\eqref{eq-1da} holds, that is, 
\begin{multline}\label{eq-1da-smix}
\sqrt{n}\,\left( \int_0^p F_{n,\ts}^{-1}(u) \dd u - \int_0^p F^{-1}(u) \dd u \right) 
\\
=\sqrt{n}\,\bigg(  {1\over n} \sum_{i=1}^n h_{\zero}(X_i)
- \mathbb{E}\left( h_{\zero}(X) \right) \bigg)  +o_{\mathbb{P}}(1),   
\end{multline}
where 
\[
h_{\zero}(x)  =\big( F^{-1}(p)-x \big)^{+} . 
\]
\end{corollary}

Our final corollary concerns with the middle-layer integral. 

\begin{corollary}\label{cor-smix-3}
Let all the conditions of Theorem~\ref{bhs-theorem} be satisfied. 
If the quantile function $F^{-1}$ is continuous at the points $p_1$ and $p_2$, then statement~\eqref{eq-1db} holds, that is,  
\begin{multline}\label{eq-1db-smix}
\sqrt{n}\,\left( \int_{p_1}^{p_2} F_{n,\ts}^{-1}(u) \dd u - \int_{p_1}^{p_2} F^{-1}(u) \dd u \right) 
\\
=\sqrt{n}\,\bigg( {1\over n} \sum_{i=1}^n h_{\half}(X_i)
-\mathbb{E}\left( h_{\half}(X) \right)  \bigg)  +o_{\mathbb{P}}(1), 
\end{multline} 
where 
\[
h_{\half}(X)  = \big( F^{-1}(p_2)-x \big)^{+} - \big( F^{-1}(p_1)-x \big)^{+} . 
\]
\end{corollary}

Corollary~\ref{cor-smix-3} does not require any moment assumption because the random variables $X_i$ are transformed by the function $h_{\half}$ that is bounded. Incidentally, note that the three functions $h_{\one}$, $h_{\zero}$ and $h_{\half}$ are Lipschitz continuous of order $1$, which will be an important property in next Section~\ref{section-mmix}.

\subsection{$M$-mixing and statements~\eqref{eq-1dupper-a}, \eqref{eq-1da-a} and \eqref{eq-1db-a}}
\label{section-mmix}

In view of the results of the previous section, in order to establish asymptotic distributions of the centered and $\sqrt{n}$-normalized integrals of quantiles, i.e., to establish statements~\eqref{eq-1dupper-a}, \eqref{eq-1da-a} and \eqref{eq-1db-a}, we are left to verify the validity of 
statements~\eqref{eq-1dupper-b}, \eqref{eq-1da-b} and \eqref{eq-1db-b}. That is, in view of Corollaries~\ref{cor-smix-1}--\ref{cor-smix-3}, we need to show  
\begin{align}
\sqrt{n}\,\bigg(  {1\over n} \sum_{i=1}^n  h_{\one}(X_i)
-\mathbb{E}\left( h_{\one}(X) \right) \bigg) 
&\stackrel{d}{\to} \mathcal{L}_{\one}(p) ,  
\label{eq-1dupper-bbb}
\\
\sqrt{n}\,\bigg(  {1\over n} \sum_{i=1}^n h_{\zero}(X_i)
- \mathbb{E}\left( h_{\zero}(X) \right) \bigg) 
& \stackrel{d}{\to}  - \mathcal{L}_{\zero}(p) , 
\label{eq-1da-bbb}
\\
\sqrt{n}\,\bigg( {1\over n} \sum_{i=1}^n h_{\half}(X_i)
-\mathbb{E}\left( h_{\half}(X) \right)  \bigg) 
&\stackrel{d}{\to} - \mathcal{L}_{\half}(p_1,p_2) ,  
\label{eq-1db-bbb}
\end{align} 
where the functions 
$h_{\one}$, $h_{\zero}$ and $h_{\half}$ are defined in Corollaries~\ref{cor-smix-1}--\ref{cor-smix-3}. Obviously, the random variables in the sums need to have at least finite second moments, and the forms of the functions $h_{\one}$ and $h_{\zero}$ translate this requirement into the conditions $F\in \mathcal{F}^{+}_{2}$ and $F\in \mathcal{F}^{-}_{2}$, respectively. Hence, unlike in Section~\ref{section-smix} where we worked with indicators and thus employed the probability-based notion of $S$-mixing, we now need to work with a moment-based notion of mixing. This leads us to the notion of $M$-mixing introduced by \citet{bhs2011}.

Namely, following \citet{bhs2011}, we say that a time series $(X_t)_{t\in \mathbb{Z}}$ is  $M_p$-mixing for some $p\ge 1$ if it satisfies the following two conditions: 
\begin{enumerate}[label={(\Alph*)}]
\setcounter{enumi}{2}
\item\label{cond-c}
For any $t\in \mathbb{Z}$ and $m \in \mathbb{N}$, there is a random variable $\Psi_{t}^{(m)}$ such that the bound 
\begin{equation}\label{cond-c1}
\left(\mathbb{E}\big( |X_{t}-\Psi_{t}^{(m)}|^p\big)\right)^{1/p} \le \varrho_m
\end{equation}
holds for some deterministic sequence $\varrho_m \to 0$. 
\item \label{cond-d}
For any disjoint intervals $I_1,\dots , I_r \subset  \mathbb{Z}$ of integers and for any positive integers $m_1,\dots , m_r \in \mathbb{N}$, the vectors $(\Psi_{t}^{(m_1)},t\in I_1), \dots, (\Psi_{t}^{(m_r)},t\in I_r)$ are independent provided that 
\begin{equation}\label{cond-d1}
\dist(I_i,I_j) >\max\{m_i,m_j \},    
\end{equation}  
where the definition of $\dist(I_i,I_j)$ is given in condition~\ref{cond-b} of $S$-mixing.   
\end{enumerate}

Note that condition~\ref{cond-c} implicitly requires $X_t$ to have a finite $p^{\textrm{th}}$ moment, that is, the cdf of $X_t$ must be in the class $\mathcal{F}_{p}= \mathcal{F}^{-}_{p} \cap \mathcal{F}^{+}_{p}$. As to the random variable $\Psi_{t}^{(m)}$ postulated in the condition, 
\citet{bhs2011} offer several recipes for constructing it. The following example illustrates the notion and how to verify it.

\begin{example}[AR(1) is $M_p$-mixing] 
Let $(X_t)_{t\in \mathbb{Z}}$ be the same AR(1) time series as in Example~\ref{ex-ar-1}. In addition, assume that $\varepsilon_t$'s have (identical) finite $p^{\textrm{th}}$ moments $\mathbb{E}(|\varepsilon_t|^p)<\infty $ for some $p\ge 2$. Hence, representation~\eqref{ar-1} holds. Denote 
\[
\Psi_{t}^{(m)}=\sum_{i=0}^{m} \varphi^i \varepsilon_{t-i}, 
\]
which is the same as $\Upsilon_{t}^{(m)}$ defined by equation~\eqref{ar-1a}. (In general, $\Psi_{t}^{(m)}$ and $\Upsilon_{t}^{(m)}$ do not need to be the same.)   
Since $\dist(I_i,I_j) >\max\{m_i ,m_j \}$, condition~\ref{cond-d} of $M_p$-mixing is satisfied, just like condition~\ref{cond-b} of $S$-mixing is.  To check condition~\ref{cond-c}, we write 
\begin{align*}
\left(\mathbb{E}\big( |X_{t}-\Psi_{t}^{(m)}|^p \big) \right)^{1/p}
&=\Bigg(\mathbb{E}\bigg( \bigg| \sum_{i=m+1}^{\infty} \varphi^i \varepsilon_{t-i}\bigg|^p \bigg) \Bigg)^{1/p}
\\
&\le\sum_{i=m+1}^{\infty} \Big(\mathbb{E}\big( \big| \varphi^i \varepsilon_{t-i}\big|^p \big) \Big)^{1/p}
\\
&=\sum_{i=m+1}^{\infty} |\varphi|^i\Big(\mathbb{E}\big( \big| \varepsilon_{0}\big|^p \big) \Big)^{1/p}
\\
&=   { |\varphi|^{m+1}\over 1-|\varphi| } \Big(\mathbb{E}\big( \big| \varepsilon_{0}\big|^p \big) \Big)^{1/p}
\\
&=: \varrho_m , 
\end{align*}
where the bound is due to Minkowski's inequality. 
Hence, $\varrho_m=O(m^{-A})$ when $m\to \infty $ for any constant $A>0$. In fact, the decay of $\varrho_m$  is exponential. 
\end{example}

The following theorem is a very special case of \citet[][Theorem~1, p.~2445]{bhs2011}, but it is exactly what we need.

\begin{theorem}\label{bhs-theorem-mmix}
Let $(X_t)_{t\in \mathbb{Z}}$ be a strictly stationary and $M_p$-mixing time series for some $p>2$ and with 
\begin{equation}\label{bhs-1-mmix}
\varrho_m =O\bigg({1\over m^{A}}\bigg), 
\end{equation}
where 
\begin{equation}\label{bhs-2-mmix}
A> \max\bigg\{ 1, {p-2\over 2\eta } \bigg( 1-{1+\eta \over p} \bigg) \bigg\}  
\quad \textrm{and} \quad {1+\eta \over p} < {1\over 2}. 
\end{equation}
Then, when $n\to \infty $,  
\begin{equation}\label{bhs-3-mmix}
\sqrt{n}\,\bigg(  {1\over n} \sum_{i=1}^n X_i-\mathbb{E}\left( X \right) \bigg) 
\stackrel{d}{\to} \mathcal{N}(0,\nu^2)
\end{equation}
with the asymptotic variance 
\[
\nu^2=\sum_{h=-\infty}^{\infty}\Cov\big( X_0,X_h\big) . 
\]
\end{theorem}

Since the functions $h_{\one}$, $h_{\zero}$ and $h_{\half}$ are Lipschitz continuous of order $1$, the three transformed time series $(h_{\one}(X_t))_{t\in \mathbb{Z}}$, $(h_{\zero}(X_t))_{t\in \mathbb{Z}}$ and $(h_{\half}(X_t))_{t\in \mathbb{Z}}$ are $M_p$-mixing with the same $p$ and $\varrho_m $ as the original time series $(X_t)_{t\in \mathbb{Z}}$. Furthermore, strict stationarity of $(X_t)_{t\in \mathbb{Z}}$ implies strict stationarity of the three transformed time series. Hence, we can now formulate the following corollary. 

\begin{corollary}\label{cor-mmix}
Let $F\in \mathcal{F}_{p}$ for some $p>2$, and let the conditions of Theorems~\ref{bhs-theorem} and~\ref{bhs-theorem-mmix}  be satisfied. 
\begin{enumerate}[label={{\rm(\arabic*)}}]
\item 
If the quantile function $F^{-1}$ is continuous at the point $p$, then statement~\eqref{eq-1dupper-bbb} and, therefore, statement~\eqref{eq-1dupper-a} hold with a limiting mean-zero normal random variable $\mathcal{L}_{\one}(p)$ whose variance is 
\[
\nu^2_{\one}=\sum_{h=-\infty}^{\infty}\Cov\big( h_{\one}(X_0),h_{\one}(X_h)\big) . 
\]
That is, when $n\to \infty $, we have 
\[
\sqrt{n}\,\left( \int_p^1 F_{n,\ts}^{-1}(u) \dd u -\int_p^1 F^{-1}(u) \dd u \right) 
\stackrel{d}{\to} \mathcal{N}\big(0,\nu^2_{\one}\big). 
\]
\item 
If the quantile function $F^{-1}$ is continuous at the point $p$, then statement~\eqref{eq-1da-bbb} and, therefore, statement~\eqref{eq-1da-a} hold with a limiting mean-zero normal random variable $\mathcal{L}_{\zero}(p)$ whose variance  is 
\[
\nu^2_{\zero}=\sum_{h=-\infty}^{\infty}\Cov\big( h_{\zero}(X_0),h_{\zero}(X_h)\big) . 
\]
That is, when $n\to \infty $, we have 
\[
\sqrt{n}\,\left( \int_0^p F_{n,\ts}^{-1}(u) \dd u - \int_0^p F^{-1}(u) \dd u \right) 
\stackrel{d}{\to} \mathcal{N}\big(0,\nu^2_{\zero}\big). 
\]
\item 
If the quantile function $F^{-1}$ is continuous at the points $p_1$ and $p_2$, then statement~\eqref{eq-1db-bbb} and, therefore, statement~\eqref{eq-1db-a} hold with a limiting mean-zero normal random variable $\mathcal{L}_{\half}(p)$ whose variance  is 
\[
\nu^2_{\half}=\sum_{h=-\infty}^{\infty}\Cov\big( h_{\half}(X_0),h_{\half}(X_h)\big) . 
\]
That is, when $n\to \infty $, we have 
\[
\sqrt{n}\,\left( \int_{p_1}^{p_2} F_{n,\ts}^{-1}(u) \dd u - \int_{p_1}^{p_2} F^{-1}(u) \dd u \right) 
\stackrel{d}{\to} \mathcal{N}\big(0,\nu^2_{\half}\big). 
\]
\end{enumerate}
\end{corollary}

In the three statements of Corollary~\ref{cor-mmix}, the asymptotic variances are complex to estimate empirically, and so practically useful assessments of the variances  should probably be done using resampling techniques \cite[e.g.][]{l2003,kp2011}.

\section{$L$-functionals} 
\label{sect-L-functional}

Given what we have established in the previous sections, we may now wish to see if those results could be extended to the class of so-called $L$-functionals, whose definition we recall next. Namely, the $L$-functional $L_{w}:\mathcal{F} \to \overline{\mathbb{R}}:=\mathbb{R}\cup \{ - \infty, \infty \} $ is given by the equation 
\[
L_{w}(F)= \int_0^1 F^{-1}(u) w(u) \dd u , 
\]
where $w:(0,1)\to \mathbb{R}$ is a function. For illustrative ``actuarial'' examples where $L$-functionals arise, we refer to \citet{w98}, and \citet{jz03}, with the latter paper being perhaps the first one to connect actuarial risk measures with $L$-functionals for the sake of developing statistical inference for the risk measures. To see the role of $L$-functionals in the theory of distorted expectations, which are of particular interest when assessing risks in economics, finance and insurance, we refer to \citet{dklt2012}. Next is a specimen of additional examples that illustrate the $w$'s that arise in statistics, finance, and insurance.

\begin{example}\label{example-loc}
Let $\mathcal{G} \subset \mathcal{F}$ be the family of normal (i.e., Gaussian) distributions. To assess the scale parameter when location is known \citep[e.g.,][p.~269]{s1980}, the $L$-functional  $L_{w}:\mathcal{G} \to \mathbb{R}$ is used with the weight function 
\[
w(u)= \Phi^{-1}(u),  
\]
where $\Phi $ denotes the standard normal cdf. 
\end{example}

\begin{example}\label{example-sca}
Let $\mathcal{F} \subset \mathcal{F}$ be the family of logistic distributions. To assess the location parameter when scale is known\citep[e.g.,][p.~269]{s1980}, the $L$-functional  $L_w:\mathcal{F} \to \mathbb{R}$ is used with the weight function 
\[
w(u)= 6u(1-u). 
\]
\end{example}

\begin{example}\label{example-gmd}
The Gini Mean Difference (GMD) is the $L$-functional $L_w:\mathcal{F} \to \overline{\mathbb{R}}$ with the weight function  
\[
w(u)=4u-2.  
\]
For properties and manifold applications of the GMD in economics and finance, we refer to \citet{ys2013}. 
\end{example}

\begin{example}\label{example-tg}
The Tail Gini \citep[p.~74]{FWZ2017} is the $L$-functional $L_w:\mathcal{F} \to \overline{\mathbb{R}}$ with the weight function  
\[
 w(u)= \frac{\mathds{1}\{p\le u < 1\}}{(1-p)^2} \big( 4u-2(1+p) \big) , 
\]
where $p\in [0,1)$ is a parameter. 
\end{example}

\begin{example}\label{example-gs}
The Gini Shortfall \citep[p.~75]{FWZ2017} is the $L$-functional $L_w:\mathcal{F} \to \overline{\mathbb{R}}$ with the weight function 
\[
w(u)=\frac{\mathds{1}\{p\le u < 1\}}{(1-p)^2}
 \bigg(1-p+4\lambda \bigg(u-\frac{1+p}{2}\bigg)\bigg) , 
\]
where $p\in [0,1)$ and $\lambda \in  [0,\infty)$ are parameters. 
\end{example}

To give a flavour of the technical path that is often taken in the literature \citep[e.g.,][p.~265]{s1980} when deriving large-sample statistical properties of $L$-integrals, we write the equations 
\begin{align} 
L_{w}(G) -L_{w}(F)
&= \int_0^1 \Big( G^{-1}(u) - F^{-1}(u)\Big)  w(u) \dd u 
\notag 
\\
&=\int_{\mathbb{R}} \Big( K_w\big(F(x)\big) - K_w\big( G(x) \big) \Big)  \dd x  
\label{eq-kw}
\end{align} 
(see Appendix~\ref{proof-eq-kw} for technical details), where 
\[
K_w(t)=\int_{0}^{t} w(u) \dd u ,
\]
assuming that $w$ is integrable on $(0,1)$. Note that equation~\eqref{eq-kw} collapses into equation~\eqref{eq-0} when $w(u)\equiv 1$. With an estimator $F_n$ instead of $G$, we can now establish asymptotic properties of the empirical $L$-estimator $L_{w}(F_n)$ when $n\to \infty $, but this path requires some form of the Taylor expansion applied on the function $K_w$, which, in turn, imposes smoothness conditions on the weight function $w$. Although in some applications such conditions are satisfied, in many other applications they are not.

As a way out of the difficulty, we can successfully use our developed theory for integrated quantiles, and the path toward achieving this goal is analogous to  that given by equation~\eqref{eq-kw}, although instead of integrating the weight function $w$ as in the definition of $K_w$, we integrate the quantile function $ F^{-1}$. The following example offers an idea how this goal can be achieved. 

\begin{example}[a proof of concept]\label{ex-concept} 
For simplicity, let $F \in \mathcal{F}^{+}_{1}$, which means that integrals $\int_{p}^1 F^{-1}(u) \dd u $ are finite for every $p\in (0,1)$, and let the weight function $w$ be the difference $w_1-w_2$ of two non-decreasing, right-continuous, and non-negative functions $w_1$ and $w_2$. An example of such a weight function $w$ would be $w(u)=6 u(1-u)$ that appears in Example~\ref{example-sca}. Hence, 
\begin{align}
L_{w}(F) 
&= \int_0^1 F^{-1}(u) w_1(u) \dd u -  \int_0^1 F^{-1}(u) w_2(u) \dd u   
\notag 
\\
&= \int_{0}^{\infty} \bigg(\int_{w_1^{-1}(x)}^1 F^{-1}(u) \dd u \bigg) \dd x 
-\int_{0}^{\infty} \bigg(\int_{w_2^{-1}(x)}^1 F^{-1}(u) \dd u \bigg) \dd x ,   
\label{eq-L-integral}
\end{align}
with the technical details justifying the right-most equation given in Appendix~\ref{proof-L-integral}. Equation~\eqref{eq-L-integral} paves a path that connects the results of Section~\ref{sect-upper} with asymptotic properties of the difference $L_{w}(F_n)-L_{w}(F)$, where $F_n$, $n\in \mathbb{N}$, is a sequence of cdf's that approximate $F$ when $n$ increases. Furthermore, via the results of \citet{dklt2012}, the equation also facilitates the development of statistical inference for distorted expectations. 
\end{example}

The following theorem, whose proof is given in Appendix~\ref{L-theorem}, implements the idea of Example~\ref{ex-concept} under minimal assumptions on $F$ and $w$, thus enabling to tackle  $L$-functionals that arise from a whole spectrum of applications. To formulate the theorem, we use the notation 
\[
L_{w,a,b}(F)=\int_a^b F^{-1}(u) w(u) \dd u .   
\]
Note that $L_{w}(F)=L_{w,a,b}(F)$ when $a=0$ and $b=1$.

\begin{theorem}\label{fund-lemma-abc}
Let the weight function $w$ be non-decreasing and right-continuous on the interval $(a,b) \subseteq (0,1)$. Then, for every cdf $F$ for which $L_{w,a,b}(F)$ is finite, we have the representation 
\begin{multline}\label{fund-lemma-eq1abcabc}
L_{w,a,b}(F)
=-\int_{-\infty}^0\mathds{1}\{w(a)<x\} \Bigg( \int_{a}^{b\wedge w^{-1}(x)}F^{-1}(u) \dd u \Bigg) \dd x 
\\ 
+\int_{0}^{\infty}\mathds{1}\{x\le w(b) \} 
\Bigg( \int_{a \vee w^{-1}(x)}^b F^{-1}(u)  \dd u \Bigg) \dd x ,  
\end{multline}
where $w^{-1}$ is the left-continuous inverse of $w$.
\end{theorem}

We see from the right-hand side of equation~\eqref{fund-lemma-eq1abcabc} that when $a=0$ and $b=1$, the $L$-integral $L_{w,a,b}(F)$, which is equal to $L_{w}(F)$, can be expressed in terms of integrated quantiles under the assumption that the weight function $w$ is non-decreasing on the entire unit interval $(0,1)$, but this is a rather restrictive assumption in many applications. For this reason, we shall show next that all $w$'s of practical relevance can be expressed in terms of linear combinations of non-decreasing functions on various subintervals of $(0,1)$. This reduces $L_{w}(F)$ to a linear combinations of $L$-integrals $L_{w,a,b}(F)$ with various non-decreasing weight functions $w$ over various intervals $(a,b) \subset (0,1)$. To develop statistical inference for the  $L$-integrals $L_{w,a,b}(F)$, we use the equation 
\begin{multline*}
L_{w,a,b}(G)-L_{w,a,b}(F) 
=-\int_{-\infty}^0\mathds{1}\{w(a)<x\} \Bigg( \int_{a}^{b\wedge w^{-1}(x)} 
\big( G^{-1}(u) - F^{-1}(u) \big)\dd u \Bigg) \dd x 
\\ 
+\int_{0}^{\infty}\mathds{1}\{x\le w(b) \} 
\Bigg( \int_{a \vee w^{-1}(x)}^b \big( G^{-1}(u) - F^{-1}(u) \big) \dd u \Bigg) \dd x 
\end{multline*} 
that follows from Theorem~\ref{fund-lemma-abc}. Using appropriate results from Sections~\ref{sect-upper}, \ref{sect-lower}, and \ref{sect-middle}, we can now convert integrals of $G^{-1}- F^{-1}$ into integrals of $F-G$ for which asymptotics with $F_n$ instead of $G$ becomes straightforward.

\begin{example}\label{example-28}
Consider the weight function 
\[
w(u)=6 u(1-u), 
\]
which we have encountered in Example~\ref{example-sca}. Obviously, we can decompose this weight function as the difference of $6u$ and $6u^2$ on the entire interval $(0,1)$. Alternatively, we can decompose it as the difference of $-6(1-u)^2 $ and $-6(1-u)$ on the entire interval $(0,1)$. However, if we adopt the first decomposition, then we need to assume the finiteness of the moment $\mathbb{E}(X^{+})$, and if we adopt the second decomposition, then we need to assume the finiteness of the moment $\mathbb{E}(X^{-})$. Either way, superfluous conditions arise, which are not required for the finiteness of the $L$-integral $L_{w}(F)$ because the original weight function $w(u)$ converges to $0$ when $u \downarrow 0$ and $u \uparrow 1$. For this reason, to reduce the problem to non-decreasing functions, we split the interval $[0,1)$ into two parts: $[0,1/2)$ where the function $w$ is increasing, and $[1/2,1)$ where the function $w$ is decreasing and thus $-w$ is increasing. Hence, we can express $w$ as the sum of two differences of non-decreasing functions, and this can be done in several ways. One way is to use the equation  
\begin{equation}\label{decomp-w}
w(u)
=\mathds{1}\{0\le u <1/2\}\Big( w_{11}(u) - w_{12}(u) \Big)  
+ \mathds{1}\{1/2\le u < 1\}\Big( w_{21}(u) - w_{22}(u) \Big)  
\end{equation}
with the functions  
\begin{gather*}
w_{11}(u) = w(u) \quad \textrm{and} \quad  w_{12}(u) = 0,  
\\  
w_{21}(u)= 0 \quad \textrm{and} \quad w_{22}(u) =-w(u).  
\end{gather*}
Another way is to use equation~\eqref{decomp-w} with the functions 
\begin{gather*}
w_{11}(u) = 6u \quad \textrm{and} \quad  w_{12}(u) = 6u^2,  
\\  
w_{21}(u)= -6(1-u)^2 \quad \textrm{and} \quad w_{22}(u) =-6(1-u).  
\end{gather*}
Note that irrespective of the way we use, we always have the bounds 
\begin{gather*}
|w_{11}(u)| \vee |w_{12}(u)|  \le c |w(u)| \quad \textrm{for all} \quad  u\in [0,1/2),  
\\  
|w_{21}(u)| \vee |w_{22}(u)|   \le c |w(u)|\quad \textrm{for all} \quad  u\in [1/2,1), 
\end{gather*}
with some positive constant $c<\infty $. Therefore, no superfluous moment conditions arise when developing statistical inference for the $L$-integrals based on the four functions $w_{ij}$. This concludes Example~\ref{example-28}. 
\end{example}

Example~\ref{example-28} naturally leads to the following condition, which plays a fundamental role in connecting statistical inference for the $L$-integral $L_{w}(F)$ with that for integrated quantiles. 

\begin{enumerate}[label={\rm(C$_{w}$)}]
\item\label{cond-w-diff}
Let the weight function $w:(0,1)\to \mathbb{R} $ be such that, for some integer $K\ge 1$, there is a partition $0=a_0<a_1<\cdots < a_K=1$ of the unit interval $(0,1)$ and also pairs $(w_{11},w_{12}), \dots, (w_{K1},w_{K2})$ of non-decreasing and right-continuous on $(0,1)$ functions such that, for a constant $c<\infty $, 
\begin{gather}
|w_{11}(u)| \vee |w_{12}(u)|  \le c |w(u)| \quad \textrm{in a neighbourhood of $0$} ,  
\label{ccc-1a}
\\  
|w_{K1}(u)| \vee |w_{K2}(u)|  \le c |w(u)|\quad \textrm{in a neighbourhood of $1$} , 
\label{ccc-1b} 
\end{gather}
and 
\begin{equation}\label{ccc-1}
w(u)
=\sum_{k=1}^K \mathds{1}\{a_{k-1}\le u < a_k\}\Big( w_{k1}(u) - w_{k2}(u) \Big)  . 
\end{equation}
\end{enumerate}

To appreciate condition~\ref{cond-w-diff}, which may look somewhat complex at first sight, we observe that if the function $w:(0,1)\to \mathbb{R} $ is piecewise monotonic, then the assumption is satisfied. All  the weight functions that we have encountered in the literature are piecewise monotonic. In the very special case when $w$ is monotonic (i.e., either non-decreasing or non-increasing) on the entire domain $(0,1)$ of its definition, condition~\ref{cond-w-diff} is satisfied with $K=1$. The following example illustrates the case $K=2$. Interestingly, we shall see from following Theorem~\ref{min-k}, whose proof is constructive and given in Appendix~\ref{min-w}, that if condition~\ref{cond-w-diff} holds with $K\ge 3$, then the condition can be reduced to the case $K=2$.

\begin{example} \label{ex-31}
Consider the weight function  
\[
w(u)=(1-u)^{\beta }
\]
defined for all $u\in [0,1)$ and parameterized by $\beta \in \mathbb{R}$. The function can be expressed as the difference of two non-decreasing functions as follows: 
\[
w(u)=\underbrace{\mathds{1}\{ \beta \le 0\}(1-u)^{\beta }}_{w_1(u)} - 
\underbrace{(-1)\mathds{1}\{ \beta >0 \}(1-u)^{\beta }}_{w_2(u)} . 
\]
The two functions $w_1$ and $w_2$ are continuous and non-decreasing, and satisfy the bound 
\[
|w_{1}(u)| \vee |w_{2}(u)|  \le |w(u)| 
\] 
for all $u\in (0,1)$. That is, for the weight function $w$ we do not need to split the interval $(0,1)$ into parts, and we can therefore set $K=1$. 
This concludes Example~\ref{ex-31}.  
\end{example}

\begin{theorem}\label{min-k} 
If condition~\ref{cond-w-diff} is satisfied with $K\ge 3$, then it is also satisfied with $K=2$. 
\end{theorem}

Theorem~\ref{min-k}, whose proof is given in Appendix~\ref{min-w}, says that there is a constant $a\in [0,1]$ and two pairs $(w_{11},w_{12})$ and $(w_{21},w_{22})$ of non-decreasing and right-continuous functions such that bounds~\eqref{ccc-1a} and~\eqref{ccc-1b} are satisfied and the representation 
\begin{equation}\label{ccc-abc}
w(u)
=\mathds{1}\{0\le u < a\}\Big( w_{11}(u) - w_{12}(u) \Big) 
+ \mathds{1}\{a\le u < 1\}\Big( w_{21}(u) - w_{22}(u) \Big)  
\end{equation}
holds. Earlier in Example~\ref{example-28}, we showed that the case $K=2$ cannot be generally reduced to $K=1$. In other words, from the theoretical point of view, it is sufficient to consider only the cases $K=1$ and $K=2$. From the practical point of view, this minimality property could sometimes be time consuming to realize, but this is not important: we can always work with any $K\ge 1$ that we can easily obtain.

\section{Conclusion}
\label{conclusion}

We have developed a general statistical inference theory for integrals of quantiles. The required conditions are formulated in terms of general approximating sequences of cdf's $F_n$, which are not attached to any specific sampling design or dependence structure. To maximally illuminate the developed theory, we have illustrated the obtained results using simple random sampling, due to everyone's familiarity with this sampling design and basic statistical results. In particular, we have discussed consistency, bias, and asymptotic normality of various integrals of quantiles and their combinations that arise in risk and economic-inequality measurements: the upside and downside tail-values-at-risk ($\ATVaR$ and $\FTVaR$, respectively), and the Lorenz and Gini curves. Furthermore, we have illustrated how our general results can be used in the case of time series data. To facilitate the appreciation of theoretical results, we have illustrated some of them numerically. Finally, we have linked the herein developed theory for integrated quantiles with topics such as $L$-integrals.

\appendix 

\section{Technicalities} 
\label{appendix}

We have subdivided this appendix into several major parts: Appendix~\ref{appendix-1} contains auxiliary lemmas that we shall need in Appendix~\ref{appendix-2} for proving  Theorems~\ref{theorem-0upper}, \ref{theorem-0lower}, and~\ref{theorem-0middle}. 
In Appendix~\ref{subsub-asresults} we shall derive asymptotic expansions for the empirical Lorenz and Gini curves that are needed for establishing their asymptotic normality stated in Examples~\ref{illustration-3} and~\ref{illustration-4}.  
Appendices~\ref{L-theorem} and~\ref{min-w} contain proofs of Theorems~\ref{fund-lemma-abc}
and~\ref{min-k}, respectively. In Appendix~\ref{proof-a1} we prove several cursory notes made earlier in the paper when introducing and discussing the main results.

\subsection{Auxiliary lemmas}
\label{appendix-1}

In this appendix we shall derive formulas for, and establish finiteness of, the first and second moments of the random variables $Y_{\one}(p)$ and $Y_{\zero}(p)$. We shall need these results in Appendix~\ref{appendix-2}. The use of Fubini's theorem and distributional equations $X \stackrel{d}{=} F^{-1}(U)$ and $Y \stackrel{d}{=} G^{-1}(U)$ will permeate our considerations, where $U$ is a uniform on $[0,1]$ random variable and $\stackrel{d}{=}$ denotes equality in distribution.

\begin{lemma}\label{lemma-1upper} 
If $F\in \mathcal{F}^{+}_{1}$, then the random variable $Y_{\one}(p)$ has a finite first moment, and the following equations hold:  
\begin{align}
\mathbb{E}\big( Y_{\one}(p) \big) 
&=\int^{\infty }_{F^{-1}(p)}  \big( 1- F(x)\big) \dd x 
\label{lemma-upper-eq1}
\\
&=\int_p^1 F^{-1}(u)\dd u -(1-p) F^{-1}(p) . 
\label{lemma-upper-eq2}
\end{align}
If $F\in \mathcal{F}^{+}_{2}$, then $Y_{\one}(p)$ has a finite second moment, and  formula~\eqref{normality-upper-var2} holds for the variance of $Y_{\one}(p)$.   
\end{lemma}

\begin{proof} 
We start with the equations 
\begin{align}
Y_{\one}(p)  
&=\big( X-F^{-1}(p) \big)^{+} 
\label{lemma-upper-eq3a}
\\
&\stackrel{d}{=}  \big( F^{-1}(U)-F^{-1}(p) \big)^{+}.  
\label{lemma-upper-eq3b}
\end{align}
The first moment of the random variable on the right-hand side of equation~\eqref{lemma-upper-eq3b} is finite whenever $F\in \mathcal{F}^{+}_{1}$. We can now use Fubini's theorem to interchange the expectation and integration operations on the right-hand side of the equation 
\[
\mathbb{E}\big(Y_{\one}(p) \big)=\mathbb{E} \bigg( \int_{F^{-1}(p)}^{\infty} \mathds{1}\{X> x\} \dd x \bigg) 
\]
and in this way arrive at equation~\eqref{lemma-upper-eq1} because 
$\mathbb{E}\big( \mathds{1}\{X>x\} \big) =1-F(x)$.  
To prove equation~\eqref{lemma-upper-eq2}, we write 
\begin{align*} 
\int_{F^{-1}(p)}^{\infty} \big( 1-F(x)) \dd x 
&= \mathbb{E} \bigg( \int_{F^{-1}(p)}^{\infty} \mathds{1}\{X> x\} \dd x \bigg) 
\\
&= \mathbb{E} \Big( \big( F^{-1}(U)-F^{-1}(p)  \big)^{+}\Big)  
\\ 
&= \int_{0}^1 \big( F^{-1}(s)-F^{-1}(p)\big)^{+}\dd s   
\\ 
&= \int_{p}^1 \big( F^{-1}(s)-F^{-1}(p)\big)\dd s   
\\ 
&= \int_{p}^1  F^{-1}(s)\dd s   -(1-p)F^{-1}(p).  
\end{align*}

To prove the second half of the lemma, we first note that in view of equation~\eqref{lemma-upper-eq3b}, the random variable $Y_{\one}(p)$ has a finite second moment whenever $F\in \mathcal{F}^{+}_{2}$. To prove equation~\eqref{normality-upper-var2}, we use 
Fubini's theorem and have 
\begin{align*} 
\Var\left( Y_{\one}(p) \right) 
&= \Var\left( \int^{\infty }_{F^{-1}(p)}  \mathds{1}\{X>x\} \dd x  \right)  
\\
&= \mathbb{E} \left( \bigg(\int^{\infty }_{F^{-1}(p)}  \big( \mathds{1}\{X>x\}-(1-F(x))\big) \dd x  \bigg)^2\right) 
\\
&= \int^{\infty }_{F^{-1}(p)}\int^{\infty }_{F^{-1}(p)} 
\mathbb{E}\Big( \big( \mathds{1}\{X\le x\} - F(x)\big)\big( \mathds{1}\{X\le y\} - F(y)\big)\Big) \dd x  \dd y 
\\
&= 
\int^{\infty }_{F^{-1}(p)} \int^{\infty }_{F^{-1}(p)} \big( F(x \wedge y) -F(x)F(y) \big) \dd x \dd y. 
\end{align*}
This completes the proof of Lemma~\ref{lemma-1upper}. 
\end{proof}

\begin{lemma}\label{lemma-1lower} 
If $F\in \mathcal{F}^{-}_{1}$, then the random variable $Y_{\zero}(p)$ has a finite first moment, and the following equations hold:  
\begin{align}
\mathbb{E}\big( Y_{\zero}(p) \big) 
&=\int_{-\infty }^{F^{-1}(p)}  F(x) \dd x 
\label{lemma-lower-eq1}
\\
&= p F^{-1}(p)-\int_0^p F^{-1}(u)\dd u. 
\label{lemma-lower-eq2}
\end{align}
If $F\in \mathcal{F}^{-}_{2}$, then $Y_{\zero}(p)$ has a finite second moment,  and formula~\eqref{normality-lower-var2} holds for the variance of  $Y_{\zero}(p)$.   
\end{lemma}

\begin{proof} 
We start with the equations 
\begin{align}
Y_{\zero}(p)  
&=\big( F^{-1}(p)-X \big)^{+} 
\label{lemma-lower-eq3a}
\\
&\stackrel{d}{=}  \big( F^{-1}(p)-F^{-1}(U) \big)^{+}. 
\label{lemma-lower-eq3b}
\end{align}
The first moment of the random variable on the right-hand side of equation~\eqref{lemma-lower-eq3b} is finite whenever $F\in \mathcal{F}^{-}_{1}$. We can now use Fubini's theorem to interchange the expectation and integration operations on the right-hand side of the equation 
\[
\mathbb{E}\big( Y_{\zero}(p) \big) = \mathbb{E} \bigg( \int_{-\infty }^{F^{-1}(p)} \mathds{1}\{X\le x\} \dd x \bigg) 
\]
and in this way arrive at equation~\eqref{lemma-lower-eq1} because 
$\mathbb{E}\big( \mathds{1}\{X\le x\} \big) =F(x)$.  
To prove equation~\eqref{lemma-lower-eq2}, we write 
\begin{align*}
\int_{-\infty }^{F^{-1}(p)}  F(x) \dd x 
&= \mathbb{E} \bigg( \int_{-\infty }^{F^{-1}(p)} \mathds{1}\{X\le x\} \dd x \bigg) 
\\
&= \mathbb{E} \Big( \big( F^{-1}(p) -  F^{-1}(U) \big)^{+}\Big)  
\\
&= \int_0^1 \big( F^{-1}(p) - F^{-1}(u) \big)^{+}\dd u 
\\
&= \int_0^p \big( F^{-1}(p) - F^{-1}(u) \big)\dd u 
\\
&= p F^{-1}(p) -\int_0^p F^{-1}(u)\dd u .  
\end{align*}

To prove the second half of the lemma, we first note that in view of equation~\eqref{lemma-lower-eq3b}, the random variable $Y_{\zero}(p)$ has a finite second moment whenever $F\in \mathcal{F}^{-}_{2}$.  To prove  equation~\eqref{normality-lower-var2}, we use 
Fubini's theorem and have 
\begin{align*} 
\Var\left( Y_{\zero}(p) \right) 
&= \Var\left( \int_{-\infty }^{F^{-1}(p)}  \mathds{1}\{X\le x\} \dd x  \right)  
\\
&= \mathbb{E} \left( \bigg(\int_{-\infty }^{F^{-1}(p)}  \big( \mathds{1}\{X \le x\}-F(x)\big) \dd x  \bigg)^2\right) 
\\
&= \int_{-\infty }^{F^{-1}(p)} \int_{-\infty }^{F^{-1}(p)}  \mathbb{E}\Big( \big( \mathds{1}\{X\le x\} - F(x)\big)\big( \mathds{1}\{X\le y\} - F(y)\big)\Big) \dd x  \dd y 
\\
&= 
\int_{-\infty }^{F^{-1}(p)} \int_{-\infty }^{F^{-1}(p)} 
 \big( F(x \wedge y) -F(x)F(y) \big) \dd x  \dd y. 
\end{align*}
This completes the proof of Lemma~\ref{lemma-1lower}. 
\end{proof}

\begin{note}\label{note-5}
The asymptotic variance $\sigma_{\half}^2(p_1,p_2)$ in statement~\eqref{normality-middle} is the variance of the random variable $Y_{\half}(p_1,p_2)$ defined by equation~\eqref{rv-middle}. Since both $F^{-1}(p_1)$ and $F^{-1}(p_2)$ are finite, the random variable $Y_{\half}(p_1,p_2)$ is bounded and thus has finite moments of all degrees irrespective of the cdf $F$. Note the equations  
\begin{align*}
Y_{\half}(p_1,p_2)
&=\int_{-\infty}^{F^{-1}(p_2)}  \mathds{1}\{X\le x\} \dd x   
- \int_{-\infty}^{F^{-1}(p_1)}  \mathds{1}\{X\le x\} \dd x   
\\
&= \big( F^{-1}(p_2)-X \big)^{+} - \big( F^{-1}(p_1)-X \big)^{+}. 
\end{align*}
This difference of the two positive parts, which may not have finite variances individually, always has a finite variance (as well as all the moments of higher degrees) irrespective of the cdf $F$, because $0<p_1<p_2<1$. Equation~\eqref{normality-middle-var2} follows immediately from  
\begin{multline*}
\Big( Y_{\half}(p_1,p_2)-\mathbb{E}\big( Y_{\half}(p_1,p_2) \big) \Big)^2
\\
= \int_{F^{-1}(p_1)}^{F^{-1}(p_2)} \int_{F^{-1}(p_1)}^{F^{-1}(p_2)}
\big( \mathds{1}\{X\le x\} - F(x)\big)\big( \mathds{1}\{X\le y\} - F(y)\big) \dd x \dd y 
\end{multline*}
and Fubini's theorem. This concludes Note~\ref{note-5}. 
\end{note}

\subsection{Proofs of Theorems~\ref{theorem-0upper}, \ref{theorem-0lower}, and~\ref{theorem-0middle}}
\label{appendix-2}

The three theorems deal with arbitrary cdf's $F$ and $G$ that, at most, are allowed to satisfy only moment conditions, with no further assumptions. It should be noted at the outset that we may find some portions of the following proofs look like simple consequences of integration-by-parts and change-of-variable  formulas of calculus, but this would be mathematically inaccurate because, strictly speaking, the quantile function  $F^{-1}$ is not an ordinary inverse of the cdf $F$, unless the latter is strictly increasing and continuous, which is not necessarily the case in the context of the present paper. Indeed, keeping in mind that our general results need to include, as a special case, the empirical cdf $F_{n,\srs}$, which is a step-wise function and thus has flat regions as well as jumps, good care  (recall \citet{w2023}) needs to be taken to maintain the largest possible class of cdf's. Hence, only moment-type assumptions are required, as specified in the formulation of theorems.

\begin{proof}[Proof of Theorem~\ref{theorem-0upper}]
By Lemma~\ref{lemma-1upper}, we have 
\begin{equation}\label{equation-0upper}
\int_p^1 F^{-1}(u)\dd u =(1-p) F^{-1}(p) + \int^{\infty }_{F^{-1}(p)}  \big( 1- F(x)\big) \dd x .   
\end{equation}
Since  the same equation also holds with $G$ instead of $F$, taking the difference between the two equations gives 
\begin{multline}
\int_p^1 \big( G^{-1}(u)-F^{-1}(u) \big) \dd u 
=(1-p) \big( G^{-1}(p)-F^{-1}(p) \big)
\\
+\int^{\infty }_{G^{-1}(p)}  \big( 1- G(x)\big) \dd x 
- \int^{\infty }_{F^{-1}(p)}  \big( 1- F(x)\big) \dd x . 
\label{qq-1upper}
\end{multline}
Next we express the penultimate integral on the right-hand side of equation~\eqref{qq-1upper} as follows:  
\begin{equation}\label{qq-2upper}
\int^{\infty }_{G^{-1}(p)}  \big( 1- G(x)\big) \dd x 
=  \int^{F^{-1}(p)}_{G^{-1}(p)}  \big( 1- G(x)\big) \dd x 
+ \int^{\infty }_{F^{-1}(p)}  \big( 1- G(x)\big) \dd x , 
\end{equation}
where the first integral on the right-hand side is equal to $-\int_{F^{-1}(p)}^{G^{-1}(p)}  \big( 1- G(x)\big) \dd x $ when $F^{-1}(p)<G^{-1}(p)$. Combining equations~\eqref{qq-1upper} and~\eqref{qq-2upper}, we arrive at equation~\eqref{eq-1upper} with the remainder term 
\[
\Rem(p;F,G)= (1-p) \big( F^{-1}(p)-G^{-1}(p) \big) 
- \int^{F^{-1}(p)}_{G^{-1}(p)}  \big( 1- G(x)\big) \dd x 
\]
that can be concisely written as the integral on the right-hand side of equation~\eqref{qq-3upper}. 

To show that $\Rem(p;F,G)$ is non-negative, we proceed as follows. If $F^{-1}(p) \ge G^{-1}(p)$, then $x\ge G^{-1}(p)$, which is equivalent to $G(x)\ge p$, and so 
\[
\Rem(p;F,G)= \int^{F^{-1}(p)}_{G^{-1}(p)}  \big( G(x)-p\big) \dd x \ge 0. 
\]
If $F^{-1}(p) < G^{-1}(p)$, then $x<G^{-1}(p)$, which is equivalent to $G(x)<p$, and so 
\[
\Rem(p;F,G)=-\int^{G^{-1}(p)}_{F^{-1}(p)}  \big( G(x)-p\big) \dd x \ge 0. 
\] 
Hence, we always have $\Rem(p;F,G)\ge 0$. 

Furthermore,
\begin{align*}
\Rem(p;F,G) &= \int^{F^{-1}(p)}_{G^{-1}(p)}  \big( G(x) - F(x)\big) \dd x  
+\int^{F^{-1}(p)}_{G^{-1}(p)}  \big( F(x) -p\big) \dd x  
\\
&= \int^{F^{-1}(p)}_{G^{-1}(p)}  \big( G(x) - F(x) \big) \dd x -\Rem(p;G,F) 
\\
&\le \int^{F^{-1}(p)}_{G^{-1}(p)}  \big( G(x) - F(x) \big) \dd x , 
\end{align*}
with the last inequality holding because $\Rem(p;G,F)\ge 0$. This establishes upper  bound~\eqref{r-prop-0upper}. Bound~\eqref{r-prop-1upper} follows immediately.
\end{proof}

\begin{proof}[Proof of Theorem~\ref{theorem-0lower}]
By Lemma~\ref{lemma-1lower}, we have the equation 
\begin{equation}\label{equation-0lower}
\int_0^p F^{-1}(u)\dd u=p F^{-1}(p)-\int_{-\infty }^{F^{-1}(p)}  F(x) \dd x .   
\end{equation}
Since  the same equation also holds with $G$ instead of $F$, taking the difference between the two equations gives 
\begin{equation}
\int_p^1 \big( G^{-1}(u)-F^{-1}(u) \big) \dd u 
=p \big( G^{-1}(p)-F^{-1}(p) \big)
-\int_{-\infty }^{G^{-1}(p)}  G(x) \dd x 
+ \int_{-\infty }^{F^{-1}(p)}  F(x) \dd x . 
\label{qq-1lower}
\end{equation}
Next we express the penultimate integral on the right-hand side of equation~\eqref{qq-1lower} as follows:  
\begin{equation}\label{qq-2lower}
\int_{-\infty }^{G^{-1}(p)}  G(x) \dd x 
=  \int_{-\infty }^{F^{-1}(p)}  G(x) \dd x 
+\int_{F^{-1}(p)}^{G^{-1}(p)}  G(x) \dd x , 
\end{equation}
where the second integral on the right-hand side is equal to $-\int^{F^{-1}(p)}_{G^{-1}(p)}  G(x) \dd x $ when $F^{-1}(p)>G^{-1}(p)$. Combining equations~\eqref{qq-1lower} and~\eqref{qq-2lower}, we arrive at equation~\eqref{eq-1lower} with the remainder term 
\begin{align*}
\Rem(p;F,G)
&= p \big( G^{-1}(p)-F^{-1}(p) \big) 
- \int_{F^{-1}(p)}^{G^{-1}(p)}  G(x) \dd x  
\\
&= -p \big( F^{-1}(p)-G^{-1}(p) \big) 
+ \int^{F^{-1}(p)}_{G^{-1}(p)}  G(x) \dd x , 
\end{align*}
that can be concisely written as the integral on the right-hand side of equation~\eqref{qq-3upper}. 
This completes the proof of Theorem~\ref{theorem-0lower}.
\end{proof}

\begin{proof}[Proof of Theorem~\ref{theorem-0middle}] 
We start with the equation 
\begin{equation}\label{equation-0middle}
\int_{p_1}^{p_2}  F^{-1}(u) \dd u  
=p_2 F^{-1}(p_2)- p_1 F^{-1}(p_1) - \int_{F^{-1}(p_1)}^{F^{-1}(p_2)}  F(x) \dd x ,  
\end{equation}
which can be verified in many different ways, and whose validity does not require any condition on the cdf $F$.  For example, we can prove equation~\eqref{equation-0middle} as follows.

Since $F^{-1}(p_1)$ and $F^{-1}(p_2)$ are finite, the integral $\int_{F^{-1}(p_1)}^{F^{-1}(p_2)}  F(x) \dd x $ is finite. With $X$ denoting a random variable whose cdf is $F$, and also noting that $X$ is equal in distribution to $F^{-1}(U)$, where $U$ is a uniform on the interval $[0,1]$ random variable, we have the  equations 
\begin{align*}
\int_{F^{-1}(p_1)}^{F^{-1}(p_2)}  F(x) \dd x 
&= \mathbb{E} \bigg( \int_{F^{-1}(p_1)}^{F^{-1}(p_2)} \mathds{1}\{X\le x\} \dd x \bigg) 
\notag 
\\
&= \mathbb{E} \Big( \big( F^{-1}(p_2) - F^{-1}(p_1) \vee X \big)^{+}\Big) 
\notag 
\\
&= \mathbb{E} \Big( \big( F^{-1}(p_2) - F^{-1}(p_1) \vee F^{-1}(U) \big)^{+}\Big) 
\notag 
\\
&= \int_0^1 \big( F^{-1}(p_2) - F^{-1}(p_1\vee u) \big)^{+}\dd u . 
\end{align*}
Splitting the integral $\int_0^1 $ on the right-hand side into the sum of $ \int_0^{p_1}$, $\int_{p_1}^{p_2}$ and $\int_{p_2}^{1}$, we obtain 
\begin{align*}
\int_{F^{-1}(p_1)}^{F^{-1}(p_2)}  F(x) \dd x 
&= \int_0^{p_1} \big( F^{-1}(p_2) - F^{-1}(p_1) \big)\dd u 
+\int_{p_1}^{p_2} \big( F^{-1}(p_2) - F^{-1}(u) \big)\dd u 
\notag 
\\
&= p_1 \big( F^{-1}(p_2) - F^{-1}(p_1) \big) +(p_2-p_1)F^{-1}(p_2) -\int_{p_1}^{p_2} F^{-1}(u)\dd u 
\notag 
\\
&= p_2 F^{-1}(p_2) -p_1 F^{-1}(p_1) -\int_{p_1}^{p_2} F^{-1}(u)\dd u .  
\end{align*}
Equation~\eqref{equation-0middle} follows. Of course, the same equation also holds with $G$ instead of $F$. Hence, 
\begin{multline*}
\int_{p_1}^{p_2}  \big( G^{-1}(u) - F^{-1}(u)\big) \dd u  
=p_2\big( G^{-1}(p_2)- F^{-1}(p_2) \big) -p_1 \big( G^{-1}(p_1)- F^{-1}(p_1) \big) 
\\
\quad -\bigg( \int_{G^{-1}(p_1)}^{G^{-1}(p_2)}  G(x) \dd x  -\int_{F^{-1}(p_1)}^{F^{-1}(p_2)}  F(x) \dd x \bigg),  
\end{multline*}
which can be rewritten as 
\begin{align*}
\int_{p_1}^{p_2}  \big( G^{-1}(u) - F^{-1}(u)\big) \dd u  
&= p_2\big( G^{-1}(p_2)- F^{-1}(p_2) \big) -p_1 \big( G^{-1}(p_1)- F^{-1}(p_1) \big) 
\\
&\quad -\int_{F^{-1}(p_1)}^{F^{-1}(p_2)}  \big( G(x) - F(x)\big) \dd x 
\\
&\qquad - \bigg( \int_{G^{-1}(p_1)}^{G^{-1}(p_2)}  G(x)\dd x 
-\int_{F^{-1}(p_1)}^{F^{-1}(p_2)} G(x) \dd x \bigg)  . 
\end{align*}
Since 
\[
 \int_{G^{-1}(p_1)}^{G^{-1}(p_2)}  G(x)\dd x 
-\int_{F^{-1}(p_1)}^{F^{-1}(p_2)} G(x) \dd x  
=\int_{G^{-1}(p_1)}^{F^{-1}(p_1)}  G(x)\dd x 
-\int_{G^{-1}(p_2)}^{F^{-1}(p_2)} G(x) \dd x , 
\]
we have 
\begin{align*}
\int_{p_1}^{p_2}  \big( G^{-1}(u) - F^{-1}(u)\big) \dd u  
&= p_2\big( G^{-1}(p_2)- F^{-1}(p_2) \big) -p_1 \big( G^{-1}(p_1)- F^{-1}(p_1) \big) 
\notag 
\\
&\quad -\int_{F^{-1}(p_1)}^{F^{-1}(p_2)}  \big( G(x) - F(x)\big) \dd x 
\notag 
\\
&\qquad - \int_{G^{-1}(p_1)}^{F^{-1}(p_1)}  G(x)\dd x 
+\int_{G^{-1}(p_2)}^{F^{-1}(p_2)} G(x) \dd x  .   
\end{align*}
Combining the terms on the right-hand side of the latter equation, we obtain 
\begin{multline*}
\int_{p_1}^{p_2}  \big( G^{-1}(u) - F^{-1}(u)\big) \dd u  
=-\int_{F^{-1}(p_1)}^{F^{-1}(p_2)}  \big( G(x) - F(x)\big) \dd x 
\\
+ \int_{G^{-1}(p_1)}^{F^{-1}(p_1)} \big(p_1- G(x) \big) \dd x
-  \int_{G^{-1}(p_2)}^{F^{-1}(p_2)} \big( p_2-G(x)\big) \dd x .    
\end{multline*}
This establishes equation~\eqref{eq-1middle} and completes the proof of Theorem~\ref{theorem-0middle}.  
\end{proof}

\subsection{Asymptotic results}
\label{subsub-asresults}

\begin{lemma}\label{LC-clt}
Let $F\in \mathcal{F}^{\mu\neq 0}_{2}$, and let the quantile function $F^{-1}$ be continuous at the point $p\in (0,1)$.  Then 
\begin{equation}\label{eq-LC-1a}
\sqrt{n}\left(\LC_{n,\srs}(p) - \LC(p)  \right) 
= -{1\over \sqrt{n}} \sum_{i=1}^n  Y_{i,\LC}(p) + o_{\mathbb{P}}(1), 
\end{equation} 
where 
\[
Y_{i,\LC}(p)
= {1\over \mu } \int_{-\infty }^{F^{-1}(p)} \big( \mathds{1}\{X_i\le x\} -F(x)\big) \dd x  
+{\LC(p) \over \mu } \big( X_i-\mu \big) . 
\]
\end{lemma}

\begin{proof} 
We start with the generic equation 
\begin{equation} \label{eq-LC-2}
{\theta_n\over \mu_n} -{\theta \over \mu } 
= \underbrace{{1\over \mu}(\theta_n-\theta)-{\theta \over \mu^2}(\mu_n-\mu)}_{\textrm{main term}} 
+ \underbrace{{\theta \over \mu^2 \mu_n}(\mu_n-\mu)^2
-{1\over \mu\mu_n}(\theta_n-\theta)(\mu_n-\mu)}_{\textrm{remainder term}},  
\end{equation}
which we next use with the following quantities: 
\begin{itemize} 
\item 
$\mu $ is the mean of $F$, 
\item 
$\theta $ is the lower-layer integral defined by equation~\eqref{int-lower}, 
\item 
$\mu_n$ is the arithmetic mean $\mu_{n,\srs}$  given by equation~\eqref{mean-1}, 
\item 
$\theta_n$ is the estimator $\theta_{n,\srs}=\int_0^p F_{n,\srs}^{-1}(u)\dd u $ of $\theta$. 
\end{itemize}
By the classical central limit theorem, we have 
\[
\sqrt{n}(\mu_{n,\srs}-\mu)=O_{\mathbb{P}}(1). 
\]
Statement~\eqref{ill-FTVaR-3} implies 
\[
\sqrt{n}(\theta_{n,\srs}-\theta)=O_{\mathbb{P}}(1). 
\]
Consequently, the remainder term on the right-hand side of equation~\eqref{eq-LC-2} is of the order $ o_{\mathbb{P}}(1/\sqrt{n})$. Hence, 
\begin{align}
\sqrt{n}\left(\LC_{n,\srs}(p) - \LC(p)  \right) 
&=\sqrt{n}\left({\theta_{n,\srs}\over \mu_{n,\srs}} -{\theta \over \mu }\right)
\notag 
\\
&= {1\over \mu}\sqrt{n}(\theta_{n,\srs}-\theta)
-{\theta \over \mu^2}\sqrt{n}(\mu_{n,\srs}-\mu)  + o_{\mathbb{P}}(1) .  
\label{eq-LC-6}
\end{align} 
By equation~\eqref{eq-1da} with $F_{n,\srs}$ instead of $F_n$, we have 
\begin{equation}\label{eq-LC-7}
\sqrt{n}(\theta_{n,\srs}-\theta)
=\sqrt{n}\,\int_{-\infty }^{F^{-1}(p)}  \big( F(x)- F_{n,\srs}(x)\big) \dd x +o_{\mathbb{P}}(1). 
\end{equation}
Combining equations~\eqref{eq-LC-6} and~\eqref{eq-LC-7}, we obtain 
\[
\sqrt{n}\left(\LC_{n,\srs}(p) - \LC(p)  \right) 
= -{1\over \mu}\sqrt{n}\,\int_{-\infty }^{F^{-1}(p)}  \big( F_{n,\srs}(x)-F(x)\big) \dd x 
-{\theta \over \mu^2}\sqrt{n}(\mu_{n,\srs}-\mu)  + o_{\mathbb{P}}(1), 
\]
which implies statement~\eqref{eq-LC-1a} and concludes the proof of Lemma~\ref{LC-clt}. 
\end{proof}

\begin{lemma}\label{GC-clt}
Let $F\in \mathcal{F}^{\mu\neq 0}_{2}$, and let the quantile function $F^{-1}$ be continuous at the points $p\in (0,1)$ and $1-p$. Then 
\[
\sqrt{n}\left(\GC_{n,\srs}(p) - \GC(p)  \right) 
= {1\over \sqrt{n}} \sum_{i=1}^n  \Big( Y_{i,\LC}(1-p)+ Y_{i,\LC}(p)\Big) + o_{\mathbb{P}}(1), 
\]
where
\[
Y_{i,\LC}(u)
= {1\over \mu } \int_{-\infty }^{F^{-1}(u)} \big( \mathds{1}\{X_i\le x\} -F(x)\big) \dd x  
+{\LC(u) \over \mu } \big( X_i-\mu \big) . 
\]
\end{lemma}

\begin{proof} 
Note that 
\begin{align*}
\GC(p) 
&=1- {1\over \mu } 
\left( \int_{0}^{1-p} F^{-1}(u)\dd u + \int_0^p F^{-1}(u)\dd u \right) . 
\\
&= 1-\LC(1-p)-\LC(p). 
\end{align*}
Hence, 
\[
\sqrt{n}\left(\GC_{n,\srs}(p) - \GC(p)  \right) 
= - \sqrt{n}\left(\LC_{n,\srs}(1-p) - \LC(1-p)  \right)
-\sqrt{n}\left(\LC_{n,\srs}(p) - \LC(p)  \right),  
\] 
and with the help of Lemma~\ref{LC-clt}, we obtain 
\[
\sqrt{n}\left(\GC_{n,\srs}(p) - \GC(p)  \right) 
= {1\over \sqrt{n}} \sum_{i=1}^n  \Big( Y_{i,\LC}(1-p)+ Y_{i,\LC}(p)\Big) + o_{\mathbb{P}}(1),   
\]
which concludes the proof of Lemma~\ref{GC-clt}. 
\end{proof}

\subsection{Proof of Theorem~\ref{fund-lemma-abc}}
\label{L-theorem}

We begin the proof by splitting the integral $L_{w,a,b}(F)$ into two parts according to whether $w(u)<0$ or $w(u)\ge 0$. This gives us the equation  
\begin{equation}\label{fund-lemma-eq2new}
\int_a^b F^{-1}(u) w(u) \dd u 
=  \int_a^b F^{-1}(u) \mathds{1}\{w(u)<0\} w(u) \dd u 
+\int_a^b F^{-1}(u) \mathds{1}\{w(u)\ge 0\} w(u) \dd u  . 
\end{equation}
We shall next analyze the two integrals on the right-hand side separately, showing that they are equal to the respective integrals on the right-hand side of equation~\eqref{fund-lemma-eq1abcabc}.

Pertaining to the first integral on the right-hand side of equation~\eqref{fund-lemma-eq2new}, since $w$ is non-decreasing and right-continuous, we have  $w(u)< 0$ 
if and only if  $u< w^{-1}(0)$, and so 
\[
\int_a^b F^{-1}(u) \mathds{1}\{w(u)<0\} w(u) \dd u 
=\int_a^b F^{-1}(u) \mathds{1}\{u< w^{-1}(0)\} w(u) \dd u .   
\]
Furthermore, since $w(u)< 0$, we have  
\begin{align*} 
w(u)
&= - \int_{w(u)}^{0} \dd x 
\\
&= - \int_{-\infty}^{0}  \mathds{1}\{u<w^{-1}(x)\} \dd x ,   
\end{align*} 
where the last equation holds because we have $w(u)<x$ if and only if $u<w^{-1}(x)$, due to our assumption that  $w$ is non-decreasing and right-continuous. 
Consequently, 
\begin{align}
\int_a^b F^{-1}(u) \mathds{1}\{w(u)<0\} w(u) \dd u 
&=   -\int_{-\infty}^{0}\Bigg( \int_a^b F^{-1}(u) \mathds{1}\{u<w^{-1}(0) \wedge w^{-1}(x)\} \dd u \Bigg) \dd x  
\notag 
\\
&=  -\int_{-\infty}^{0}\Bigg(  \int_a^b F^{-1}(u) \mathds{1}\{u<w^{-1}(x)\} \dd u \Bigg) \dd x  ,  
\label{first-2}
\end{align}
where right-most equation holds because $w^{-1}$ is non-decreasing. 
The double integral on the right-hand side of equation~\eqref{first-2} is zero when $w^{-1}(x)\le a$. Hence, continuing with equation~\eqref{first-2} for only those $x$ that satisfy $a< w^{-1}(x)$,  we have  
\begin{align}
\int_a^b F^{-1}(u) \mathds{1}\{w(u)<0\} w(u) \dd u 
&= -\int_{-\infty}^{0}\mathds{1}\{a< w^{-1}(x)\}
\Bigg( \int_a^b F^{-1}(u) \mathds{1}\{u< w^{-1}(x)\} \dd u \Bigg) \dd x
\notag 
\\ 
&= -\int_{-\infty}^0\mathds{1}\{a< w^{-1}(x)\} \Bigg( \int_{a}^{b\wedge w^{-1}(x)}F^{-1}(u) \dd u \Bigg) \dd x .   
\label{abc-1}
\end{align}
Since $w$ is non-decreasing and right-continuous, we have that $a<w^{-1}(x)$ is equivalent to $w(a)<x$, and so equation~\eqref{abc-1} gives the first integral on the right-hand side of equation~\eqref{fund-lemma-eq1abcabc}.

Analogously we tackle the second integral on the right-hand side of equation~\eqref{fund-lemma-eq2new}. First, since  $w(u)\ge 0$ if and only if  $u\ge w^{-1}(0)$, we have 
\[
\int_a^b F^{-1}(u) \mathds{1}\{w(u)\ge 0\} w(u) \dd u 
=\int_a^b F^{-1}(u) \mathds{1}\{u\ge w^{-1}(0)\} w(u) \dd u . 
\]
Furthermore, since $w(u)\ge 0$, we have   
\begin{align*} 
w(u)
&= \int_{0}^{w(u)} \dd x 
\\
&= \int_{0}^{\infty} \mathds{1}\{u \ge w^{-1}(x)\} \dd x , 
\end{align*} 
where the last equation holds because $w(u)\ge x$ if and only if $u\ge w^{-1}(x) $. 
Consequently, 
\begin{align}
\int_a^b F^{-1}(u) \mathds{1}\{w(u)\ge 0\} w(u) \dd u 
&=  \int_{0}^{\infty}\Bigg( \int_a^b F^{-1}(u) \mathds{1}\{u\ge w^{-1}(0) \vee w^{-1}(x)\} \dd u \Bigg) \dd x 
\notag 
\\
&=   \int_{0}^{\infty}\Bigg( \int_a^b F^{-1}(u) \mathds{1}\{u\ge w^{-1}(x)\} \dd u \Bigg) \dd x , 
\label{first-3}
\end{align}
where right-most equation holds because $w^{-1}$ is non-decreasing. 
The double integral on the right-hand side of equation~\eqref{first-3} is zero when $w^{-1}(x)>b$. Hence, continuing with equation~\eqref{first-3} for only those $x$ that satisfy $w^{-1}(x) \le b$, we have 
\begin{align}
\int_a^b F^{-1}(u) \mathds{1}\{w(u)\ge 0\} w(u) \dd u 
&=\int_{0}^{\infty}\mathds{1}\{w^{-1}(x) \le b \} 
\Bigg( \int_a^b F^{-1}(u) \mathds{1}\{u\ge w^{-1}(x)\} \dd u \Bigg) \dd x 
\notag 
\\ 
&=\int_{0}^{\infty}\mathds{1}\{w^{-1}(x) \le b \} 
\Bigg( \int_{a \vee w^{-1}(x)}^b F^{-1}(u)\dd u \Bigg) \dd x . 
\label{abc-2}
\end{align}
Since $w$ is non-decreasing and right-continuous, we have that $w^{-1}(x)\le b$ is equivalent to $x\le w(b)$, and so equation~\eqref{abc-2} gives the second integral on the right-hand side of equation~\eqref{fund-lemma-eq1abcabc}. This concludes the proof of Theorem~\ref{fund-lemma-abc}.

\subsection{Proof of Theorem~\ref{min-k}.}
\label{min-w}

Let condition~\ref{cond-w-diff} be satisfied with some $K\ge 3$. Equation~\eqref{ccc-1} is equivalent to 
\begin{align*}
w(u)
&=\mathds{1}\{0\le t < a_1\}\Big( w_{11}(u) - w_{12}(u) \Big)
\notag 
\\
&\qquad + \sum_{k=2}^{K-1} \mathds{1}\{a_{k-1}\le t < a_k\}\Bigg( \bigg\{w_{k1}(u)+\sum_{i=2}^{k}c_i \bigg\} 
-  \bigg\{w_{k2}(u)+\sum_{i=2}^{k}c_i \bigg\}  \Bigg)   
\notag 
\\
&\qquad +\mathds{1}\{a_{K-1}\le t < 1\}\Big( w_{K1}(u) - w_{K2}(u) \Big) , 
\end{align*}
where the constants $c_2,\dots , c_{K-1}$ can be any real numbers, but we choose them so that the inequalities 
\[
w_{k1}(a_{k-1})+\sum_{i=2}^{k}c_i \ge w_{(k-1)1}(a_{k-1}-)+\sum_{i=2}^{k-1}c_i 
\qquad \Big ( \Longleftrightarrow w_{k1}(a_{k-1})+c_k \ge w_{(k-1)1}(a_{k-1}-)  \Big) 
\]
and  
\[
w_{k2}(a_{k-1})+\sum_{i=2}^{k}c_i \ge w_{(k-1)2}(a_{k-1}-)+\sum_{i=2}^{k-1}c_i  
\qquad \Big ( \Longleftrightarrow w_{k2}(a_{k-1})+c_k \ge w_{(k-1)2}(a_{k-1}-)  \Big)  
\]
would hold for every $k=2,\dots , K-1$. The functions 
\[
w_{11}^*(u)=\mathds{1}\{0\le u < a_1\}w_{11}(u) 
+ \sum_{k=2}^{K-1} \mathds{1}\{a_{k-1}\le u < a_k\}\bigg\{w_{k1}(u)+\sum_{i=2}^{k}c_i \bigg\}  
\]
and 
\[
w_{12}^*(u)=\mathds{1}\{0\le u < a_1\}w_{12}(u) 
+ \sum_{k=2}^{K-1} \mathds{1}\{a_{k-1}\le u < a_k\}\bigg\{w_{k2}(u)+\sum_{i=2}^{k}c_i \bigg\}  
\]
are non-decreasing and right-continuous on the interval $[0,a_{K-1})$.  
Furthermore, the functions 
\[
w_{21}^*(u)=w_{K1}(u) 
\quad \textrm{and} \quad 
w_{22}^*(u)=w_{K2}(u) 
\]
are non-decreasing and right-continuous on the interval $[a_{K-1},1)$. We obviously have the equation 
\[
w(u)
=\mathds{1}\{0 \le u < a_{K-1}\}\Big( w_{11}^*(u) - w_{12}^*(u) \Big) 
+ \mathds{1}\{a_{K-1}\le u < 1\}\Big( w_{21}^*(u) - w_{22}^*(u) \Big) , 
\]
which establishes representation~\eqref{ccc-abc} and completes the proof of Theorem~\ref{min-k}.

\subsection{Proofs of cursory notes}
\label{proof-a1}

In this appendix we present proofs of some of the cursory notes that we have made while discussing main results.

\subsubsection{Proof of equation~\eqref{eq-0}} 
\label{proof-eq-0}
 
Let $U$ denote a uniform on $[0,1]$ random variable. Since $X$ and $Y$ have the same distributions as $F^{-1}(U)$ and $G^{-1}(U)$, respectively, we therefore have 
\begin{align}
\int_0^1 \big( G^{-1}(u)-F^{-1}(u) \big) \dd u 
&= \mathbb{E}\big( G^{-1}(U) \big) - \mathbb{E}\big( F^{-1}(U) \big)
\notag
\\
&=\mathbb{E}(Y)-\mathbb{E}(X). 
\label{eq-00}
\end{align}
Furthermore, using Fubini's theorem, we have 
\begin{align}
\int^{\infty }_{-\infty}  \big( F(x)- G(x)\big) \dd x 
&= \mathbb{E}\left(\int^{\infty }_{-\infty}  \big( \mathds{1}\{X\le x\}- \mathds{1}\{Y\le x\}\big) \dd x \right )
\notag
\\
&= \mathbb{E}\left(\int^{\infty }_{-\infty}  \big( \mathds{1}\{Y> x\}- \mathds{1}\{X> x\}\big) \dd x \right )
\notag
\\
&=\mathbb{E}(Y-X),  
\label{eq-01}
\end{align}
where the right-most equation holds because the integral inside the expectation is equal to $Y-X$. Since the right-hand sides of equations~\eqref{eq-00} and~\eqref{eq-01} are equal, we have equation~\eqref{eq-0}.

\subsubsection{On the validity of statements~\eqref{h-fail-1} and~\eqref{z-conv-32}-\eqref{z-conv-12}}
\label{app-z-conv}

Using the definition of the empirical quantile function $H_{n,\textrm{SRS}}^{-1}$, statement~\eqref{h-fail-1} is equivalent to
\[
Z_{\lceil n/2 \rceil : n} \stackrel{\mathbb{P}}{\nrightarrow} \frac{1}{2}.
\]
To prove it, we want to find $\epsilon >0$ such that 
\begin{equation}\label{abc-10}
\lim_{n \to \infty}\mathbb{P} \left(\left| Z_{\lceil n/2 \rceil : n} - \frac{1}{2} \right| > \epsilon\right) > 0. 
\end{equation}
We set $\epsilon = 1/2$ and note that statement~\eqref{abc-10} follows if we establish at least one of statements~\eqref{z-conv-32} and~\eqref{z-conv-12}, but for completeness of the argument, we shall next establish both of them.  Using a well-known formula \citep[e.g.,][eq.~(2.1.3)]{DN2003} for the cdf of order statistics, we have
\begin{align} 
\mathbb{P} \left(Z_{\lceil n/2 \rceil : n} > \frac{3}{2}\right)
&= 1 - \sum_{j = \lceil n/2 \rceil}^{n} \mathbb{P} \left(Z_{j : n} \leq \frac{3}{2}\right) 
\notag 
\\
&= 1 - \sum_{j = \lceil n/2 \rceil}^{n} \binom{n}{j} \left(F_{Z}\left(\frac{3}{2}\right)\right)^{j} \left(1 - F_{Z}\left(\frac{3}{2}\right)\right)^{n-j} 
\notag 
\\
&= 1 - \sum_{j = \lceil n/2 \rceil}^{n} \binom{n}{j} \left(\frac{1}{2}\right)^{j} \left(\frac{1}{2}\right)^{n-j} 
\notag 
\\
&= 1 - \left(\frac{1}{2}\right)^{n} \sum_{j = \lceil n/2 \rceil}^{n} \binom{n}{j}. 
\label{equlaity-1}
\end{align}
Similarly, 
\begin{align} 
\mathbb{P} \left(Z_{\lceil n/2 \rceil : n} < \frac{1}{2}\right)
&= \sum_{j = \lceil n/2 \rceil}^{n} \mathbb{P} \left(Z_{j : n} < \frac{1}{2}\right) 
\notag 
\\
&= \sum_{j = \lceil n/2 \rceil}^{n} \binom{n}{j} \left(F_{Z}\left(\frac{1}{2}\right)\right)^{j} \left(1 - F_{Z}\left(\frac{1}{2}\right)\right)^{n-j} 
\notag 
\\
&= \sum_{j = \lceil n/2 \rceil}^{n} \binom{n}{j} \left(\frac{1}{2}\right)^{j} \left(\frac{1}{2}\right)^{n-j} 
\notag 
\\
&= \left(\frac{1}{2}\right)^{n} \sum_{j = \lceil n/2 \rceil}^{n} \binom{n}{j}.
\label{equlaity-2}
\end{align}
We next rearrange the sums of binomials on the right-hand sides of equations~\eqref{equlaity-1} and~\eqref{equlaity-2}. When $n$ is even, $\lceil n/2 \rceil = n/2$ and $2^{n}$ can be expressed in $n+1$ (which is odd) terms
\begin{align*} 
2^{n} &= \binom{n}{0} + \binom{n}{1} + \cdots + \binom{n}{\frac{n}{2} - 1} + \binom{n}{\frac{n}{2}} + \binom{n}{\frac{n}{2} + 1} + \cdots + \binom{n}{n-1} + \binom{n}{n} \\
&= \binom{n}{\frac{n}{2}} + 2  \left( \binom{n}{\frac{n}{2} + 1} + \cdots + \binom{n}{n-1} + \binom{n}{n} \right). 
\end{align*}
Hence, we have 
\begin{align*} 
\sum_{j = \lceil n/2 \rceil}^{n} \binom{n}{j} &= \binom{n}{\lceil \frac{n}{2} \rceil} + \binom{n}{\lceil \frac{n}{2} \rceil + 1} + \cdots + \binom{n}{n-1} + \binom{n}{n} \\
&= \binom{n}{\frac{n}{2}} + \binom{n}{\frac{n}{2} + 1} + \cdots + \binom{n}{n-1} + \binom{n}{n} \\
&= {\frac{1}{2}}  \left( 2^{n} + \binom{n}{\frac{n}{2}} \right) \\
&= 2^{n-1} + {\frac{1}{2}}  \binom{n}{\frac{n}{2}}.
\end{align*}
When $n$ is odd, $2^{n}$ can be expressed in $n+1$ (which is even) terms as follows: 
\begin{align*} 
2^{n} &= \binom{n}{0} + \binom{n}{1} + \cdots + \binom{n}{\lceil \frac{n}{2} \rceil - 1} + \binom{n}{\lceil \frac{n}{2} \rceil} + \cdots + \binom{n}{n-1} + \binom{n}{n} \\
&= 2  \left( \binom{n}{\lceil \frac{n}{2} \rceil} + \cdots + \binom{n}{n-1} + \binom{n}{n} \right). 
\end{align*}
Hence, we have
\begin{align*}
\sum_{j = \lceil n/2 \rceil}^{n} \binom{n}{j} &= \binom{n}{\lceil \frac{n}{2} \rceil} + \binom{n}{\lceil \frac{n}{2} \rceil + 1} + \cdots + \binom{n}{n-1} + \binom{n}{n} \\
&= {\frac{1}{2}}  2^{n} \\
&= 2^{n-1}.
\end{align*}
Going back to the original expressions~\eqref{equlaity-1} and~\eqref{equlaity-2} of probabilities, we have
\begin{align*}
\mathbb{P} \left(Z_{\lceil n/2 \rceil : n} > \frac{3}{2}\right) 
&= 1 - \left(\frac{1}{2}\right)^{n} \sum_{j = \lceil n/2 \rceil}^{n} \binom{n}{j} \\
&= \begin{cases}
    \frac{1}{2} - \frac{1}{2^{n+1}}  \binom{n}{n/2} & \textrm{when} \;n\; \textrm{is even}; \\
    \frac{1}{2} & \textrm{when} \;n\; \textrm{is odd};
   \end{cases}
\end{align*}
and
\begin{align*}
\mathbb{P} \left(Z_{\lceil n/2 \rceil : n} < \frac{1}{2}\right)
&= \left(\frac{1}{2}\right)^{n} \sum_{j = \lceil n/2 \rceil}^{n} \binom{n}{j} \\
&= \begin{cases}
    \frac{1}{2} + \frac{1}{2^{n+1}}  \binom{n}{n/2} & \textrm{when} \;n\; \textrm{is even}; \\
    \frac{1}{2} & \textrm{when} \;n\; \textrm{is odd}.
   \end{cases}
\end{align*}
As $n \to \infty$, the term $\frac{1}{2^{n+1}}  \binom{n}{n/2}$ vanishes. This establishes statements~\eqref{z-conv-32} and~\eqref{z-conv-12}, and concludes the proof of statement~\eqref{h-fail-1}.  

\subsubsection{Derivation of expression~\eqref{eq-variance-spec}} 
\label{proof-eq-variance-spec}

Starting with equation~\eqref{eq-variance-spec0}, we need to calculate the first two moments of the random variable $\big( H^{-1}(U)-(1-a) \big)^{+}$. From Figure~\ref{figure-gap-1} we see that the random variable $H^{-1}(U)-(1-a) $ is (strictly) positive if and only if $U>1/2$, and under this condition, $H^{-1}(U)$ is equal to $2(1-a)U+2a$. Hence, 
\begin{equation}
\big( H^{-1}(U)-(1-a) \big)^{+} = \big( 2(1-a)U+2a \big) \mathds{1}\{ U>1/2\}. 
\label{eq-variance-spec1}
\end{equation} 
By calculating the first two moments of the random variable on the right-hand side of equation~\eqref{eq-variance-spec1}, we arrive at expression~\eqref{eq-variance-spec}.

\subsubsection{Proof of equation~\eqref{eq-kw}} 
\label{proof-eq-kw}

We start with the equations 
\begin{align*} 
K_w(F(x))
&=\int_{0}^{F(x)} w(u) \dd u 
\\
&=\int_{0}^{1} \mathds{1}\{u\le F(x)\} w(u) \dd u  
\\
&=\int_{0}^{1} \mathds{1}\{F^{-1}(u)\le x\} w(u) \dd u .   
\end{align*} 
Analogous equations hold for $K_w(G(x))$. Hence, 
\begin{align*} 
\int_{\mathbb{R}} \Big( K_w\big(F(x)\big) - K_w\big( G(x) \big) \Big)  \dd x 
&=\int_{\mathbb{R}} \bigg( \int_{0}^{1} \Big( \mathds{1}\{F^{-1}(u)\le x\} 
-\mathds{1}\{G^{-1}(u)\le x\} \Big) w(u) \dd u  \bigg)  \dd x 
\notag 
\\
&=\int_{0}^{1} \bigg( \int_{\mathbb{R}} \Big( \mathds{1}\{F^{-1}(u)\le x\} 
-\mathds{1}\{G^{-1}(u)\le x\} \Big)  \dd x  \bigg) w(u) \dd u  
\notag 
\\
&= \int_0^1 \Big( G^{-1}(u) - F^{-1}(u)\Big)  w(u) \dd u , 
\end{align*} 
where the right-most equation holds because the inner integral is equal to $G^{-1}(u) - F^{-1}(u)$.
This completes the proof of equation~\eqref{eq-kw}.

\subsubsection{Proof of equation~\eqref{eq-L-integral}}
\label{proof-L-integral}

Since $w_1$ is non-decreasing and right-continuous, and such that $w_1(0)\ge 0$, we have 
\begin{align*}
\int_0^1 F^{-1}(u) w_1(u) \dd u 
& =\int_0^1 F^{-1}(u) \bigg(\int_{0}^{\infty} \mathds{1}\{x\le w_1(u)\} \dd x\bigg) \dd u 
\\
&=  \int_0^1 F^{-1}(u) \bigg(\int_{0}^{\infty} \mathds{1}\{w_1^{-1}(x)\le u\} \dd x\bigg) \dd u 
\\
&=  \int_{0}^{\infty} \bigg(\int_0^1 F^{-1}(u) \mathds{1}\{w_1^{-1}(x)\le u\}\dd u \bigg) \dd x 
\\
&=  \int_{0}^{\infty} \bigg(\int_{w_1^{-1}(x)}^1 F^{-1}(u) \dd u \bigg) \dd x . 
\end{align*}
Analogous equations hold for $w_2$. 
This establishes equation~\eqref{eq-L-integral}. 
\end{document}